\documentclass[10pt, aop,preprint, leqno]{imsart}
\pdfoutput=1

\setattribute{journal}{name}{}

\usepackage[T1]{fontenc}	

\usepackage{graphicx}
\usepackage{amsmath, amsthm, amssymb,centernot}

\usepackage{upgreek}
\usepackage{mathrsfs}

\usepackage{pdfsync}
\usepackage[a4paper,top=3cm,bottom=3cm,left=2.5cm,right=2.5cm, bindingoffset=0mm]{geometry}

\usepackage[usenames,dvipsnames]{xcolor}
\usepackage[inline,shortlabels]{enumitem}
\usepackage[sort,numbers]{natbib}								
\usepackage[hyphens]{url}									
\usepackage{verbatim}								

\usepackage{dsfont}									

\definecolor{DarkGray}{gray}{0.35}
\usepackage{mathtools}
\usepackage{caption}
\usepackage{subcaption}

\arxiv{}

\startlocaldefs

\usepackage{xparse}

\ExplSyntaxOn
\NewDocumentCommand{\makeabbrev}{mmm}
 {
  \yoruk_makeabbrev:nnn { #1 } { #2 } { #3 }
 }

\cs_new_protected:Npn \yoruk_makeabbrev:nnn #1 #2 #3
 {
  \clist_map_inline:nn { #3 }
   {
    \cs_new_protected:cpn { #2 } { #1 { ##1 } }
   }
 }
 \ExplSyntaxOff

\makeabbrev{\textbf}{tbf#1}{a,b,c,d,e,f,g,h,i,j,k,l,m,n,o,p,q,r,s,t,u,v,w,x,y,z,A,B,C,D,E,F,G,H,I,J,K,L,M,N,O,P,Q,R,S,T,U,V,W,X,Y,Z}

\makeabbrev{\textbf}{bf#1}{a,b,c,d,e,f,g,h,i,j,k,l,m,n,o,p,q,r,s,t,u,v,w,x,y,z,A,B,C,D,E,F,G,H,I,J,K,L,M,N,O,P,Q,R,S,T,U,V,W,X,Y,Z}

\makeabbrev{\textsf}{tsf#1}{a,b,c,d,e,f,g,h,i,j,k,l,m,n,o,p,q,r,s,t,u,v,w,x,y,z,A,B,C,D,E,F,G,H,I,J,K,L,M,N,O,P,Q,R,S,T,U,V,W,X,Y,Z}

\makeabbrev{\mathsf}{mss#1}{a,b,c,d,e,f,g,h,i,j,k,l,m,n,o,p,q,r,s,t,u,v,w,x,y,z,A,B,C,D,E,F,G,H,I,J,K,L,M,N,O,P,Q,R,S,T,U,V,W,X,Y,Z}

\makeabbrev{\mathfrak}{mf#1}{a,b,c,d,e,f,g,h,i,j,k,l,m,n,o,p,q,r,s,t,u,v,w,x,y,z,A,B,C,D,E,F,G,H,I,J,K,L,M,N,O,P,Q,R,S,T,U,V,W,X,Y,Z}

\makeabbrev{\mathrm}{mrm#1}{a,b,c,d,e,f,g,h,i,j,k,l,m,n,o,p,q,r,s,t,u,v,w,x,y,z,A,B,C,D,E,F,G,H,I,J,K,L,M,N,O,P,Q,R,S,T,U,V,W,X,Y,Z}

\makeabbrev{\mathbf}{mbf#1}{a,b,c,d,e,f,g,h,i,j,k,l,m,n,o,p,q,r,s,t,u,v,w,x,y,z,A,B,C,D,E,F,G,H,I,J,K,L,M,N,O,P,Q,R,S,T,U,V,W,X,Y,Z}

\makeabbrev{\mathcal}{mc#1}{A,B,C,D,E,F,G,H,I,J,K,L,M,N,O,P,Q,R,S,T,U,V,W,X,Y,Z}

\makeabbrev{\mathbb}{mbb#1}{A,B,C,D,E,F,G,H,I,J,K,L,M,N,O,P,Q,R,S,T,U,V,W,X,Y,Z}

\makeabbrev{\mathscr}{ms#1}{A,B,C,D,E,F,G,H,I,J,K,L,M,N,O,P,Q,R,S,T,U,V,W,X,Y,Z}

\makeabbrev{\mathrm}{#1}{
Id,id,ran,rk,diag,stab,ann,conv,pr,ev,tr,End,Hom,sgn,im,op,can,fin,ext,red,tot,
rot,usc,lsc,Lip,lip,bSymLip,osc,AC,loc,coz,z,erf,
supp,Opt,Adm,Cpl,Geo,GeoOpt,GeoAdm,GeoCpl,reg,
bd,co,Ric,Exp,dExp,dist,seg,Seg,cut,fcut,Cut,SDiff,Iso,Isom,diam,cl,Homeo,Diff,Der,vol,dvol,inj,relint, Graph,sub,Tube,codim,
var,law,Var,Poi,Gam,pa,so,iso,fs,inv,pqi,mix,Cov,
TestF,
}

\makeabbrev{\mathsf}{#1}{CD,BE,MCP,Ent,wMTW,MTW,Ch,RCD,EVI,Rad,dRad,SL,cSL,dSL,ScL,Irr,SC,wFe,VA}

\makeabbrev{\mathsc}{msc#1}{g}

\newcommand{\spec}[1]{\mathrm{spec}(#1)}

\newcommand{\FGF}[2][]{\mathsf{FGF}^{#1}_{#2}}
\newcommand{\gFGF}[2][]{\mathring{\mathsf{FGF}}^{#1}_{#2}}

\newcommand{\eps}{\varepsilon}

\newcommand{\mathsc}[1]{\text{\textsc{#1}}}
\newcommand{\emparg}{{\,\cdot\,}}

\newcommand{\dom}{\mcF}

\DeclareMathOperator{\eqdef}{\coloneqq}

\let\epsilon\varepsilon

\newcommand{\longrar}{\longrightarrow}
\newcommand{\rar}{\rightarrow}

\newcommand{\diff}{\mathop{}\!\mathrm{d}}						

\newcommand{\tabs}[1]{\big\lvert#1\big\rvert}	
\newcommand{\abs}[1]{\left\lvert#1\right\rvert}						
\newcommand{\tnorm}[1]{\big\lVert#1\big\rVert}					
\newcommand{\norm}[1]{\left\lVert#1\right\rVert}					
\newcommand{\set}[1]{\left\{#1\right\}}							
\newcommand{\tset}[1]{\big\{#1\big\}}							
\newcommand{\ttset}[1]{\{#1\}}									
\newcommand{\ceiling}[1]{\left\lceil#1\right\rceil}					
\newcommand{\paren}[1]{\left(#1\right)}							
\newcommand{\tparen}[1]{\big({#1}\big)}
\newcommand{\quadre}[1]{\left[#1\right]}							

\newcommand{\tquadre}[1]{\big[#1\big]}							
\newcommand{\scalar}[2]{\left\langle #1 \,\middle |\, #2\right\rangle}		

\DeclareSymbolFont{symbolsC}{U}{pxsyc}{m}{n}
\SetSymbolFont{symbolsC}{bold}{U}{pxsyc}{bx}{n}
\DeclareFontSubstitution{U}{pxsyc}{m}{n}
\DeclareMathSymbol{\medcirc}{\mathbin}{symbolsC}{7}
\DeclareSymbolFont{symbolsZ}{OMS}{pxsy}{m}{n}
\SetSymbolFont{symbolsZ}{bold}{OMS}{pxsy}{bx}{n}
\DeclareFontSubstitution{OMS}{pxsy}{m}{n}

\newcommand{\seq}[1]{\paren{#1}}								
\newcommand{\tseq}[1]{{\big(#1\big)}}

\DeclareMathOperator{\car}{\mathds 1}

\newcommand{\N}{{\mathbb N}}
\newcommand{\R}{{\mathbb R}}
\DeclareMathOperator{\Q}{{\mathbb Q}}
\newcommand{\C}{{\mathbb C}} 	
\DeclareMathOperator{\Z}{{\mathbb Z}}

\newcommand{\EEE}{\mathbf{E}}
\newcommand\PPP{\mbfP}
\newcommand\bigdot{\bullet}
\newcommand\M{{\sf M}}
\newcommand\g{{\sf g}}

\usepackage{booktabs}

\allowdisplaybreaks

\newcommand{\iref}[1]{\ref{#1}}

\newcommand{\comma}{\,\,\mathrm{,}\;\,}
\newcommand{\semicolon}{\,\,\mathrm{;}\;\,}
\newcommand{\fstop}{\,\,\mathrm{.}}

\renewcommand{\iint}{\int\!\!\!\!\int}
\newcommand{\imu}{\mathrm{i}}
\newcommand{\av}[1]{\left\langle#1\right\rangle}
\newcommand{\tav}[1]{\big\langle#1\big\rangle}

\let\temp\phi
\let\phi\varphi
\let\varphi\temp

\numberwithin{equation}{section}
\theoremstyle{plain}
\newtheorem{thm}{Theorem}[section]
\newtheorem*{thm*}{Theorem}
\newtheorem*{mthm*}{Main Theorem}

\newtheorem{prop}[thm]{Proposition}
\newtheorem{lem}[thm]{Lemma}
\newtheorem{cor}[thm]{Corollary}

\theoremstyle{definition}
\newtheorem{defs}[thm]{Definition}
\newtheorem{notat}[thm]{Notation}
\newtheorem*{defs*}{Definition}

\theoremstyle{remark}
\newtheorem{rem}[thm]{Remark}
\newtheorem{ese}[thm]{Example}
\newtheorem{ass}[thm]{Assumption}
\newtheorem*{ass*}{Assumption}

\newcommand{\purple}[1]{{\color{purple}{#1}}}

\newcommand{\Test}{\msD}

\endlocaldefs

\begin{document}

\begin{frontmatter}
\title{A Discovery Tour in Random Riemannian Geometry
}
\runtitle{}
\thankstext{T1}{Data sharing not applicable to this article as no datasets were generated or analyzed during the current study.}

\begin{aug}
\author{\fnms{Lorenzo} \snm{Dello Schiavo}\thanksref{t1}\ead[label=e1]{lorenzo.delloschiavo@ist.ac.at}}

\author{\fnms{Eva} \snm{Kopfer}\thanksref{t2}\ead[label=e2]{eva.kopfer@iam.uni-bonn.de}}

\author{\fnms{Karl-Theodor} \snm{Sturm}\thanksref{t2}\ead[label=e3]{sturm@uni-bonn.de}}


\thankstext{t1}{This author gratefully acknowledges financial support by the Deutsche Forschungsgemeinschaft through CRC 1060 as well as through SPP 2265, and by the Austrian Science Fund (FWF) grant F65 at Institute of Science and Technology Austria.
He also acknowledges funding of his current position by the Austrian Science Fund (FWF) through the ESPRIT Programme (grant No.~208).}

\thankstext{t2}{This author gratefully acknowledges funding by the Deutsche Forschungsgemeinschaft through the Hausdorff Center for Mathematics and through CRC 1060 as well as through SPP 2265.}

\runauthor{L.~Dello~Schiavo, E.~Kopfer, K.-T.~Sturm}

\affiliation{\thanksref{t1}Institute of Science and Technology Austria}

\affiliation{\thanksref{t2}Institut f\"ur angewandte Mathematik\\Rheinische Friedrich-Wilhelms-Universit\"at Bonn}

\address{Institute of Science and Technology Austria\\
Am Campus 1\\
3400 Klosterneuburg\\
Austria\\
\printead{e1}\\
}

\address{Institut f\"ur Angewandte Mathematik\\
Rheinische Friedrich-Wilhelms-Universit\"at Bonn\\
Endenicher Allee 60\\
53115 Bonn\\
Germany\\
\printead{e2}\\
\printead{e3}
}

\end{aug}

\begin{abstract} {\it Abstract.}
We study random perturbations of a Riemannian manifold $(\mssM,\g)$ by means of so-called \emph{Fractional Gaussian Fields}, which are defined intrinsically by the given manifold.
The fields $h^\bullet: \omega\mapsto h^\omega$ will act on the manifold via the conformal transformation $\g\mapsto \g^\omega\eqdef e^{2h^\omega}\,\g$.
Our focus will be on the regular case with Hurst parameter $H>0$, the critical case~$H=0$ being the celebrated Liouville geometry in two dimensions.
We want to understand how basic geometric and functional-analytic quantities like: diameter, volume, heat kernel, Brownian motion, spectral bound, or spectral gap change under the influence of the noise. 
And if so, is it possible to quantify these dependencies in terms of key parameters of the noise?
Another goal is to define and analyze in detail the Fractional Gaussian Fields on a general Riemannian manifold, a fascinating object of independent interest.
\end{abstract}

\vspace{.5cm}
\today
\vspace{.5cm}

\begin{keyword}[class=MSC2020]
\kwd{Primary: 60G15}
\kwd{secondary: 58J65}
\kwd{31C25.}
\end{keyword}

\begin{keyword}
\kwd{Fractional Gaussian Fields}
\kwd{Gaussian Free Field}
\kwd{random geometry}
\kwd{Liouville Quantum Gravity}
\kwd{Liouville Brownian Motion}
\kwd{spectral gap estimates.}
\end{keyword}

\end{frontmatter}

\tableofcontents

\section{Introduction}
\subsection{Random Riemannian Geometry}
Given a Riemannian manifold  $(\M,\g)$ and a Gaussian random field $h^\bullet: \, \Omega\to {\mathcal C}(\M), \ \omega\mapsto h^\omega$, we study random perturbations  $(\M,\g^\omega)$ of the given manifold with conformally changed metric tensors $\g^\omega\eqdef e^{2h^\omega}\g$. For this \emph{Random Riemannian Geometry} 
\begin{equation*}
(\M,\g^\bullet)\quad\text{with}\quad \g^\bullet\eqdef e^{2h^\bullet}\g
\end{equation*}
we  want to understand how basic geometric and functional analytic quantities like diameter, volume, heat kernel, Brownian motion, or spectral gap change under the influence of the noise.
If  possible, we want to quantify these dependencies in terms of key parameters of the noise.

\begin{figure}[htb!]
	\includegraphics[scale=.7]{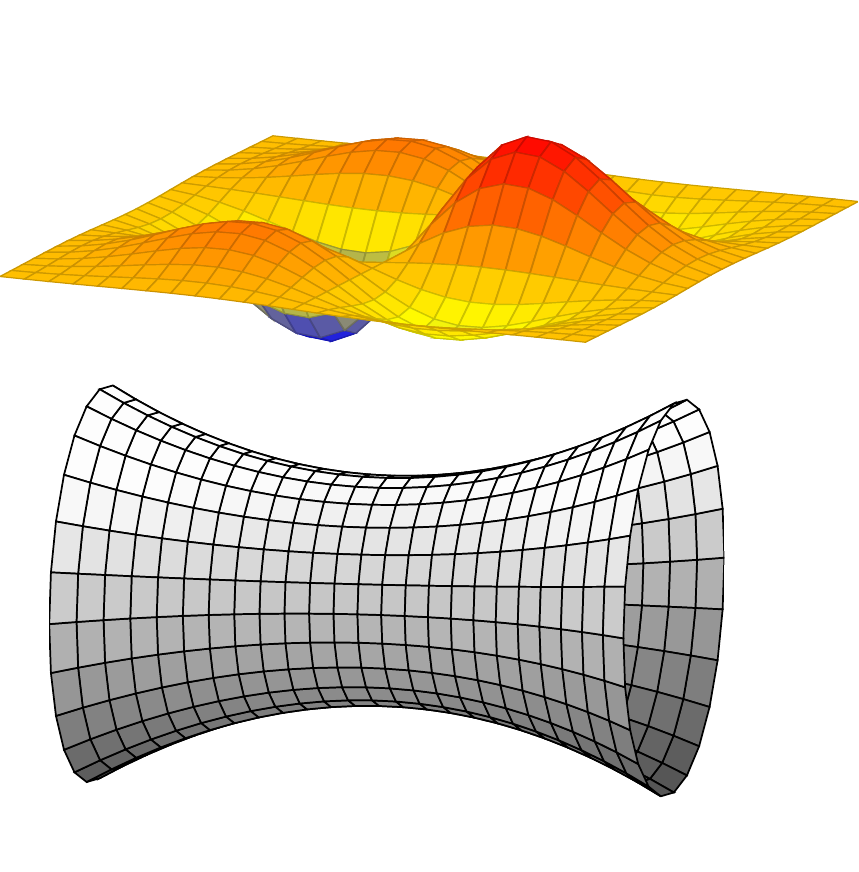}
	\caption{Gaussian random field over a toroid.}
	\label{fig: toroid}
	\end{figure}

Our main interest in the sequel will be in the case~$h^\bullet\not\in {\mathcal C}^2(\M)$ a.s., where standard Riemannian calculus is
not directly applicable and where no classical curvature concepts are at our disposal.
Our approach to geometry, spectral analysis, and stochastic calculus on the randomly perturbed Riemannian manifolds~$(\M,\g^\bullet)$ will be based on Dirichlet form techniques.

For convenience, we will assume throughout that the reference manifold $(\M,\g)$ has bounded geometry.

\begin{thm}
For every $\omega$,  a regular, strongly local  Dirichlet form 
is given by
\begin{align}
\mcE^\omega(\phi,\psi)= \frac{1}{2}\int_\M  \scalar{\nabla \phi}{\nabla \psi}_{\g}\,
 e^{(n-2)h^\omega}\diff\vol_\g\qquad\text{on}\quad L^2\big(\M,e^{nh^\omega}\vol_\g\big) \fstop
\end{align}
\end{thm}
The associated 
\emph{Laplace--Beltrami operator}~$\big(\Delta^\omega,\mcD(\Delta^\omega)\big)$ on $(\M,\g^\omega)$ is uniquely characterized by
$\mcD(\Delta^\omega)\subset \mcD(\mcE^\omega)$ and
$\mcE^\omega(\phi,\psi)=-\frac{1}{2}\int (\Delta^\omega\phi)\ \psi\, e^{nh^\omega}\dvol_\g$
for $\phi\in\mcD(\Delta^\omega),\, \psi\in \mcD(\mcE^\omega)$. 

The associated 
 \emph{Riemannian metric}  is given by
 \begin{align*} 
 \mssd^\omega(x,y)\eqdef\inf\set{ \int_0^1e^{h^\omega(\gamma_r)}\, |\dot\gamma_r|\diff r: \ \gamma\in\mathcal{AC}\big([0,1];\mssM\big)\comma \gamma_0=x\comma\gamma_1=y }\comma
 \end{align*}  
 where~$\abs{\dot\gamma_r}\eqdef \sqrt{\mssg(\dot\gamma_r,\dot\gamma_r)}$ denotes the speed of an absolutely continuous curve~$\gamma_r$.

\begin{prop}
The heat semigroup $\big(e^{t\Delta^\omega/2}\big)_{t>0}$  has an integral kernel $p^\omega_t(x,y)$ which is jointly locally H\"older continuous in $t,x,y$.
\end{prop}

The \emph{Brownian motion on $(\M,\g^\omega)$}, defined as the reversible, Markov diffusion process $\mbfB^{\omega}$ associated with the heat semigroup
$\big(e^{t\Delta^\omega/2}\big)_{t>0}$, allows for a more explicit construction if the conformal weight $h^\omega$ is differentiable.

\begin{prop} If $h^\omega\in {\mathcal C}^1(\M)$ then  $\mbfB^{\omega}$ is obtained from the Brownian motion  $\mbfB$ on $(\M,\g)$ by the combination of time change with weight $e^{2h^\omega}$ and Girsanov transformation with weight $(n-2)h^\omega$.
\end{prop}

We will compare the random volume, random length, and random distance in the random Riemannian manifold $(\M,\g^\bullet)$ with analogous quantities in  deterministic geometries obtained by suitable  conformal weights.

\begin{prop}
Put $\theta(x)\eqdef \EEE[h^\bigdot(x)^2]\ge0$ and 
$\overline\g^n\eqdef e^{n\,\theta}\g$, $\overline\g^1\eqdef e^{\theta}\g$. 
Then for every measurable~$A\subset \M$,
\begin{equation*}
\EEE [\vol_{\g^\bigdot}(A)]=\vol_{\overline\g^n}(A)\ge \vol_{\g}(A)\comma
\end{equation*}
and for every absolutely continuous curve $\gamma: [0,1]\to \M$,
\begin{align*}
\EEE [L_{\g^\bullet}(\gamma)]=L_{\overline\g^1}(\gamma)\ge L_{\g}(\gamma)\fstop
\end{align*}
\end{prop}

Of particular interest is  the rate of convergence to equilibrium for the random Brownian motion. 
\begin{thm}\label{intr-gap}
Assume that $\M$ is compact. Let~$\lambda^1$ be the spectral gap of $\Delta$, and for each $\omega$, denote by~$\lambda^\omega_1$ the spectral gap of $\Delta^\omega$. Then
\begin{align}\EEE\Big[\big|\log\lambda_1^\bullet-\log\lambda_1\big|\Big]\le \alpha\, \EEE\Big[\sup |h^\bullet|\Big] 
\end{align}
with $\alpha:=2(n-1)$ if $n\ge2$ and $\alpha:=2$ if $n=1$.
\end{thm}

Let us emphasize that classical estimates for the spectral gap, based on Ricci curvature estimates, require that the metric tensor is of class~${\mathcal C}^2$, whereas our Theorem \ref{intr-gap} --- combined with Theorem \ref{intr-max} below --- will apply whenever the random metric tensor is of class ${\mathcal C}^0$.
\subsection{Fractional Gaussian Field (FGF)}
In our approach to Random Riemannian Geometry, we will restrict ourselves to the case where the random field $h^\bullet$  
 is a {Fractional Gaussian Field}, defined intrinsically by the given manifold.
It is a fascinating object of independent interest.

Given a Riemannian manifold  $(\M,\g)$ of bounded geometry,
for~$m>0$ and $s\in\R$, 
we define the Sobolev spaces
\begin{align*}
H^s_m(\M)\eqdef \paren{m^2-\tfrac12\Delta}^{-s/2}\tparen{L^2(\M)}\comma \qquad \norm{u}_{H^s_m}\eqdef \norm{\paren{m^2-\tfrac12\Delta}^{s/2} u}_{L^2}\fstop
\end{align*}
The scalar product~$\scalar{u}{v}_{L^2}$ 
extends to a continuous bilinear pairing between~$H^s_m(\M)$ and~$H^{-s}_m(\M)$ as well as between $\Test(\M)$ and $\Test'(\M)$. 
It follows, that the functional~\mbox{$u\mapsto \exp\big(-\!\tfrac{1}{2}\!\|u\|^2_{H^{-s}_m}\big)$} is continuous on~$\Test(\M)$, and is therefore the Fourier transform of a unique centered Gaussian field 
with variance~$\norm{u}^2_{H^{-s}_m}$ by Bochner--Minlos Theorem applied to the nuclear space~$\Test'(\M)$.

\begin{thm} For every $s\in\R$ and $m>0$, there exists a unique centered Gaussian field $h^\bullet$ with
\begin{equation}\EEE\, e^{\imu\,\scalar{u}{h^\bigdot}}=e^{-\frac12\norm{u}^2_{H^{-s}_m}}, \qquad u\in \Test(\M)\comma
\end{equation}
 called  \emph{$m$-massive Fractional Gaussian Field on~$\M$ of regularity~$s$}, briefly {\sf FGF}$^\M_{s,m}$.
\end{thm}

For $s=0$ this is the white noise on $\M$.
Note that, if $h^\bigdot$ is distributed according to $\FGF[\M]{s,m}$ on some compact~$\M$, then $\big(m^2-\frac12\Delta\big)^{\frac{r-s}2}h^\bigdot$ is distributed according to $\FGF[\M]{r,m}$. 

\begin{thm}
For $s>0$, the Fractional Gaussian Field $\FGF[\M]{s,m}$ is uniquely characterized as the centered  Gaussian process $h^\bullet$ with covariance
\begin{equation}
\begin{aligned}
\Cov\tquadre{\scalar{h^\bullet}{\phi}, \scalar{h^\bullet}{\psi}}= \iint \phi(x) \, G_{s,m}(x,y)\, \psi(y) \diff\vol_\g^{\otimes 2}(x,y)
\end{aligned}
\comma \quad \phi,\psi\in\Test\subset H^{-s}_m\comma
\end{equation}
  where
$
G_{s,m}(x,y)
 \eqdef\frac1{\Gamma(s)}\int_0^\infty p_t(x,y)\, e^{-m^2t}\,t^{s-1}\diff t
$. 
For $s>n/2$, this characterization simplifies to 
\begin{align}
\mbfE\tquadre{h^\bullet(x)\, h^\bullet(y)}=G_{s,m}(x,y) \comma \qquad x,y\in \M\fstop
\end{align}
\end{thm}

Indeed, for $s>n/2$, the Fractional Gaussian Field $\FGF[\M]{s,m}$ is almost surely a continuous function.
More precisely,
\begin{prop}
Assume~$\M$ is compact and let~$h^\bullet\sim\FGF[\M]{s,m}$ with~$s> n/2+k, \ k\in\N_0$.
Then~$h^\omega \in \mcC^{k}(\M)$ for a.e.~$\omega$.
%
\end{prop}

A crucial role in our geometric estimates and functional inequalities for the Random Riemannian Geometry is played by estimates for the expected maximum of the random field.
\begin{thm}\label{intr-max} For every compact manifold $\M$ there exists a constant $C=C(\M)$ such that for~$h^\bullet\sim\FGF[\M]{s,m}$ with any $m>0$,
\begin{align*}
\mbfE\bigg[\sup_{x\in \M} h^\bullet(x)\bigg]\leq 
\begin{cases}
C\cdot (\lambda_1/2)^{-s/2}, \qquad& s\ge \frac n2 +1\comma\\
     C\cdot (s-n/2)^{-3/2}, \qquad& s\in\big(\frac n2 , \frac n2+1\big]\fstop
   \end{cases}
\end{align*}
\end{thm}

If $\M$ is compact, then an analogous construction also works in the case $m=0$ provided all function spaces $H^{-s}_m$ are replaced by the subspaces $\mathring H^{-s}_m$ obtained via the \emph{grounding} map $u\mapsto \mathring u\eqdef u-\frac1{\vol_\g(\M)}\int u\dvol_\mssg$.
The $\mathring{\sf FGF}^\M_{s,m}$
for $s=1, m=0$  is the celebrated \emph{Gaussian Free Field} (GFF) on~$\M$.

In the compact case, the Fractional Gaussian Field also admits a quite instructive series representation. 

\begin{thm}
Let 
$\seq{\phi_j}_{j\in\N_0}$ be a complete orthonormal basis in $L^2$ consisting of eigenfunctions of~$-\Delta$ with corresponding eigenvalues~$(\lambda_j)_{j\in\N_0}$, and 
 let a sequence~$\big(\xi_j^\bullet\big)_{j\in \N_0}$  of independent, $\mcN(0,1)$-distributed  random variables  be given. 
Then for~$s>n/2$ and $m\ge0$, the series 
\begin{align*}
h^\omega(x)\eqdef 
 \sum_{j\in \N} \frac{\phi_j(x) \, \xi_j^\omega}{(m^2+\lambda_j/2)^{s/2}}
\end{align*}
converges and provides a pointwise representation of~$h^\bullet\sim \gFGF[\M]{s,m}$.
\end{thm}

\begin{rem}
\begin{enumerate}[$(a)$, wide]
\item For  Euclidean spaces $\M=\R^n$, the $\gFGF[\M]{s,m}$ is well studied with particular focus on the  massless case $m=0$. Here some additional effort is required  to deal with the kernel of $\big(-\frac12\Delta\big)^{s/2}$ which is resolved by factoring out polynomials of degree~$\le s$.
The real white noise, the 1d Brownian motion, the L\'evy Brownian motion,  and the Gaussian Free Field on the Euclidean space  are all instances of random fields in the larger family of Fractional Gaussian Fields. 
The article~\cite{LodSheSunWat16} by Lodhia, Sheffield, Sun, and Watson provides an excellent survey.

Despite the fact that it seems to be regarded as common knowledge (in particular in the physics literature), even in the most prominent case $s=1$, the Riemannian context is addressed only occasionally, e.g.~\cite{Gel14,guillarmou2016polyakov,dang}.
In particular, Gelbaum~\cite{Gel14} studies the existence on complete Riemannian manifolds of the fractional Brownian motions~$\FGF[\M]{s,0}$, $s\in (n/2,n/2+1)$, and of the massive $\FGF[\M]{s,1}$, with the same values of~$s$.
Fractional Brownian motions are also constructed on Sierpi\'nski gaskets and related fractals in~\cite{BauLac20}. 

\item
The particular case of the FGF with $s=1$ is the  \emph{Gaussian Free Field}, discussed and analyzed in detail in the landmark article~\cite{She07} by Sheffield.
The GFF arises as scaling limit of various discrete models of random (hyper-)surfaces over $n$-dimensional simplicial lattices, e.g.\ Discrete Gaussian Free Fields (DGFF) or harmonic crystals~\cite{She07}.
The two-dimension\-al case is particularly relevant, for the GFF is then invariant under conformal transformations of~$D\subset\R^2\cong\C$, and constitutes therefore a useful tool in the study of conformally invariant random objects.
For instance, the zero contour lines of the GFF (despite being random distributions, not functions) are well-defined SLE curves~\cite{SchShe13}.

\item Again in the two-dimensional case, the GFF gives rise to an impressive random geometry, the \emph{Liouville Quantum Gravity}.
It is a hot topic of current research with plenty of fascinating, deep results
--- despite the fact that many classical geometric quantities become meaningless, see e.g.~\cite{GarRhoVar14,GarRhoVar16,AndKaj16,MSLQG1,LeGall19,guillarmou2016polyakov,DuplantierMillerSheffield,DKRV16}.

\indent 
In this paper, our focus will be on the Random Riemannian Geometry in the `regular' case of Hurst parameter $H\eqdef s-n/2>0$ in arbitrary dimension.
In general, this geometry is not conformally invariant, since neither the Laplace--Beltrami operator nor its powers are conformally covariant.
For \emph{compact} manifolds of arbitrary \emph{even} dimension~$n$, we shall address in~\cite{DSHerKopStu21} the conformally invariant case at the critical scale~$s=n/2$, a high-dimensional Liouville Quantum Gravity.
 \end{enumerate}
\end{rem}

\subsection{Higher Order Green Kernel}
The regularity of the Fractional Gaussian Field $h^\bullet$ and the quantitative geometric and functional analytic estimates for the Random Riemannian Geometry $(\M,\g^\bullet)$ will be determined by the Green kernel of order $s$,
\begin{align}
G_{s,m}(x,y) \eqdef\frac1{\Gamma(s)}\int_0^\infty p_t(x,y)\, e^{-m^2t}\,t^{s-1}\diff t
\end{align}
and, in the compact case, by  its grounded counterpart
\begin{align}
\mathring G_{s,m}(x,y) \eqdef\frac1{\Gamma(s)}\int_0^\infty \mathring p_t(x,y)\, e^{-m^2t}\,t^{s-1}\diff t, \qquad \mathring p_t(x,y)\eqdef p_t(x,y)-\frac1{\vol_\g(\M)}\fstop
\end{align}
The latter is also well-behaved in the massless case $m=0$ whereas the application of the former is restricted to the case of positive mass parameter $m$.
We analyze these Green kernels in detail and derive explicit formulas for model spaces, including Euclidean spaces, tori, 
hyperbolic spaces, 
and spheres. 
\begin{thm}
For points~$x,y$, let~$r\eqdef \mssd(x,y)$. Then,
\begin{enumerate}[$(a)$]
\item For the 1-dimensional torus ${\mathbb T}\eqdef \R/\Z$, 
\begin{align*}\mathring G^{{\mathbb T}}_{1,0}(r)&=\bigg(r-\frac12\bigg)^2-\frac1{12}\comma\qquad
 \mathring G^{{\mathbb T}}_{2,0}(r)=-\frac1{6}\bigg(r-\frac12\bigg)^4+ \frac1{12}\bigg(r-\frac12\bigg)^2-\frac7{1440} \fstop
\end{align*}
\item For the sphere in 2 and 3 dimensions, 
\begin{align*}
\mathring G_{1,0}^{\mathbb S^2}(r)=&-\tfrac{1}{2\pi}\tparen{1+2\log\sin\tfrac{r}{2}}\comma& 
\mathring G_{2,0}^{\mathbb S^2}(r)=&\frac1{\pi}\int_0^{\sin^2(r/2)} \frac{\log t}{1-t}\,\diff t+\frac1\pi
 \comma
\\
\mathring G_{1,0}^{\mbbS^3}(r)=&\tfrac{1}{2\pi^2}\paren{ -\tfrac{1}{2}+(\pi-r)\cdot\cot r }\comma &
\mathring G_{2,0}^{\mathbb S^3}(r)=&\frac{(\pi-r)^2}{4\pi^2}+\frac1{8\pi^2}-\frac1{12}
\fstop
\end{align*}

\item For the hyperbolic space in three dimensions and $m>0$,
\begin{align*}
G_{1,m}^{\mathbb H^3}(r)=\frac1{2\pi\,\sinh r}\, e^{-\sqrt{2m^2+1}\,r}\comma\qquad
G_{2,m}^{\mathbb H^3}(r)=\frac r{2\pi\,  \sqrt{2m^2+1}\,\sinh r}\, e^{-\sqrt{2m^2+1}\,r} \fstop
\end{align*}
\end{enumerate}
\end{thm}
Of particular interest is the asymptotics of the Green kernel close to the diagonal.
\begin{thm}
Let~$\mssM$ be a compact manifold, $m\ge0$, and $s>n/2$. Then for every $\alpha\in (0,1]$ with $\alpha<s-n/2$  there exists a constant~$C=C(\M)$ so that
\begin{align*}
\abs{\mathring G_{s,m}(x,x)+\mathring G_{s,m}(y,y)-2\, \mathring G_{s,m}(x,y)}^{1/2}\ \leq \
 C\cdot  \mssd(x,y)^{\alpha}  \fstop
\end{align*}
\end{thm}

\section*{Acknowledgements}
The authors would like to thank Matthias~Erbar and Ronan~Herry for valuable discussions on this project.
They are also grateful to Nathana\"el Berestycki, 
and Fabrice Baudoin for respectively pointing out the references~\cite{Ber15},
and~\cite{BauLac20,Gel14}, and to
Julien Fageot and Thomas Letendre for pointing out a mistake in a previous version of the proof of Proposition~\ref{p:Properties}.
The authors feel very much indebted to an anonymous reviewer for their careful reading and the many valuable suggestions that have significantly contributed to the improvement of the paper.

\section{The Riemannian Manifold}\label{s:Setting}
Throughout this paper,~$(\mssM,\g)$ will be a complete connected $n$-dimensional smooth Riemannian manifold without boundary,~$\Delta$ will denote its Laplace--Beltrami operator and~$p_t(x,y)$ the associated heat kernel. 
The latter is symmetric in $x,y$, and   
as a function of $t,x$ it solves the heat equation $\frac12\Delta u=\frac{\partial}{\partial t}u$.
For convenience, we always assume that~$(\M,\g)$ is stochastically complete, i.e.,
\begin{align*}
\int p_t(x,y)\,\dvol_\g(y)=1\comma \qquad x\in X, \ t>0\comma
\end{align*}
which is a well-known consequence of uniform lower bounds for the Ricci curvature, see e.g.~\cite[Thm.~5.2.6]{Dav89}.

\begin{notat}\label{notation:Asymp}
Throughout the paper, for functions~$a,b\colon \R\to (0,\infty)$ and~$r_0\in\R$ apparent from the context we write~$a\lesssim b$ if there exist~$\eps>0$ and~$c>0$ so that~$a(r)\leq c\cdot b(r)$ for all~$r$ so that~$\abs{r-r_0}<\eps$, and we set
\begin{align*}
a(r) \asymp b(r) \iff \lim_{r\rar r_0} \frac{a(r)}{b(r)}=1 \qquad \text{and} \qquad a(r) \approx b(r) \iff a\lesssim b \lesssim a \fstop
\end{align*}
\end{notat}

\subsection{Higher Order Green Operators} For~$m>0$, consider the positive self-adjoint operator
\begin{align*}
A_m\eqdef m^2-\tfrac12\Delta
\end{align*}
on $L^2=L^2(\vol_\g)$, and its powers  $A_m^{s}$ defined by 
means of the Spectral Theorem  for all $s\in\R$.
On appropriate domains, $A_m^{s}\circ A_m^{r}=A_m^{r+s}$ for all $r,s\in\R$.
For~$s> 0$, the operator $A_m^{-s}$, called the \emph{Green operator of order $s$ with mass parameter $m$},
admits the representation
\begin{align}\label{eq:BesselP}
A_m^{-s}\eqdef \frac{1}{\Gamma(s)}\int_0^\infty e^{-m^2\, t} t^{s-1} e^{t\Delta/2} \diff t\qquad \text{on} \quad L^2(\vol_\g)\fstop
\end{align}


\begin{lem}\label{l:Representation}
\begin{enumerate}[$(i)$]
\item\label{i:l:Representation:1} For $s>0$, the Green operator of order $s$
 is an integral operator
\begin{align*}\tparen{A_m^{-s}\, f}(x)=\int G_{s,m}(x,y)\, f(y)\diff\vol_\g(y)\end{align*}
with  density given by the \emph{Green kernel of order $s$ with mass parameter $m$}, 
\begin{equation}\label{eq:Bessel}
G_{s,m}(x,y)\eqdef \frac{1}{\Gamma(s)}\int_0^\infty e^{-m^2 t} \ t^{s-1} \, p_t(x,y)\diff t\comma
\end{equation}
where~$p_t(x,y)$ is the heat kernel (i.e.~the density for the operator~$e^{t\Delta/2}$).

\item\label{i:l:Representation:2} For each $m>0$, the family~$(G_{s,m})_{s>0}$ is a convolution semigroup of kernels, viz.~$G_{r+s,m}=G_{r,m}*G_{s,m}$ for~$r,s>0$. In particular,~$G_{k,m}=(G_{1,m})^{*k}$ for integer~$k\geq 1$. 

\item\label{i:l:Representation:3} Moreover,~$\int G_{s,m}(x,\emparg) \diff\vol_\g=m^{-2s}$ for all~$x\in \M$, $s>0$.
%
\end{enumerate}
\end{lem}

\begin{proof} \iref{i:l:Representation:1}
In light of~\eqref{eq:BesselP}, for every~$f\in L^2(\vol_\g)^+$,
\[
(A_m^{-s} f)(x) = \frac{1}{\Gamma(s)}\int_0^\infty e^{-m^2 t} t^{s-1} \int p_t(x,y)\, f(y)\, \dvol_\g(y)\diff t \comma
\]
and the conclusion follows by Tonelli's Theorem and the definition~\eqref{eq:Bessel} of~$G_{s,m}$.
Assertions~\iref{i:l:Representation:2} and~\iref{i:l:Representation:3} are straightforward.
%
\end{proof}

\subsection{The case of manifolds of bounded geometry}\label{ss:SobolevSpaces}
Let~$\mcC^\infty_c$ be the space of all smooth compactly supported functions on~$\M$.
We recall some definitions of spaces of weakly differentiable functions on~$\M$.

\subsubsection{Bessel potential spaces}
Fix $m>0$, let~$p\in [1,\infty)$ and denote by~$p'\eqdef \frac{p}{p-1}$ the H\"older conjugate of~$p\in (1,\infty)$.
Following~\cite{Str83}, we define the \emph{Bessel potential spaces}~$L^{s,p}_m$, $s\geq 0$, as 
the space of all~$u\in L^p$ so that~$u=A_m^{-s/2}v$ for some~$v\in L^p$, endowed with the norm $\norm{u}_{L^{s,p}_m}\eqdef \norm{v}_p$. For~$s<0$, we define~$L^{s,p}_m$ as the space of all distributions~$u$ on~$\mssM$ of the form~$u=A_m^kv$, where~$v\in L^{2k+s,p}_m$ and~$k$ is any integer so that~$2k+s>0$, endowed with the norm~$\norm{u}_{L^{s,p}_m}\eqdef \norm{v}_{L^{2k+s,p}_m}$.

As it turns out, the above definition is well-posed, i.e.\ independent of~$k$, and we have the following result of R.~S.~Strichartz'.

\begin{lem}[\cite{Str83}, \S4]\label{l:Equivalence}
The spaces~$L^{s,p}_m$,~$s\in\R$, are Banach spaces (Hilbert spaces for~$p=2$).
The natural inclusion~$L^{s,p}_m\subset L^{r,p}_m$,~$s>r$, is bounded and dense for every~$r,s\in \R$ and~$p\in (1,\infty)$.
Furthermore,~$\mcC^\infty_c$ is dense in~$L^{s,p}_m$ for every~$s\in \R$, $m>0$ and~$p\in (1,\infty)$.
As a consequence, the $L^2$-scalar product~$\scalar{\phi}{\psi}_{L^2}$,~$\phi,\psi\in \mcC^\infty_c$, extends to a bounded bilinear form between~$L^{s,p}_m$ and~$L^{-s,p'}_m$, $s>0$, thus establishing isometric isomorphisms between~$L^{s,p}_m$ and~$(L^{-s,p'}_m)'$, $s\in \R$, $p\in(1,\infty)$.
%
%
For every~$m,s>0$, the space $L^{s,p}_m$ coincides with the $L^p$-domain of~$(-\Delta)^{s/2}$, and
the norm~$\norm{\emparg}_{L^{s,p}_m}$ is equivalent to the graph-norm~$\norm{\emparg}_p+\tnorm{(-\Delta)^{s/2}\,\emparg}_p$.
\end{lem}

We note that, for $m_1,m_2>0$, the spaces~$L^{s,p}_{m_1}=L^{s,p}_{m_2}$ coincide setwise, and the corresponding norms are bi-Lipschitz equivalent.
For the sake of notational simplicity, we set~$H_m^{s}\eqdef L^{s,2}_m$ for~$s\in\R$,~$m>0$.

\subsubsection{Standard Sobolev spaces}
For a given local chart on~$\M$ let~$\nabla_{\alpha_i}$ be the corresponding covariant derivatives.
For smooth~$f\colon \mssM\rar\R$ and a non-negative integer~$k$, we set~$\abs{\nabla^0f}\eqdef \abs{f}$ and let~$\abs{\nabla^k f}$ be defined by
\begin{align*}
\abs{\nabla^k f}^2\eqdef \mssg^{\alpha_1\beta_1}\cdots \mssg^{\alpha_k\beta_k}\nabla_{\alpha_1}\cdots \nabla_{\alpha_k} f \cdot \nabla_{\beta_1} \cdots \nabla_{\beta_k} f \fstop
\end{align*}
For~$p\in (1,\infty)$, we  denote by~$E^{k,p}$ the space of all functions~$f\in\mcC^\infty$ so that~$\abs{\nabla^i f}$ is in~$L^p=L^p(\vol_\g)$ for every~$0\leq i\leq k$, and define the Sobolev space~$W^{k,p}$ as the completion of~$E^{k,p}$ with respect to the norm
\begin{align*}
\norm{f}_{W^{k,p}}\eqdef \sum_{i=0}^k \norm{\abs{\nabla^i f}}_p\comma \qquad f\in E^{k,p}\fstop
\end{align*}
The space~$W^{k,p}_\ast$ is the closure in~$W^{k,p}$ of~$\mcC^\infty_c$.

\subsubsection{Manifolds of bounded geometry} To simplify the presentation, at some places in the sequel we make  the following assumption, corresponding to~$\msH_\infty$ in~\cite[D\'{e}f.~3]{Aub76}. 

\begin{ass}\label{ass:Aubin}
$(\mssM,\g)$ has \emph{bounded geometry,} i.e.\ the injectivity radius is bounded away from~$0$, and
for every~$k\in \N_0$ there exists a constant~$C_k=C_{k,\g}$ so that the $k^\textrm{th}$-covariant derivative~$\nabla^k R^\g$ of the Riemann tensor~$R^\g$ satisfies $\abs{\nabla^k R^\g}_\g\leq C_k$.
\end{ass}

\begin{rem}
It is the main result of~\cite{MulNar15} that, on an arbitrary smooth differential manifold, the conformal class~$[\tilde\mssg]$ of any chosen Riemannian metric~$\tilde\mssg$ contains a Riemannian metric~$\mssg$ of bouded geometry.
Thus, Assumption~\ref{ass:Aubin} poses no topological restriction on the class of manifolds we consider.

Our main interest lies in compact manifolds and in homogeneous spaces. All these spaces satisfy the above assumption.
\end{rem}

By Lemma~\ref{l:Equivalence} above and e.g.~\cite[\S7.4.5]{Tri92}, under Assumption~\ref{ass:Aubin}, we have that~$W^{k,p}_\ast=W^{k,p}$ and $W^{k,p}\cong L^{k,p}_m$ (bi-Lipschitz equivalence) for every integer~$k\geq 0$  and $m>0$.
Furthermore,~$L^{s,p}_m$ for~$s\in\R$ may be equivalently defined via localization and pull-back onto~$\R^d$, by using geodesic normal coordinates and corresponding fractional Sobolev spaces on~$\R^d$, see \cite[\S\S{7.2.2, 7.4.5}]{Tri92} or~\cite{GroSch13}.
In particular we have the following:

\begin{lem}\label{l:SobolevEmbedding}
Under Assumption~\ref{ass:Aubin}, all the standard Sobolev--Morrey and Rellich--Kondrashov embeddings hold for~$L^{s,p}_m$.
\end{lem}

\begin{rem}\label{r:Assumption}
There exist complete non-compact manifolds with Ricci curvature bounded below for which the whole scale of Sobolev embeddings fails, that is~$W^{1,p}\not\hookrightarrow L^q$ for all~$1\leq q < n$ and~$1/p=1/q-1/n$, e.g.~\cite[Prop.~3.13, p.~30]{Heb96}.
\end{rem}

We conclude this section with an auxiliary result.
\begin{lem}\label{l:HilbertIsometry}
$A_m^{(r-s)/2}\colon H^r_m\longrar H^s_m$ is an isometry of Hilbert spaces for every~$r,s\in\R$ and~$m>0$.
\begin{proof}
By duality, it suffices to show the statement for~$r,s>0$. In this case, by the definition of~$H^t_m$, $t>0$, and by the semigroup property 
of~$t\mapsto A_m^t$, $t>0$,
\[
\tnorm{A^{(r-s)/2}\phi}_{H^s_m}=\tnorm{A^{s/2}_m A^{(r-s)/2}\phi}_{L^2}=\tnorm{A^{r/2}_m\phi}_{L^2}=\norm{\phi}_{H^r_m} \comma \qquad \phi\in\Test \fstop
\]
The extension to~$H^s_m$ follows by the density of~$\Test$ in~$H^s_m$, Lemma~\ref{l:Equivalence}.
\end{proof}
\end{lem}

\subsubsection{Test functions}
Denote by~$\msD\eqdef \mcC^\infty_c$ the space of smooth compactly supported functions on~$\mssM$ endowed with its canonical LF topology.
It is noted in the comments preceding~\cite[Ch.~II, Thm.~10, p.~55]{Gro66} that~$\msD$ is a nuclear space.
We denote by~$\msD'$ the topological dual of~$\msD$, and by~$\scalar{\emparg}{\emparg}={}_{\msD'}\!\scalar{\emparg}{\emparg}_{\msD}$ the canonical duality pairing, extending the $L^2(\vol_\g)$-scalar product.
The weak topology~$\sigma(\Test',\Test)$ is the coarsest topology for which all functionals of the form~$\scalar{\emparg}{\phi}$, with~$\phi\in\Test$, are continuous.
We write~$\Test'_\sigma$ for the space~$\Test'$ endowed with the weak topology.
Recall that a set~$B\subset \Test$ is bounded if for every neighborhood~$U\subset \Test$ of the origin in~$\Test$ there exists~$\lambda\geq 0$ such that~$B\subset \lambda U$.
The strong topology~$\beta(\Test',\Test)$ on~$\Test'$ is the topology of uniform convergence on bounded sets in~$\Test$, e.g.~\cite[II.19, Example~IV, p.~198]{Tre67}.
We write~$\Test'_\beta$ for the space~$\Test'$ endowed with the strong topology.

\begin{lem}\label{l:ContEmbedding}
The space~$\msD$ embeds continuously into~$H^s_m$ for every~$s\in\R$ and every~$m>0$.
\begin{proof}
A proof is standard in the case when~$s>0$ is a positive integer. The conclusion for general~$s$ follows since the identical inclusion~$H^s_m\hookrightarrow H^k_m$ is continuous for every integer~$k\leq s$ by the very definition of Bessel potential space.
\end{proof}
\end{lem}

\subsubsection{Heat-kernel estimates}
We collect here some estimates for the heat kernel on~$(\M,\g)$, which we shall make use of throughout the rest of the work. 
We also provide estimates on its first and second derivatives, which we need for the Green kernel asymptotics in Section \ref{sec:asymptotics}. These estimates are sharp.

\begin{lem}\label{l:EstimatesBG} Let~$(\mssM,\g)$ be a Riemannian manifold of bounded geometry. 
Then:
	\begin{enumerate}[$(i)$]
		\item\label{i:l:EstimatesBG:1} there exists a constant~$C>0$, so that for all~$x,y\in \mssM$ and every~$t>0$
		\begin{align}
			\label{eq:l:EstimatesBG:1}
			p_t(x,y)&\leq C(t^{-n/2}\vee 1)\, e^{-\frac{\mssd^2(x,y)}{C t}} \semicolon
		\end{align}
		
		\item\label{i:l:EstimatesBG:2} 
		there exists a constant~$C>0$ 
		, so that for all~$x,y\in \mssM$ and every~$t>0$
		\begin{align}
			\label{eq:l:EstimatesBG:2}
			\abs{\nabla p_t(x,y)}\leq&\ C \tparen{t^{-n/2-1/2}\vee 1}\, e^{-\frac{\mssd^2(x,y)}{C t}} \semicolon
		\end{align}
		
		\item\label{i:l:EstimatesBG:3} there exists a constant~$C>0$, 
	 so that for all~$x,y\in\M$ and every~$t>0$
		\begin{align}\label{eq:l:EstimatesBG:3}
			\abs{\Delta\, p_t(x,y)}\leq C\tparen{t^{-n/2-1}\vee 1}
			\, e^{-\frac{\mssd^2(x,y)}{C t}} \semicolon
		\end{align}
	\item\label{i:l:EstimatesBG:4} 
	there exists a constant~$C>0$, 
so that for all~$x,y\in \M$ and every~$t>0$
\begin{align}\label{eq:l:EstimatesBG:4}
\abs{\nabla_1\nabla_2\, p_t(x,y)}\leq C\tparen{t^{-n/2-1}\vee 1}
\, e^{-\frac{\mssd^2(x,y)}{C t}}  \fstop
\end{align}
	\end{enumerate}
	
\end{lem}

\begin{proof}
Throughout the proof~$C>0$ is a constant only depending on~$(\M,\g)$, possibly changing from line to line.
\iref{i:l:EstimatesBG:1} In light of the bounded geometry assumption we have the Gaussian heat kernel estimate 
\begin{align*}
p_t(x,y)\leq \frac{C}{t\, \vol_\g \tparen{B_{\sqrt t\wedge 1}(x)}}\left(1+\frac{\mssd^2(x,y)}{t}\right)^{\nu_0/2}e^{-\frac{\mssd^2(x,y)}{4t}}
\end{align*}
for some $0<r_0<\textrm{inj}(\M)$ and $\nu_0>0$ \cite[Thm.~4.2]{sc}. The claim follows since $\vol_\g B_r(x)\geq Cr^n$ for all $r<\textrm{inj}(\M)$ by virtue of \cite[Prop.~14]{Cr}. 

\iref{i:l:EstimatesBG:2} 
Let $Q=B_{\sqrt t}(x)\times[t/2,t]$. Let $u(z,\tau)=p_\tau(z,y)$ on $Q$. Then by~\cite[Thm.~1.1]{S.Z.} we have
\begin{align}\label{eq: grad est}
\frac{|\nabla u|}u\leq C\left(\frac{1}{\sqrt t}+\sqrt{K}\right)\left(1+\log\frac{\sup_Q u}{u}\right),
\end{align}
where $-K$, $K\geq0$, is a lower bound of the Ricci curvature.
By \cite[Theorem 4.2]{sc} we have
\begin{align*}
u(z,\tau) \leq \frac{C}{\vol_\g \tparen{B_{\sqrt\tau\wedge r_0}(z)}}e^{-\mssd^2(z,y)/C\tau}\leq \frac{C}{\vol_\g \tparen{B_{\sqrt{t}\wedge r_0}(x)}},
\end{align*}
where we used the volume-doubling property
and consequently by \eqref{eq: grad est} and \eqref{eq:l:EstimatesBG:1}
\begin{align*}
{|\nabla u(x,t)|}\leq &\ C\left(\frac{1}{\sqrt t}+\sqrt{K}\right)\left(1+\log\frac{\vol_\g \tparen{B_{\sqrt{t}\wedge r_0}(x)}^{-1}}{u(x,t)}\right)u(x,t)\\
\leq &\ C\left(\frac{1}{\sqrt t}+\sqrt{K}\right)\left(1+\log\frac{\vol_\g \tparen{B_{\sqrt{t}\wedge r_0}(x)}^{-1}}{u(x,t)}\right)(t^{-n/2}\vee 1)\, e^{-\frac{\mssd^2(x,y)}{C t}}.
\end{align*}
In order to estimate $u(x,t)$ from below we use Corollary 1.2 in \cite{Li-Zhang} and obtain for $Kt\leq 1$
\begin{align*}
|\nabla u(x,t)|\leq \frac{C}{\sqrt t}(t^{-n/2}\vee 1)\, e^{-\frac{\mssd^2(x,y)}{C t}} e^{C t}\leq  C \tparen{t^{-n/2-1/2}\vee 1}\, e^{-\frac{\mssd^2(x,y)}{C t}} \fstop
\end{align*}
For $Kt> 1$ we use Corollary 1.7 in \cite{Li-Zhang}
\begin{align*}
|\nabla u(x,t)|\leq C\, e^{-\frac{\mssd^2(x,y)}{C t}} \leq \ C \tparen{t^{-n/2-1/2}\vee 1}\, e^{-\frac{\mssd^2(x,y)}{C t}} \comma
\end{align*}
which finishes the proof.

\iref{i:l:EstimatesBG:3} 
In light of the bounded geometry assumption, \cite[Thm.~4.2]{sc} yields
\begin{align*}
|\partial_tp_t(x,y)|\leq \frac{C}{t\, \vol_\g \tparen{B_{\sqrt t\wedge r_0}(x)}} \left(1+\frac{\mssd^2(x,y)}{t}\right)^{\nu_0/2+1}e^{-\frac{\mssd^2(x,y)}{4t}}
\end{align*}
for~$\nu_0$ and~$r_0$ as in~\iref{i:l:EstimatesBG:1}.
We estimate the volume of the ball from below as in \iref{i:l:EstimatesBG:1} by applying \cite[Prop.~14]{Cr}. Noting that $\partial_tp_t(x,y)=\Delta p_t(x,y)$ the result follows.

\iref{i:l:EstimatesBG:4} 
It follows from~\cite[Thm.~2.1]{Li91} that there exists a constant~$C>3$ depending on~$(\mssM,\g)$ so that, for all~$x,y\in \mssM$,
\begin{align*}
\abs{\nabla_1\nabla_2\, p_t(x,y)}\leq C\, \partial_t p_t(x,y) +C\tparen{t^{-1}\vee 1}p_t(x,y)\comma\qquad t>0\fstop
\end{align*}
Since~$p_t(\emparg, y)$ is a solution to the heat equation, and using~\eqref{eq:l:EstimatesBG:3}, we have for all~$t\in (0,2]$ and every~$x,y\in \mssM$,
\begin{align}
\nonumber
\abs{\nabla_1\nabla_2\, p_t(x,y)}\leq&\ C\tparen{\Delta p_t(\emparg,y)}(x) +C\tparen{t^{-1}\vee 1}p_t(x,y)
\\
\nonumber
\leq&\ C \tparen{t^{-n/2-1}\vee 1}e^{-\frac{\mssd^2(x,y)}{Ct}}+C\tparen{t^{-1}\vee 1}p_t(x,y)
\end{align}
for some constant~$C>0$ only depending on~$(\mssM,\g)$ and possibly changing from line to line.
Combining this with the heat kernel estimate \eqref{eq:l:EstimatesBG:1} yields the claim for $t\leq 2$.
For $t\geq 2$ the claim follows from the bound for $t\le 1$ combined with the following inequalities for~$t\geq 1$:
\begin{align}\label{eq:l:EstimatesBG:9}
\begin{aligned}
\big|\nabla_x\nabla_y p_{t+1}(x,y)\big|&=
  \bigg|\iint \nabla_x  p_{1/2}(x,u) \ p_{t}(u,v) \ \nabla_y p_{1/2}(v,y)\, \dvol_\g(u)\,\dvol_\g(v)\bigg|\\
&\le \sup_{u\in\M}|\nabla_x  p_{1/2}(x,u)|\cdot \sup_{v\in\M}| \nabla_y  p_{1/2}(y,v)|\cdot \iint  p_t(u,v) \,\dvol_\g(u)\, \dvol_\g(v)\le C \comma
\end{aligned}
\end{align}
which concludes the proof.
\end{proof}

\subsection{The case of closed manifolds}
Let us now specialize our constructions to the case when~$\M$ is additionally \emph{closed}, i.e.\ compact and without boundary.

If~$\mssM$ is closed, the operator~$(m^2-\frac12\Delta)^{-1}$ is compact on~$L^2(\vol_\g)$, and thus has discrete spectrum.
We denote by~$\seq{\phi_j}_{j\in\N_0}$ the complete $L^2$-orthonormal system consisting of eigenfunctions of~$-\Delta$, each with corresponding eigenvalue~$\lambda_j$, so that~$(\Delta+\lambda_j)\phi_j=0$ for every~$j$.
Since~$\mssM$ is connected, we have~$0=\lambda_0<\lambda_1$ and~$\phi_0\equiv \vol_\g(\mssM)^{-1/2}$.
Weyl's asymptotic law implies that for some $c>0$,
\begin{align}\label{eq: weyl}
\lambda_j\ge c \, j^{2/n},\qquad j\in\N\fstop
\end{align}

\subsubsection{Grounding}
If $\M$ is closed, we further define the  \emph{grounded Green operator of order $s$ with mass parameter $m$} as the (bounded) self-adjoint operator $\mathring A_m^{-s}f\eqdef A_m^{-s}(\mathring f)$ on $L^2(\mssM)$ with
\begin{equation*}
\mathring f\eqdef f-\frac1{\vol_\g(\M)}\int f\diff\vol_\g\fstop
\end{equation*}

We start by refining the heat-kernel estimates in Lemma~\ref{l:EstimatesBG} to the closed case.

\begin{lem}[Heat kernel estimates: compact case]\label{l:EstimatesC} Let~$(\mssM,\g)$ be a closed Riemannian manifold. Then,
\begin{enumerate}[$(i)$]
\item\label{i:l:EstimatesC:1} there exists a constant~$C>0$, so that for all~$x,y\in \mssM$ and every~$t>0$
\begin{align}
\label{eq:l:EstimatesC:1}
 p_t(x,y)&\leq C( t^{-n/2}\vee 1)\, e^{-\frac{\mssd^2(x,y)}{C t}} \comma\\
\label{eq:l:EstimatesC:1gr}
 |\mathring  p_t(x,y)|&\leq C ( t^{-n/2}\vee 1)
  \,e^{-\lambda_1\,t/2}\semicolon
\end{align}

\item\label{i:l:EstimatesC:2} for every~$\ell\in \N_0$ 
there exists a constant~$C=C(\ell)>0$
, so that for all~$x,y\in \mssM$ and every~$t>0$
\begin{align}
\label{eq:l:EstimatesC:2}
\abs{\nabla^\ell p_t(x,y)}\leq&\ C \tparen{t^{-n/2-\ell/2}\vee 1}
 \, e^{-\frac{\mssd^2(x,y)}{C t}} \,e^{-\lambda_1\,t/2} \semicolon
\end{align}


\item\label{i:l:EstimatesC:3} there exists a constant~$C>0$, 
so that for all~$x,y\in \M$ and every~$t>0$
\begin{align}\label{eq:l:EstimatesC:3}
\abs{\nabla_1\nabla_2\, p_t(x,y)}\leq C\tparen{t^{-n/2-1}\vee 1}
\, e^{-\frac{\mssd^2(x,y)}{C t}} \,e^{-\lambda_1\,t/2} \fstop
\end{align}
\end{enumerate}
\end{lem}

\begin{proof}
\iref{i:l:EstimatesC:1}
The estimate~\eqref{eq:l:EstimatesC:1} was already shown in Lemma~\ref{l:EstimatesBG}.
We provide here an alternative proof which we subsequently adapt to the case of~$\mathring{p}_t$.
For $t\ge1$, the estimate~\eqref{eq:l:EstimatesC:1} immediately follows from the fact that by compactness of $\M$  the heat kernel is uniformly bounded on $[1,\infty)\times\M\times\M$. For $t\le1$ it follows from the celebrated estimate of Li and Yau~\cite[Cor.~3.1]{LiYau86}, combined with the fact that ~$\vol_\mssg(B_{\sqrt{t}}(x))\ge\frac1C t^{n/2}$ for each~$x\in \mssM$, which in turn follows from Bishop--Gromov volume comparison and compactness of~$\mssM$, see, e.g.,~\cite[Lem.~9.1.36, p.~269]{Pet06}. 
%
%

Since $-C\le \mathring p_t(x,y)\le  p_t(x,y)$, the estimate  \eqref{eq:l:EstimatesC:1gr} for $t\leq 2$ follows immediately from the previous estimate.
In order to prove \eqref{eq:l:EstimatesC:1gr} for $t\geq 2$, note that, for~$t\geq 1$,
\begin{align*}|\mathring p_{t+1}(x,y)|&=\bigg|\iint \mathring p_{1/2}(x,u)\, \mathring p_{t}(u,v)\, \mathring p_{1/2}(v,y)  \,\dvol_\g(u)\, \dvol_\g(v)\bigg|\\
&\le \sup_{u\in\M}|\mathring p_{1/2}(x,u)|\cdot \sup_{v\in\M}|\mathring p_{1/2}(y,v)|\cdot \iint \big|\mathring p_t(u,v)\big|\,\dvol_\g(u)\, \dvol_\g(v)\\
&\le C \iint \big|\mathring p_t(u,v)\big|\,\dvol_\g(u)\, \dvol_\g(v)
\end{align*}
uniformly in $x,y\in\M$.
Moreover, note that by the standard spectral calculus for~$\Delta$ and ultracontractivity of the heat semigroup, see e.g.~\cite[Thm.~2.1.4]{Dav89}, we may express the grounded heat kernel on~$\mssM$ as the uniform limit of the series
\begin{align*}
\mathring p_t(x,y)=\sum_{j\in \N} e^{-t\lambda_j/2} \phi_j(x)\, \phi_j(y) \comma \qquad x,y\in \mssM\comma
\end{align*}
and with this we obtain
\begin{align}\label{average-decay}
\iint \big|\mathring p_t(x,y)\big|\,\dvol_\g(x)\, \dvol_\g(y)
&=\iint\bigg| \sum_{j=1}^\infty e^{-\lambda_jt/ 2} \phi_j(x)\, \phi_j(y)
\bigg|\,\dvol_\g(x)\, \dvol_\g(y)\\
&\le \nonumber
C\, \sum_{j=1}^\infty e^{-\lambda_j t/2}\leq 
C\, e^{-\lambda_1 t/2}\sum_{j=1}^\infty e^{(\lambda_1-\lambda_j)/2}\\
&=C\, e^{-\lambda_1t/2}\, e^{\lambda_1/2}\int \mathring p_1(x,x)\, \dvol_\g(x)\nonumber\\
&=C'\, e^{-\lambda_1t/2} \nonumber\fstop
\end{align}
This proves the claim.

\iref{i:l:EstimatesC:2} It is shown in~\cite[Eqn.~(1.1)]{StrTur98} that for every~$x,y\in \mssM$
\begin{align*}
\abs{\tparen{\nabla^\ell \log p_t(\emparg,y)}(x)}\leq&\ C_\ell\paren{\frac{1}{t}+\frac{\mssd^2(x,y)}{t^2}}^{\ell/2} \comma \qquad t\in (0,2]\comma
\end{align*}
for some constant~$C_\ell$, henceforth possibly changing from line to line. As a consequence,
\begin{align}\label{eq:l:EstimatesC:4}
\abs{\tparen{\nabla^\ell p_t(\emparg,y)}(x)}\leq C_\ell \paren{\frac{1}{t}+\frac{\mssd^2(x,y)}{t^2}}^{\ell/2} p_t(x,y) \comma \qquad t\in (0,2]\fstop
\end{align}
In combination with the heat kernel estimate \eqref{eq:l:EstimatesC:1} from above this yields the claim for $t\leq 2$.
As in part~\iref{i:l:EstimatesC:1},  the claim for $t\geq 2$ follows from the  bound for $t\leq 1$ together with the fact that, for~$t\geq 1$,
\begin{align*}
\big|\nabla_x^\ell p_{t+1}(x,y)\big|&=
 \big|\nabla_x^\ell \mathring p_{t+1}(x,y)\big|\\
 &=\bigg|\iint \nabla_x^\ell \mathring p_{1/2}(x,u) \ \mathring p_{t}(u,v) \ \mathring p_{1/2}(v,y)\, \dvol_\g(u)\, \dvol_\g(v)\bigg|\\
&\le \sup_{u\in\M}|\nabla_x^\ell \mathring p_{1/2}(x,u)|\cdot \sup_{v\in\M}|\mathring p_{1/2}(y,v)|\cdot \iint \big| \mathring p_t(u,v)\big|\,\dvol_\g(u)\, \dvol_\g(v)\\
&\le C \, e^{-\lambda_1 t/2}
\end{align*}
according to the previous estimates \eqref{eq:l:EstimatesC:4}, \eqref{eq:l:EstimatesC:1}, and \eqref{average-decay}.

\smallskip

\iref{i:l:EstimatesC:3}
Let us first note that~\cite[Thm.~2.1]{Li91} holds with identical proof also in the case of closed~$\mssM$.
Similarly to the proof of Lemma~\ref{l:EstimatesBG}, it follows from~\cite[Thm.~2.1]{Li91} that there exists a constant~$C>3$ depending on~$(\mssM,\g)$ so that, for all~$x,y\in \mssM$,
\begin{align*}
\abs{\nabla_1\nabla_2\, p_t(x,y)}\leq C\, \partial_t p_t(x,y) +C\tparen{t^{-1}\vee 1}p_t(x,y)\comma\qquad t>0\fstop
\end{align*}
Since~$p_t(\emparg, y)$ is a solution to the heat equation, and using~\eqref{eq:l:EstimatesC:4}, we have for all~$t\in (0,2]$ and every~$x,y\in \mssM$,
\begin{align}
\nonumber
\abs{\nabla_1\nabla_2\, p_t(x,y)}\leq&\ C\tparen{\Delta p_t(\emparg,y)}(x) +C\tparen{t^{-1}\vee 1}p_t(x,y)
\\
\nonumber
\leq&\ C \abs{\tparen{\nabla^2p_t(\emparg,y)}(x)}+C\tparen{t^{-1}\vee 1}p_t(x,y)
\\
\nonumber
\leq&\ C \tparen{t^{-1}\vee 1}\paren{\frac{\mssd^2(x,y)}{t}+1} p_t(x,y)+C\tparen{t^{-1}\vee 1}p_t(x,y)
\\
\label{eq:l:EstimatesBG:8}
\leq&\ C\tparen{t^{-1}\vee 1} \paren{\frac{\mssd^2(x,y)}{t}+1} p_t(x,y)\comma
\end{align}
for some constant~$C>0$ depending on~$(\mssM,\g)$ and possibly changing from line to line.
Combining this with the heat kernel estimate \eqref{eq:l:EstimatesBG:1} yields the claim for $t\leq 2$.
Again, for $t\geq 2$ the claim follows from the bound for $t\le 1$ combined with the following inequalities for~$t\geq 1$:
\begin{align*}
\big|\nabla_x\nabla_y p_{t+1}(x,y)\big|&=
  \bigg|\iint \nabla_x \mathring p_{1/2}(x,u) \ \mathring p_{t}(u,v) \ \nabla_y\mathring p_{1/2}(v,y)\, \dvol_\g(u)\,\dvol_\g(v)\bigg|\\
&\le \sup_{u\in\M}|\nabla_x \mathring p_{1/2}(x,u)|\cdot \sup_{v\in\M}| \nabla_y \mathring p_{1/2}(y,v)|\cdot \iint \big| \mathring p_t(u,v)\big|\,\dvol_\g(u)\, \dvol_\g(v)\\
&\le C \, e^{-\lambda_1 t/2} \fstop &&\qedhere
\end{align*}
%
%
\end{proof}

\begin{lem}\label{l:RepresentationGrounded}
If $\M$ is closed and $s>0$, then $\mathring A_m^{-s}$ is an integral operator with density given by the \emph{grounded Green kernel of order $s$ with mass parameter~$m\geq 0$}, defined in terms of the \emph{grounded heat kernel},
\begin{equation}\label{eq:Bessel2}
\mathring G_{s,m}(x,y)\eqdef \frac{1}{\Gamma(s)}\int_0^\infty e^{-m^2t}\ t^{s-1} \ \mathring p_t(x,y)\diff t\comma
\qquad
\mathring p_t(x,y)\eqdef p_t(x,y)-\frac1{\vol_\g(\mssM)}\fstop
\end{equation}
For each $m\ge0$, the family~$(\mathring G_{s,m})_{s>0}$ is a convolution semigroup of kernels, and~$\int \mathring G_{s,m}(x,\emparg) \diff\vol_\g=0$ for all~$x\in \mssM$, $s>0$.
\end{lem}
Of particular interest will be  $\mathring G_{s,0}$, the  \emph{massless grounded Green kernel of order $s$}.

\begin{proof}[Proof of Lemma~\ref{l:RepresentationGrounded}] Let us first observe that $\mathring G_{s,m}(x,y)$ as defined above is finite for all $x\not=y$ by virtue of \eqref{eq:l:EstimatesC:1gr}.
We claim that the integral
\begin{align}\label{eq:l:RepresentationGrounded:1}
(\mathring{\sf G}_{s,m}f)(x)\eqdef \int
\mathring G_{s,m}(x,y)\,
f(y)\,\dvol_\g(y)
\end{align}
is absolutely convergent for every $f\in L^2$ and a.e.~$x$. Indeed, it defines an $L^2$-function according to
\begin{align*}
\int&\abs{ \frac{1}{\Gamma(s)} \int \int_0^1 e^{-m^2 t} t^{s-1} \mathring p_t(x,y)\diff t\, f(y)\dvol_\g(y)}^2 \dvol_\g(x)\\
&\le 
\int\left( \frac{1}{\Gamma(s)} \int \int_0^1  t^{s-1} \left[p_t(x,y)+\frac1{\vol_\g(\M)}\right]\, \diff t\, |f|(y)\dvol_\g(y)\right)^2 \dvol_\g(x)\\
&\le 2e\, \|{\sf G}_{s,1}f\|_{L^2}^2 +2\, \frac1{(s\,\Gamma(s))^2\cdot\vol_\g(\M)}\, \|f\|^2_{L^2}\\
&\le C'\, \|f\|^2_{L^2}<\infty
\end{align*}
and since, due to \eqref{eq:l:EstimatesC:2},
\begin{align*}
\int&\abs{ \frac{1}{\Gamma(s)} \int \int_1^\infty e^{-m^2 t} t^{s-1} \mathring p_t(x,y)\diff t\, f(y)\dvol_\g(y)}^2 \dvol_\g(x)\\
&\le C\, \int\left[ \frac{1}{\Gamma(s)} \int \int_1^\infty e^{-m^2 t} t^{s-1} e^{-\lambda_1\, t/2}\diff t\, |f|(y)\dvol_\g(y)\right]^2 \dvol_\g(x)\\
&\le C'\, \|f\|^2_{L^2}\fstop\end{align*}
Thus,
\begin{align}
\mathring{\sf G}_{s,m}\colon L^2\longrar L^2 \quad \text{is a bounded operator}
\end{align}
and moreover, (due to the absolute convergence of the integrals) by Fubini's Theorem,
\begin{align}
\mathring{\sf G}_{s,m}f(x)&=
\frac1{\Gamma(s)}\int_0^\infty e^{-m^2t}t^{s-1}\int\mathring p_t(x,y)\, f(y)\,\dvol_\g(y)\diff t=(\mathring A_m^{-s}f)(x)\fstop &&\qedhere
\end{align}
\end{proof}

\begin{rem}\label{r:DistributionalSolution}
\begin{enumerate}[$(a)$]
\item\label{i:r:DistributionalSolution:1} For $m>0$ 
\begin{equation*}
\mathring G_{s,m}(x,y)=G_{s,m}(x,y)-\frac1{m^{2s}\, \vol_\g(\M)}\fstop
\end{equation*}
\item\label{i:r:DistributionalSolution:2} For each $s>0$,~$m\geq 0$ and $x\in\M$, the distribution~$\mathring G_{s,m}(x,\emparg)\vol_\g$ is the unique distributional solution to 
\begin{align}\label{distr-G0}
\paren{m^2-\tfrac{1}{2}\Delta}^s u=\delta_x -\frac1{\vol_\g(\M)}\, \vol_\g \qquad \text{and} \qquad \scalar{u}{\car}=0 \fstop
\end{align}
\end{enumerate}
\end{rem}

\begin{proof}
As~\ref{i:r:DistributionalSolution:1} is straightforward, we only prove~\ref{i:r:DistributionalSolution:2}.
It is also standard that~$\mathring G_{s,m}(x,\emparg)\vol_\g$ is \emph{a} distributional solution to~\eqref{distr-G0}, thus it suffices to show that the associated homogeneous equations~$A_m^s u=0$ and~$\scalar{u}{\car}=0$ admit a unique solution for every~$s\in\R$.

To this end, denote by~$\mathring\Test\eqdef\set{\mathring\phi: \phi\in\Test}$ the space of grounded test functions.
Equivalently, we show that~$A_m^s\colon \mathring\Test\to\mathring\Test$ is a bijection for every~$s\in\R$.
The fact that~$A_m^k(\Test)\subset\Test$ for integer~$k$ holds by the standard Schauder estimates for elliptic operators (for closed manifolds see e.g.~\cite[Thm.~III.5.2 (iii), p.~193]{LawsonMichelsohn89}).
This is readily extended to~$s\in\R$ noting that the integral operator~$\mathsf{G}_{s,m}$ with kernel~$G_{s,m}(x,\emparg)$ is a smoothing operator for~$s>n/2$.

For~$m>0$, the injectivity on~$\mathring\Test$ (in fact on~$L^2(\vol_\g)$) holds by Lemma~\ref{l:HilbertIsometry}, and the surjectivity by Lemma~\ref{l:Representation}.
For~$m=0$, the injectivity holds since~$\ker A_0^k=\ker (-\Delta_\g)^k$ only consists of the constant functions for every non-negative integer~$k$, and the surjectivity holds by Lemma~\ref{l:RepresentationGrounded}.
We omit the details.
\end{proof}



\subsubsection{Eigenfunction expansion}\label{ss:eig-exp}
We conclude the analysis of the closed case by discussing the expansion of the Green kernels~$G_{s,m}$ and~$\mathring{G}_{s,m}$ in terms of eigenfunctions of the Laplace--Beltrami operator.

\begin{lem}\label{lm:eig-fct} Assume that $\mssM$ is closed. Then for all $m>0$ and $s>n/2$,
\begin{align}\label{eq:GCompactM}
G_{s,m}(x,y)=\sum_{j\in \N_0} \frac{\phi_j(x)\, \phi_j(y)}{(m^2+\lambda_j/2)^s} \comma \qquad  x,y\in \mssM\comma
\end{align}
where the series is absolutely convergent for every~$x,y\in \mssM$.

Furthermore, for all $m\ge0$ and $s>n/2$,
\begin{align}\label{eq:G0-CompactM}
\mathring G_{s,m}(x,y)=\sum_{j\in \N} \frac{\phi_j(x)\, \phi_j(y)}{(m^2+\lambda_j/2)^s} \comma \qquad x,y\in \mssM\fstop
\end{align}
(Note that the summation now starts at $j=1$.) In particular,
\begin{align}\label{eq:G00-CompactM}
\mathring G_{s,0}(x,y)=2^{s}\,\sum_{j\in \N} \frac{\phi_j(x)\, \phi_j(y)}{\lambda_j^s} \comma \qquad  x,y\in \mssM\fstop
\end{align}

\end{lem}

\begin{proof}
By the spectral calculus (e.g.~\cite[Thm.~2.1.4]{Dav89}), we may express the heat kernel on~$\mssM$ as the uniform limit of the series
\begin{align}\label{eq:Heat-CompactM}
p_t(x,y)=\sum_{j\in \N_0} e^{-t\lambda_j/2} \phi_j(x)\, \phi_j(y) \comma \qquad x,y\in \mssM\fstop
\end{align}
By virtue of \eqref{eq:Bessel}, \eqref{eq:l:EstimatesC:1}, and $s>\frac n2$ we have that $G_{s,m}(x,x)<\infty$.
By Dominated Convergence the representation \eqref{eq:GCompactM} follows for $x=y$. For $x,y\in \mssM$ we have that the series $\sum_{j\in\N_0} \frac{\phi_j(x)\, \phi_j(y)}{(m^2+\lambda_j/2)^s}$ is absolutely convergent due to Cauchy--Schwarz.
Hence \eqref{eq:GCompactM} follows again by Dominated Convergence.
With the same arguments but using \eqref{eq:Bessel2} and \eqref{eq:l:EstimatesC:1gr} instead of~\eqref{eq:Bessel} and \eqref{eq:l:EstimatesC:1}, we can show \eqref{eq:G0-CompactM}.
\end{proof}

\begin{rem}
The grounded Green kernel~$\mathring G_{s,0}(x,y)$ coincides, up to the multiplicative factor~$2^s$, with the celebrated Minakshisundaram--Pleijel $\zeta$-function~$\zeta^\Delta_{x,y}(s)$ of the Laplace--Beltrami operator on~$\mssM$, introduced in~\cite{MinPle49}.
The massive grounded Green kernel~$\mathring G_{s,m}(x,y)$ is therefore the Hurwitz regularization of~$\zeta^\Delta$ with parameter~$m^2$.
\end{rem}

\subsubsection{Sobolev spaces on compact manifolds} 
Again assume that $\M$ is closed, and let $(\phi_j)_{j\in\N_0}$ and $(\lambda_j)_{j\in\N_0}$ be as above.
Then for each $m>0$ and $s\in\R$,
\begin{equation*}
H^s_m=\Big\{f\in \Test':\ f=\sum_{j\in\N_0} \alpha_j\phi_j, \quad \sum_{j=0}^\infty \alpha_j^2\big(m^2+\lambda_j/2\big)^{s}<\infty\Big\}
\end{equation*}
with
$\big\|f\big\|^2_{H^s_m}=\sum_{j=0}^\infty \alpha_j^2\big(m^2+\lambda_j/2\big)^{s}$
and
$\scalar{f}{\psi}=\sum_{j=0}^\infty \alpha_j\,\scalar{\phi_j}{\psi}$ for $\psi\in  \Test$.
Note that for all $\psi\in  \Test$ and $k\in\N$ we have $\sum_{j=0}^\infty \abs{\lambda_j^k\scalar{\phi_j}{\psi}}^2<\infty$.

\begin{defs}\label{d:GroundedHSpaces}
If $\M$ is closed we define the \emph{grounded Sobolev spaces} for  $m\ge 0$ and $s\in\R$ by
\begin{equation*}
\mathring H^s_m=\set{f\in \Test':\ f=\sum_{j\in\N} \alpha_j\phi_j, \quad \sum_{j=1}^\infty \alpha_j^2\big(m^2+\lambda_j/2\big)^{s}<\infty }\comma
\end{equation*}
regarded as a subspace of~$H^s_m$.
\end{defs}

\begin{lem}\label{l:EquivalenceNorms}
Assume that $\M$ is closed.
\begin{enumerate}[$(i)$]
\item\label{i:l:EquivalenceNorms:1} For all $m\ge 0$ and $r,s\in\R$,
\begin{equation*}
\mathring A_m^{-(r-s)/2}=A_m^{-(r-s)/2}\colon \mathring H^s_m \longrightarrow \mathring H^r_m
\end{equation*}
is an isometry of Hilbert spaces.

\item\label{i:l:EquivalenceNorms:2} For all   $m> 0$ and $s\in\R$,
\begin{align*}
\mathring H^s_m = \tset{f\in H^s_m: \ \scalar{f}{\car}=0}\fstop
\end{align*}

\item\label{i:l:EquivalenceNorms:3} For all   $m> 0$ and $s\in\R$,
the spaces~$\mathring H^s_m$ and $\mathring H^s_0$ coincide setwise, and the corresponding norms are bi-Lipschitz equivalent.
\end{enumerate}
\end{lem}

\begin{proof} \iref{i:l:EquivalenceNorms:1} follows from Lemma~\ref{l:HilbertIsometry}.
\iref{i:l:EquivalenceNorms:2} follows by spectral calculus.
\iref{i:l:EquivalenceNorms:3} For $s\geq 0$,
\begin{equation*}
\sum_{j=1}^\infty \alpha_j^2\big(\lambda_j/2\big)^{s}\le \sum_{j=1}^\infty \alpha_j^2\big(m^2+\lambda_j/2\big)^{s}\le\bigg(\frac{m^2+\lambda_1/2}{\lambda_1/2}\bigg)^s\cdot \sum_{j=1}^\infty \alpha_j^2\big(\lambda_j/2\big)^{s}\comma
\end{equation*}
thus
\begin{equation*}
\norm{f}_{\mathring H^s_0}\le \norm{f}_{\mathring H^s_m}\le\Big(1+2m^2/\lambda_1\Big)^{s/2}\cdot \norm{f}_{\mathring H^s_0}\fstop
\end{equation*}
Similarly for $s<0$,
\begin{equation*}
\norm{f}_{\mathring H^s_0}\ge \norm{f}_{\mathring H^s_m}\ge\Big(1+2m^2/\lambda_1\Big)^{s/2}\cdot \norm{f}_{\mathring H^s_0}\fstop \qedhere
\end{equation*}
\end{proof}

\subsection{The noise distance}
Given any positive numbers $s,m$, a pseudo-distance~$\rho_{s,m}$ on~$\mssM$, called \emph{noise distance} (for reasons which become clear in Corollary \ref{cor-noise-dist}), is defined by
\begin{align}\label{green-dist}
\rho_{s,m}(x,y)\eqdef \bigg(\frac{1}{\Gamma(s)}\int_0^\infty\int_\M  e^{-m^2t}\ t^{s-1} \Big[  p_{t/2}(x,z)-p_{t/2}(y,z)\Big]^2\dvol_\g(z)\diff t\bigg)^{1/2}\fstop
\end{align}
Indeed, symmetry and triangle inequality are immediate consequences of the fact that this is the $L^2$-distance between $p_{\cdot/2}(x,\cdot)$ and $p_{\cdot/2}(y,\emparg)$ w.r.t.~a (possibly infinite) measure on $\R_+\times\M$.
In the case of closed~$\M$, the analogous definition for~$\mathring p_{\cdot/2}(\emparg,\emparg)$ results in~$\mathring\rho_{s,m}=\rho_{s,m}$.

\begin{rem} Note that by the symmetry and the Chapman--Kolmogorov property of the heat kernel,
\begin{equation*}
\int_\M  \Big[  p_{t/2}(x,z)-p_{t/2}(y,z)\Big]^2\dvol_\g(z)=p_t(x,x)+p_t(y,y)-2p_t(x,y)\fstop
\end{equation*}
Hence, for all $s,m\in(0,\infty)$ and all $x,y\in\M$ with $G_{s,m}(x,y)<\infty$,
\begin{align*}
\rho_{s,m}(x,y)=\ \Big[G_{s,m}(x,x)+ G_{s,m}(y,y)-2\, G_{s,m}(x,y)\Big]^{1/2}\fstop
\end{align*}
\end{rem}

\section{The Fractional Gaussian Field}
Let us now define Fractional Gaussian Fields.

\begin{thm}\label{ex-fgf}
For~$m>0$ and $s\in\R$, 
there exists a unique Radon Gaussian measure~$\mu_{m,s}$ on~$\Test'_\sigma$ with 
characteristic functional
\[
\Test\ni\phi\longmapsto \int_{\Test'} e^{\imu\, \scalar{\emparg}{\phi}} \diff\mu_{m,s}\comma \qquad \phi\in\Test\comma
\]
equal to
\begin{align}\label{eq:ChiBetaMS}
\chi_{m,s}\colon \phi\longmapsto \exp\quadre{-\tfrac{1}{2}\norm{\phi}_{H^{-s}_m}^2} \comma\qquad \phi\in\Test\fstop
\end{align}
\end{thm}

\begin{proof}
Note that~$\chi_{m,s}(0)=1$ and that~$\chi_{m,s}$ is positive definite, e.g.,~\cite[Prop.~2.4]{LodSheSunWat16}. Furthermore,~$\chi_{m,s}$ is additionally continuous on~$\Test$, since~$\Test$ embeds continuously into~$H^{-s}_m$ for every~$s\in \R$ and~$m>0$ by Lemma~\ref{l:ContEmbedding}.
Note that~$\beta(\Test',\Test)$ is finer than~$\sigma(\Test',\Test)$, hence every Radon \emph{probability} measure on~$\Test'_\beta$ restricts to a Radon probability measure on~$\Test'_\sigma$.
Since~$\Test$ is nuclear, by Bochner--Minlos Theorem in the form~\cite[\S{VI.4.3}, Thm.~4.3, p.~410]{VakTarCho87}, there exists a Radon probability measure~$\mu_{m,s}$ on~$\Test'_\beta$, and the conclusion follows by restricting this measure to a (non-relabeled) Radon measure on $\Test'_\sigma$.
\end{proof}

Everywhere in the following, $(\Omega,\msF,\mbfP)$ denotes a probability space supporting countably many i.i.d.\ Gaussian random variables.

\begin{defs} Let~$m>0$ and~$s\in\R$. An \emph{$m$-massive Fractional Gaussian Field on~$\mssM$ with regularity~$s$}, in short:~$\FGF[\mssM]{s,m}$, is any $\Test'$-valued random field~$h^\bullet$ on~$\Omega$ distributed according to~$\mu_{m,s}$.
\end{defs}
We omit the superscript~$\mssM$ from the notation whenever apparent from context, and write~$h^\bullet\sim \FGF{s,m}$ to denote an $m$-massive Fractional Gaussian Field with regularity~$s$. 
Here and henceforth, for random variables $X^\bullet: \omega\mapsto X^\omega$
 on $\Omega$ 
the superscript ${}^\bullet$ will indicate the $\omega$-dependence.

\medskip

The case~$h^\bullet\sim \FGF{s,m}$ with~$s=0$ is singled out in the scale of all $\FGF{}$'s on~$\mssM$ as the only one independent of~$m$.
It corresponds to the Gaussian \emph{White Noise} on~$\mssM$ induced by the nuclear rigging~$\Test\subset L^2(\vol_\g)\subset \Test'$, where we note that~$L^2(\vol_\g)=H^0_m$ for all~$m> 0$.
\begin{rem}
	The White Noise $W^\bullet$ on~$\mssM$ is the  $\Test'$-valued, centered Gaussian random field uniquely characterized by \emph{either} one of the following properties, see e.g.\ the monograph \cite{Kuo96}:
\begin{align*}
\EEE \quadre{e^{\imu \scalar\phi{W^\bullet}}}=&\ e^{-\frac12\|\phi\|^2_{L^2}} \comma & \phi\in&\Test \semicolon
\\
\EEE\Big[ \scalar\phi{W^\bullet}^2\Big]=&\ \|\phi\|^2_{L^2(\vol_\mssg)} \comma & \phi\in&\Test \semicolon
\\
\EEE\Big[ \scalar\phi{W^\bullet} \cdot \scalar\psi{W^\bullet}\Big]=&\ \int\phi\,\psi\,\dvol_\mssg \comma & \phi,\psi\in&\Test \fstop
\end{align*}
\end{rem}

\subsection{Some characterizations}
Let us now characterize the Fractional Gaussian Field $h^\bullet\sim \FGF{s,m}$ in terms of the associated Gaussian Hilbert space. 
We recall that a \emph{Gaussian Hilbert space} on~$(\Omega,\msF,\mbfP)$ is a closed linear subspace of~$L^2(\Omega)$ consisting of centered Gaussian random variables, cf.\ e.g.~\cite[Dfn.~2.5]{LodSheSunWat16}.
We say that a Gaussian Hilbert space~$\set{X_v: v\in V}$ is \emph{linearly indexed} by~$V$ if~$V$ is a linear space and~$v\mapsto X_v$ is a linear map.


\begin{prop}\label{p:GaussianHilbert}
Given~${h}^\bullet\sim\FGF{s,m}$ on~$(\Omega,\msF,\mbfP)$,
the collection
\begin{align}\label{eq:GaussianHilbert}
\mcH_{s,m}\eqdef \set{\scalar{h^\bullet}{f} : f\in H^{-s}_m}
\end{align}
(with $\scalar{h^\bullet}{f}$ suitably defined in the proof) is a Gaussian Hilbert space with covariance structure
\begin{align}\label{eq:Distribution}
\scalar{h^\bullet}{f}\sim \mcN\tparen{0,\norm{f}^2_{H^{-s}_m}} \comma \qquad f\in H^{-s}_m \fstop
\end{align}
Vice versa, every Gaussian Hilbert space
\begin{equation}\label{eq:p:GaussianHilbert:0}
\tilde\mcH_{s,m}\eqdef \set{X^\bullet_f\colon \Omega\longrar \R : f\in H^{-s}_m}
\end{equation}
on~$(\Omega,\msF,\mbfP)$ linearly indexed by~$H^{-s}_m$ and satisfying
\begin{equation}\label{eq:p:GaussianHilbert:1}
X^\bullet_f \sim \mcN\tparen{0, \norm{f}^2_{H^{-s}_m}} \comma \qquad f\in H^{-s}_m\comma
\end{equation}
is isomorphic to~$\mcH_{s,m}$ as a Hilbert space via the map~$X^\bullet_f \mapsto \scalar{h^\bullet}{f}$.
\end{prop}
The  space~$\mcH_{s,m}$ is called the Gaussian Hilbert space of~$h^\bullet\sim\FGF{s,m}$.

\begin{cor}\label{c:Consistency}
Let~$h^\bullet\sim\FGF[\M]{s,m}$ and~$\tilde\mcH_{s,m}$ be any Gaussian Hilbert space linearly indexed by~$H^{-s}_m$ defined as in~\eqref{eq:p:GaussianHilbert:0} and satisfying~\eqref{eq:p:GaussianHilbert:1}.
Further suppose that there exists a $\Test'$-valued Gaussian field~$X^\bullet$ on~$(\Omega,\msF,\mbfP)$ so that~$\scalar{X^\bullet}{\phi}=X_\phi^\bullet$ for every~$\phi\in\Test$.
Then~$X^\bullet\sim \FGF[\M]{s,m}$.
\end{cor}

\begin{rem}[Constructions with Schwartz functions]
Suppose~$\M=\R^n$ is a standard Euclidean space, and denote by~$\msS$ the space of Schwartz functions on~$\M$ endowed with its canonical Fr\'echet topology, and by~$\msS'_\sigma$ the space of tempered distributions on~$\M$ endowed with the weak topology~$\sigma(\msS',\msS)$.
Recall that~$\msS$ is a nuclear space, and embeds densely and continuously into~$H^s_m$ for every~$s\in\R$ and~$m>0$.
By the very same proof of Theorem~\ref{ex-fgf}, there exists a centered Gaussian field~$X^\bullet$ on~$\Omega=\msS'_\sigma$ with characteristic functional satisfying~\eqref{eq:ChiBetaMS} for every~$\phi\in\msS$.
By comparison with the massless case, see e.g.\ the survey~\cite{LodSheSunWat16}, the field~$X^\bullet$ too would deserve the name of \emph{massive Fractional Gaussian Field} on~$\M=\R^n$.
In fact, we have~$X^\bullet\sim\FGF[\M]{s,m}$ in our sense.

\begin{proof}
Since the identical embedding~$\Test\hookrightarrow \msS$ is continuous, the space~$\msS'_\sigma$ of tempered distributions on~$\M$ embeds identically and continuously (in particular, measurably) into~$\Test'_\sigma$.
Thus,~$X^\bullet$ is in particular~$\msD'$-valued, and it may be regarded as defined on~$\Omega=\Test'_\sigma$.
The conclusion follows in light of Corollary~\ref{c:Consistency}.
\end{proof}
\end{rem}

\begin{proof}[Proof of Proposition~\ref{p:GaussianHilbert}]
For every~$\phi\in\Test$, the map~$t\mapsto\chi_{m,s}(t\phi)$ as in~\eqref{eq:ChiBetaMS} is analytic in~$t$ around~$t=0$.
Differentiating it twice at~$t=0$ shows that the assignment~$\Test\ni\phi\mapsto\scalar{h^\bullet}{\phi}$ defines an isometry of~$\tparen{\Test,\norm{\emparg}_{H^{-s}_m}}$ into~$L^2(\Omega)$.
By density of~$\Test$ in~$H^{-s}_m$, the latter extends to a linear isometry~$H^{-s}_m\rar L^2(\Omega)$.
Thus, by construction,~$\mcH_{s,m}$ forms a closed linear subspace of~$L^2(\Omega)$.
By the definition of~$\chi_{m,s}$, the random variable~$\scalar{h^\bullet}{\phi}$ has centered Gaussian distribution with variance~$\norm{\phi}_{H^{-s}_m}^2$ for every~$\phi\in\Test$.
By the $H^{-s}_m$-continuity in~$\phi$ of the corresponding characteristic function, the latter distributional characterization extends to~$H^{-s}_m$ which yields \eqref{eq:Distribution}.

Vice versa, let~$\tilde\mcH_{s,m}$ be as in~\eqref{eq:p:GaussianHilbert:0} and~\eqref{eq:p:GaussianHilbert:1}. Since the indexing assignment~$\iota\colon f\mapsto X^\bullet_f$ is linear,~\eqref{eq:p:GaussianHilbert:1} shows that it is injective, and therefore an isomorphism of linear spaces.
Analogously,~$f\mapsto \scalar{h^\bullet}{f}$ is an isomorphism of linear spaces by~\eqref{eq:Distribution}.
Thus, the map~$X^\bullet_f\mapsto \scalar{h^\bullet}{f}$ too is an isomorphism of linear spaces, being the composition of~$\iota^{-1}\colon \tilde\mcH_{s,m}\to H^{-s}_m$ and~$\scalar{h^\bullet}{\emparg}\colon H^{-s}_m\to \mcH_{s,m}$.
Combining~\eqref{eq:Distribution} and~\eqref{eq:p:GaussianHilbert:1} shows that~$X^\bullet_f\mapsto \scalar{h^\bullet}{f}$ is additionally an $L^2(\Omega)$-isometry, which concludes the proof. 
\end{proof}

In particular, we have the following:

\begin{cor}\label{thm:cov-repr}
For~$s>0$, $h^\bullet\sim \FGF{s,m}$ is uniquely characterized as the centered Gaussian process with covariance
\begin{equation}\label{eq:Covariance}
\begin{aligned}
\Cov\tquadre{\scalar{h^\bullet}{\phi}, \scalar{h^\bullet}{\psi}}= \iint G_{s,m}(x,y)\,\phi(x) \, \psi(y) \diff\vol_\g^{\otimes 2}(x,y)
\end{aligned}
\comma \quad \phi,\psi\in\Test\subset H^{-s}_m\fstop
\end{equation}
\end{cor}

\begin{prop}\label{p:Rescaling} Let~$s\in\R$,~$m>0$, and~$h^\bigdot\sim\FGF[\mssM]{s,m}$. Then, the following assertions hold:
\begin{enumerate}[$(i)$]
\item\label{i:p:Rescaling:1}~$A_m^k h^\bullet$ is a well-defined $\msD'$-valued random field on~$(\Omega,\msF,\mbfP)$ satisfying~$A_m^k h^\bullet\sim \FGF[\M]{s-2k,m}$ for every~$k\in\Z$;
\item\label{i:p:Rescaling:2} if~$\M$ is closed, then~$A_m^{-(r-s)/2} h^\bigdot\sim\FGF[\mssM]{r,m}$ for every~$r\in\R$.
\end{enumerate}
\end{prop}
\begin{proof}
\ref{i:p:Rescaling:1}
Fix~$k\in\N$. Since~$A_m\colon \Test\to\Test$, the operator~$A_m^k\colon \Test'\to\Test'$ is well-defined on~$\Test'$ by transposition.
Thus,~$A_m^k h^\bullet$ is $\mbfP$-a.s.\ a well-defined element of~$\Test'$.
By definition of~$A_m^k\colon \Test'\to\Test'$, we have
\begin{equation}\label{eq:p:Rescaling:1}
\scalar{A_m^k h^\bullet}{\phi}= \scalar{h^\bullet}{ A_m^k \phi}\comma \qquad \phi\in\Test \fstop
\end{equation}

By Lemma~\ref{l:HilbertIsometry}, we have~$A_m^k f\in H^{-s}_m$ for every~$f\in H^{-(s-2k)}_m$. Thus, similarly to the proof of the forward implication in Proposition~\ref{p:GaussianHilbert}, the equality in~\eqref{eq:p:Rescaling:1} extends from~$\Test$ to~$H^{-(s-2k)}_m$, and
\[
\tilde\mcH_{s-2k,m}\eqdef \set{\scalar{A_m^k h^\bullet}{f}:f\in H^{-(s-2k)}_m}
\]
is a Gaussian Hilbert space on~$(\Omega,\msF,\mbfP)$ linearly indexed by~$H^{-(s-2k)}_m$.
Furthermore, we conclude again from~\eqref{eq:p:Rescaling:1} and Lemma~\ref{l:HilbertIsometry} that
\[
\scalar{A_m^k h^\bullet}{\phi}= \scalar{h^\bullet}{A_m^k\phi}\sim \mcN\paren{0,\norm{A_m^k \phi}^2_{H^{-s}_m}}= \mcN\paren{0,\norm{\phi}^2_{H^{-(s-2k)}_m}} \comma \qquad \phi\in\Test \fstop
\]
Again as in Proposition~\ref{p:GaussianHilbert}, the above equality extends from~$\Test$ to~$H^{-(s-2k)}_m$, and we conclude that~$\tilde\mcH_{s-2k,m}$ has covariance structure
\[
\scalar{A_m^k h^\bullet}{f} \sim \mcN\paren{0,\norm{f}^2_{H^{-(s-2k)}_m}} \comma \qquad f\in H^{-(s-2k)}_m \fstop
\]
By the converse implication in Proposition~\ref{p:GaussianHilbert},~$\tilde\mcH_{s-2k,m}$ is isomorphic as a Hilbert space to the Gaussian Hilbert space~$\mcH_{s-2k,m}$ of an~$\FGF[\M]{s-2k,m}$.
Thus,~$A_m^k h^\bullet\sim \FGF[\M]{s-2k,m}$ by Corollary~\ref{c:Consistency}.

\ref{i:p:Rescaling:2} Since~$\M$ is closed,~$A_m^r\colon \Test\to\Test$ for every~$r\in\R$, thus~$A_m^r\colon\Test'\to\Test'$ is well-defined by transposition.
The rest of the proof follows exactly as in~\ref{i:p:Rescaling:1} replacing~$k$ by~$(s-r)/2$.
\end{proof}

\begin{cor} The following assertions hold:
\begin{enumerate}[$(i)$]
\item all the Fractional Gaussian Fields $h_s^\bullet\sim\FGF[\mssM]{s,m}$ for $s\in\R$ and $m>0$ may be obtained from~$h_{s-2k}^\bullet\sim\FGF[\mssM]{s-2k,m}$
as
\begin{align*}
h^\bullet_s \eqdef A_m^{-2k} h_{s-2k}^\bigdot \comma
\end{align*}
where~$k$ is the only integer so that~$s-2k\in [0,2)$.

\item if~$\M$ is closed, then all the Fractional Gaussian Fields $h^\bullet\sim\FGF[\mssM]{s,m}$ for $s\in\R$ and $m>0$ may be obtained from the White Noise $W^\bigdot$ on $\M$
as
\begin{align*}
h^\bullet\eqdef \big(m^2-\tfrac12\Delta\big)^{-s/2}W^\bigdot \fstop
\end{align*}
\end{enumerate}
\end{cor}

\subsection{Continuity of the FGF} The basic property concerning differentiability and H\"older continuity of~$\FGF{}$'s is as follows.

\begin{prop}\label{p:Properties}
Let~$h^\bullet\sim\FGF[\mssM]{s,m}$. Then, the following assertions hold:
\begin{enumerate}[$(i)$]
\item\label{i:p:Properties:1} Assume that~$(\M,g)$ has bounded geometry. If~$s>n/2+\alpha$ with $\alpha\in[0,1)$, then~$h^\bullet \in \mcC^{0,\alpha}_\loc(\mssM)$ a.s.;
\item\label{i:p:Properties:2} Assume that~$(\M,g)$ is closed. If~$s>n/2+k+\alpha$ with $k\in\N_0$ and $\alpha\in[0,1)$, then~$h^\bullet \in \mcC^{k,\alpha}(\mssM)$ a.s.;
\item\label{i:p:Properties:3} If~$s>n/2+1$, then  ~$h^\bullet \in W^{1,2}_\loc(\mssM)$ a.s.
\end{enumerate}
\end{prop}
In particular, the continuity of $h^\bullet$ in the case $s>n/2$ will allow us to
 rewrite~\eqref{eq:Covariance}  in a more comprehensive and suggestive form.

\begin{cor}\label{c:Characterization}
For each~$s>n/2$ the centered Gaussian process $h^\bullet\sim \FGF{s,m}$ is uniquely characterized by
\begin{align}
\mbfE\tquadre{h^\bullet(x)\, h^\bullet(y)}=G_{s,m}(x,y) \comma \qquad x,y\in \mssM\fstop
\end{align}
\end{cor}

\begin{cor}\label{cor-noise-dist}
For each~$s>n/2$, the pseudo-distance $\rho_{s,m}$ is indeed a distance. It is given in terms of the  process $h^\bullet\sim \FGF{s,m}$ by
\begin{align}\label{eq:c:NoiseDist:0}
\rho_{s,m}(x,y)=
\mbfE\Big[\big|h^\bullet(x)- h^\bullet(y)\big|^2\Big]^{1/2} \comma \qquad x,y\in \mssM\fstop
\end{align}

\end{cor}

\begin{proof}[Proof Proposition~\ref{p:Properties}]
\ref{i:p:Properties:1}
Let~$h^\bullet\sim \FGF[\mssM]{s,m}$ with~$s>n/2$. Lemma~\ref{l:SobolevEmbedding} implies that~$H^s_m$ embeds continuously into a space of continuous functions on~$\mssM$ by Morrey's inequality.
As a consequence,~$\delta_x\in H^{-s}_m$.
Thus, Proposition~\ref{p:GaussianHilbert} implies that~$h^\omega(x)\eqdef \scalar{h^\omega}{\delta_x}$ is $\mbfP$-a.s.\ well-defined for every fixed~$x\in \mssM$.
Together with Corollary~\ref{thm:cov-repr}, this proves the representation~\eqref{eq:c:NoiseDist:0} in Corollary~\ref{cor-noise-dist}.

Combining~\eqref{eq:c:NoiseDist:0} and Theorem~\ref{t:EstimatesC} we have therefore that
\begin{align*}
\mbfE\Big[\big|h^\bullet(x)- h^\bullet(y)\big|^2\Big]^{1/2}\leq C_\alpha\cdot \mssd(x,y)^\alpha \comma \qquad x,y\in \mssM\comma
\end{align*}
for some constant~$C_\alpha>0$.
In particular,~$\omega\mapsto \tparen{h^\omega(x)-h^\omega(y)}$ is a centered Gaussian random variable with covariance dominated by~$C_\alpha\cdot\mssd(x,y)^\alpha$.
Therefore, it has finite moments of all orders~$p>1$, and, for every such~$p$, there exists a constant~$C_{\alpha,p}>0$ so that
\begin{align}\label{eq:p:Continuity:2}
\mbfE\Big[\big|h^\bullet(x)- h^\bullet(y)\big|^{p}\Big]\leq C_{\alpha,p}\cdot \mssd(x,y)^{\alpha p} \comma \qquad x,y\in \mssM\fstop
\end{align}

Since~$\M$ is smooth, there exists an atlas of charts~$(U,\Phi)$, with~$\Phi \colon U\to \Phi(U)\subset \R^n$ so that
\begin{align}\label{eq:p:Continuity:3}
C_U^{-1} \abs{\Phi(x)-\Phi(y)}\leq \mssd(x,y) \leq C_U \abs{\Phi(x)-\Phi(y)} \comma \qquad x,y\in U\comma
\end{align}
for some constant~$C_U>0$ possibly depending on~$U$.
Define a random field on~$\Phi(U)$ by setting~$h^\bullet_\Phi\eqdef h^\bullet\circ \Phi^{-1}$.
Combining~\eqref{eq:p:Continuity:3} with~\eqref{eq:p:Continuity:2},
\begin{align*}
\mbfE\Big[\big|h_\Phi^\bullet(a)- h_\Phi^\bullet(b)\big|^{p}\Big]\leq C_U\cdot C_{\alpha,p}\cdot \abs{a-b}^{\alpha p}\comma \qquad a,b\in \Phi(U)\subset \R^n \fstop
\end{align*}
By the standard Kolmogorov--Chentsov Theorem, e.g.~\cite[Thm.~I.2.1]{RevYor91}, we conclude that, for every~$\epsilon>0$ and every~$p>1$, the function~$h_\Phi^\bullet$ satisfies~$h_\Phi^\bullet\in \mcC^{0,\alpha-\eps-n/p}(\Phi(U))$ almost surely for all~$\alpha\in (0,s-n/2)$.
By arbitrariness of~$\eps$ and~$p$, and since~$\alpha$ ranges in an open interval, we may conclude that~$h_\Phi^\bullet\in \mcC^{0,\alpha}(\Phi(U))$ almost surely for all~$\alpha\in (0,s-n/2)$.
Finally, since~$\Phi$ is smooth, it follows that~$h^\bullet \in\mcC^{0,\alpha}(U)$, and therefore that~$h^\bullet \in\mcC^{0,\alpha}_\loc(\mssM)$ almost surely.


\ref{i:p:Properties:2}
Now assume that~$h^\bullet\sim \FGF[\mssM]{s,m}$ with $s>n/2+k+\alpha$ with $k\in\N$ and $\alpha\in(0,1)$. Note that~$A^{k/2}_m h^\bullet\sim \FGF[\mssM]{s-k,m}$ by Proposition~\ref{p:Rescaling}, and~$A^{-k/2}_m\colon \mcC^{0,\alpha}(\mssM)\to \mcC^{k,\alpha}(\mssM)$ for every~$k\in \N$. Thus the claim follows by the previous part~\ref{i:p:Properties:1}.

\ref{i:p:Properties:3}: Let $K$ be a bounded convex subset of $\M$ with smooth boundary, and denote $p_t^K$ the heat kernel with Neumann boundary conditions on $K$.
Recall that a function $f\in L^2(\M)$ belongs to $W^{1,2}(K)$---the form domain for the Neumann heat semigroup on $\M$---if and only if
\begin{equation}
\lim_{t\to0}\frac1t\int_K\int_K |f(x)-f(y)|^2\, p_t^K(x,y) \diff\vol_\g(x) \diff\vol_\g(y)<\infty
\end{equation} 
by the very definition of the Neumann heat semigroup on $K$.
Furthermore, the $\lim_{t\to0}$ is in fact a monotone limit.

In the case $s>n/2+1$, Theorem~\ref{t:EstimatesC} below (applied with $\alpha=1$) implies that  the continuous random function~$h^\bullet\sim \FGF[\mssM]{s,m}$ satisfies
\begin{align*}
 &\mbfE\left[\lim_{t\to0}\frac1t \,\int_K\int_K |h^\bullet(x)-h^\bullet(y)|^2\, p_t^K(x,y) \diff\vol_\g(x) \diff\vol_\g(y)\right]\\
 &=
\lim_{t\to0}\frac1t \,\int_K\int_K  \mbfE\Big[|h^\bullet(x)-h^\bullet(y)|^2\Big]\, p_t^K(x,y) \diff\vol_\g(x) \diff\vol_\g(y)\\
&\le\lim_{t\to0}\frac Ct \, \int_K\int_K \mssd(x,y)^2\, p_t^K(x,y) \diff\vol_\g(x) \diff\vol_\g(y)\le C'.
\end{align*}
where the last inequality follows from the Li--Yau estimate~\cite[Thm.~3.2]{LiYau86} on the Neumann heat kernel. Thus 
$$\lim_{t\to0}\frac1t \,\int_K\int_K |h^\omega(x)-h^\omega(y)|^2\, p_t^K(x,y) \diff\vol_\g(x) \diff\vol_\g(y)<\infty$$ for a.e.~$\omega$, which
by the preceding comment implies $h^\omega\in W^{1,2}(K)$.
By arbitrariness of $K$, the latter implies $h^\omega\in W_\loc^{1,2}(\M)$.
\end{proof}

\begin{rem}
The regularity of~$h^\bullet$ provided by Proposition~\ref{p:Properties} is sharp, in the sense that~$h^\bullet$ is \emph{not} an element of~$\mcC^{k,\gamma}$ for any~$\gamma\in [s-n/2-k,1]$.
\end{rem}

\subsection{Series Expansions in the Compact Case}
If~$\mssM$ is  closed, Fractional Gaussian Fields may be approximated by their expansion in terms of eigenfunctions of the Laplace--Beltrami operator~$\Delta$.
As before in \S\ref{ss:eig-exp}, we denote by~$\seq{\phi_j}_{j\in\N_0}\subset \Test$ the complete $L^2$-orthonormal system consisting of eigenfunctions of~$\Delta$, each with corresponding eigenvalue~$\lambda_j$, so that~$(\Delta+\lambda_j)\phi_j=0$ for every~$j$.
Recall the representations of heat kernel \eqref{eq:Heat-CompactM}, Green kernel \eqref{eq:GCompactM}, and grounded Green kernel \eqref{eq:G0-CompactM} in terms of this eigenbasis. 

\begin{figure}[htb!]
\includegraphics[scale=.60]{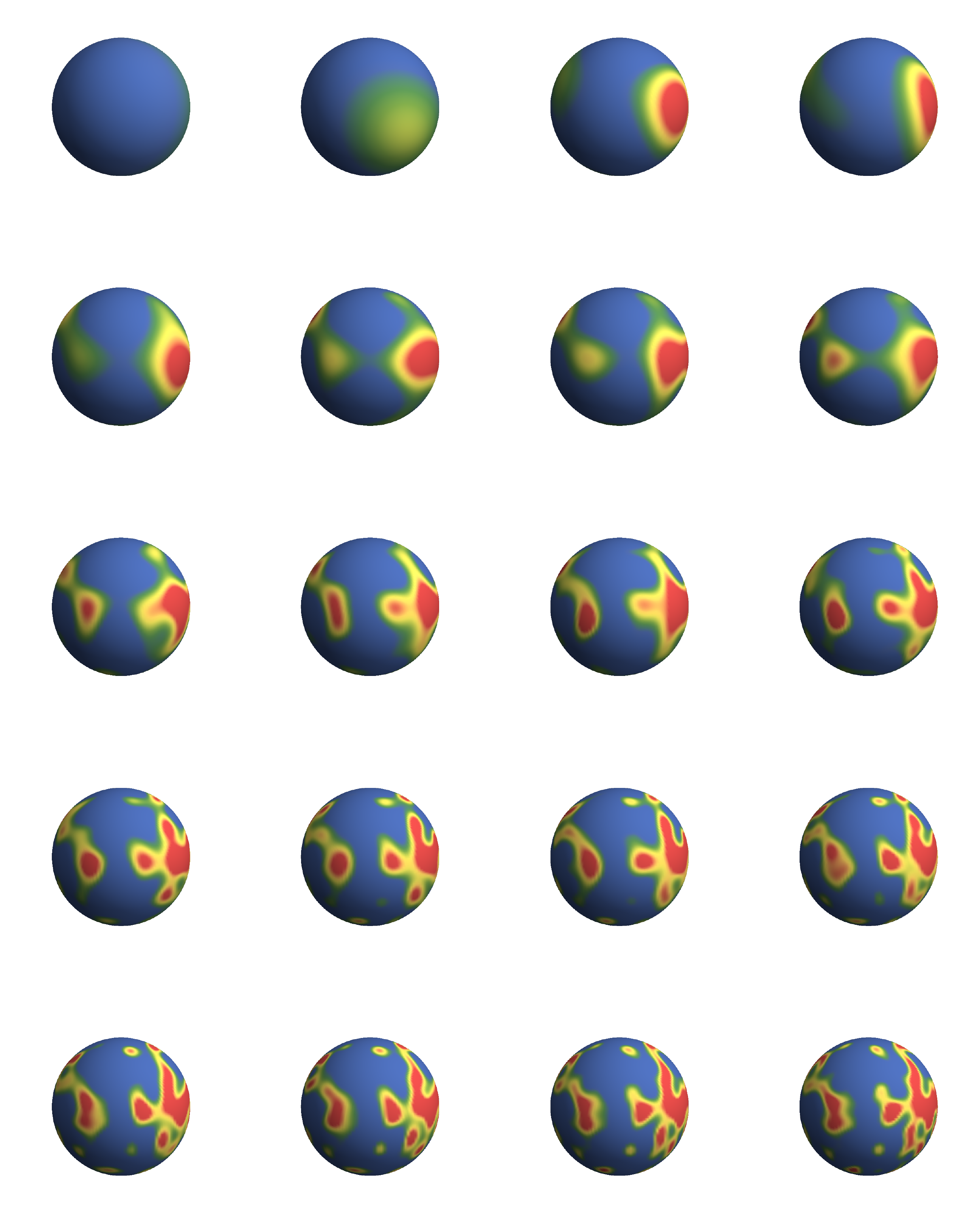}
\caption{A realization of~$h^\bullet_\ell$ in~\eqref{eq:PartialSum} on the unit sphere~$\mbbS^2$ with,~$m=s=1$ (critical case), and~$\ell\in\set{1,\dotsc, 20}$.}
\end{figure}

Let now a sequence~$\seq{\xi_j^\bullet}_{j\in \N_0}$  of i.i.d.\ random variables on a common probability space $(\Omega,\msF,\mbfP)$ be given with~$\xi_j^\bullet\sim\mcN(0,1)$. For each~$\ell>0$, define a random variable $h_\ell^\bullet: \Omega\to \Test$  by
\begin{align}\label{eq:PartialSum}
h_\ell^\omega(x)\eqdef  \sum_{j=0}^\ell \frac{\phi_j(x)\, \xi_j^\omega}{(m^2+\lambda_j/2)^{s/2}}\fstop
\end{align}

\begin{thm}\label{t:Isomorphism}
\begin{enumerate}[$(i)$]
\item\label{i:t:Isomorphism:0}
For every $s\in\R$ and~$f\in H^{-s}_m$, the family~$\seq{\scalar{h^\bullet_\ell}{f}}_{\ell\in \N}$ is a centered, $L^2$-bounded martingale on~$(\Omega,\msF,\mbfP)$.
\item\label{i:t:Isomorphism:1}
As $\ell\to\infty$,  it converges, both a.e.\ and in $L^2$, to the random variable~$\scalar{h}{f}^\bullet\in L^2(\Omega)$ given for a.e.~$\omega$ by
\begin{align*}
\scalar{h}{f}^\omega\eqdef \sum_{j\in\N_0} \frac{\scalar{\phi_j}{f} \xi_j^\omega}{(m^2+\lambda_j/2)^{s/2}}\fstop
\end{align*}
\item\label{i:t:Isomorphism:2}
$\scalar{h}{f}^\bullet$ is a centered Gaussian random variable with variance $\norm{f}_{H^{-s}_m}^{2}$.

\end{enumerate}
\end{thm}
\begin{proof} Assertion \ref{i:t:Isomorphism:0} and \ref{i:t:Isomorphism:1} follow by standard arguments on centered Gaussian variables, e.g.~\cite[Thm.~1.1.4]{Bog98}.
For \ref{i:t:Isomorphism:2}, observe that by definition, $\scalar{h}{f}^\bullet$ is a centered Gaussian random variable with variance
\begin{align}\label{eq:GaussianMartingale}
\mbfE\quadre{\tparen{\scalar{h}{f}^\bullet}^2}=\sum_{j\in \N_0} \frac{\scalar{\phi_j}{f}^2}{(m^2+\lambda_j/2)^s}= \tnorm{A_m^{-s/2} f }^2_2=\norm{f}_{H^{-s}_m}^2 \comma
\end{align}
where the first equality holds by orthogonality of~$\seq{\phi_j}_{j\in \N_0}$ and since~$\seq{\xi_j^\bullet}_{j\in\N_0}$ are i.i.d.~$\sim \mcN(0,1)$, the second equality since~$\seq{\phi_j}_{j\in \N_0}$ is a complete $L^2$-orthonormal system of eigenfunctions of~$A_m$ as well, and the third equality by the definition of the norm of~$H^{-s}_m$.
\end{proof}

\begin{cor}\label{c:GHSRevisited}
The family of random variables
\begin{align*}
\tilde\mcH_{s,m}\eqdef\set{\scalar{h}{f}^\bullet: f\in H^{-s}_m}\comma \qquad s\in\R\comma m>0\comma
\end{align*}
is a Gaussian Hilbert space, isomorphic to~$\mcH_{s,m}$ in~\eqref{eq:GaussianHilbert} via the map~$\iota\colon\scalar{h}{f}^\bullet\mapsto \scalar{h^\bullet}{f}$.

\begin{proof}
It is shown in Theorem~\ref{t:Isomorphism}\iref{i:t:Isomorphism:2} that~$\tilde\mcH_{s,m}$ is a Gaussian linear space, closed in~$L^2(\Omega)$ by completeness of~$H^{-s}_m$ and~\eqref{eq:GaussianMartingale}, and thus a Gaussian Hilbert space.
Since~$\Test$ embeds continuously into~$H^{-s}_m$ for every~$s\in\R$, the map~$\iota\colon \scalar{h}{\phi}^\bullet\mapsto \scalar{h^\bullet}{\phi}$ is well-defined for every~$\phi\in\Test$.
Equation~\eqref{eq:Distribution} together with Theorem~\ref{t:Isomorphism}\iref{i:t:Isomorphism:2} show that it is as well an isometry, and thus extends to~$\tilde\mcH_{s,m}$ by density of~$\Test$ in~$H^{-s}_m$ and~\eqref{eq:GaussianMartingale}, again for every~$s\in\R$.
Since~$\set{\scalar{h^\bullet}{\phi}:\phi\in\Test}$ is dense in~$\mcH_{s,m}$ by construction, as in the proof of Proposition~\ref{p:GaussianHilbert}, the map~$\iota$ has dense image.
Since isometries of Hilbert spaces have closed range, it is as well surjective, and thus an isomorphism of (Gaussian) Hilbert spaces.
\end{proof}
\end{cor}

\begin{thm}\label{fgf-ptw} For~$s>n/2$, the series
\begin{align*}
h^\omega(x)\eqdef 
 \sum_{j\in \N_0} \frac{\phi_j(x) \, \xi_j^\omega}{(m^2+\lambda_j/2)^{s/2}}
\end{align*}
converges in $L^2(\Omega)$ and almost surely on~$\Omega$ for each $x\in\mssM$. Moreover it converges on $L^2(\Omega\times\mssM)$ and in $L^2(\vol_\g)$ almost surely.
\end{thm}

\begin{proof}
The $L^2(\Omega\times \mssM)$ as well as the $L^2(\Omega)$ convergence follow by
combining the identities
\begin{align*}
\mbfE\bigg[\int\bigg( \sum_{j=\ell+1}^{\ell'} \frac{\phi_j(x) \, \xi_j^\omega}{(m^2+\lambda_j/2)^{s/2}}\bigg)^2\diff\vol_\g\bigg]=\sum_{j=\ell+1}^{\ell'}\frac1{(m^2+\lambda_j/2)^{s}}
\comma\\
\mbfE\bigg[\bigg( \sum_{j=\ell+1}^{\ell'} \frac{\phi_j(x) \, \xi_j^\omega}{(m^2+\lambda_j/2)^{s/2}}\bigg)^2\bigg]=\sum_{j=\ell+1}^{\ell'}\frac{\phi_j(x)^2}{(m^2+\lambda_j/2)^{s}}
\comma
\end{align*}
and the fact that the terms on the right hand side of both equations converge to 0
as $\ell,\ell'\to\infty$ according to Weyl's asymptotics \eqref{eq: weyl} and \eqref{eq:GCompactM} respectively. The almost sure convergence for each $x$ as well as the almost sure convergence for the $L^2(\vol_\g)$ sequence follow by Theorem \ref{t:Isomorphism} and Doob's Martingale Convergence Theorem. 
\end{proof}

\subsection{The Grounded FGF}
Assume now that   $\mssM$ is closed.
Then, the same arguments used to derive Theorem~\ref{ex-fgf} also apply for the grounded norms, and in this case even for $m\ge0$.

In order to state the next result, let us set~$\mathring{\Test}\eqdef \set{\psi\in\Test:\scalar{\vol_\g}{\psi}=0}$, and denote by~$\mathring\Test'$ the topological dual of~$\mathring\Test$.
We note that~$\mathring\Test$ is a nuclear space when endowed with the subspace topology inherited from~$\Test$, since every linear subspace of a nuclear space is itself nuclear, e.g.~\cite[Prop.~50.1, (50.3), p.~514]{Tre67}.

\begin{thm} For~$m\ge 0$ and $s\in\R$, 
there exists  a unique Radon Gaussian measure~$\mathring\mu_{m,s}$ on~$\mathring{\Test'}$ with 
characteristic functional given by
\begin{align}\label{eq:ChiBetaMS-grdd}
\mathring\chi_{m,s}\colon \phi\longmapsto \exp\quadre{-\tfrac{1}{2}\norm{\phi}_{\mathring{H}^{-s}_m}^2} \comma\qquad \phi\in\mathring{\Test}\fstop
\end{align}

\begin{proof}
Analogously to Theorem~\ref{ex-fgf}, it suffices to show that~$\mathring{\Test}$ embeds continuously into~$\mathring{H}^{-s}_m$.
In turn, this follows from the continuity of the embedding of~$\Test$ into~$H^s_m$ and Lemma~\ref{l:EquivalenceNorms}\iref{i:l:EquivalenceNorms:2}.
\end{proof}
\end{thm}

\begin{defs} Let~$m\ge0$ and~$s\in\R$.
A  \emph{grounded $m$-massive Fractional Gaussian Field on~$\mssM$ with regularity~$s$}, in short:~$\gFGF[\mssM]{s,m}$, is any $\Test'$-valued random field~$h^\bullet$ on~$\Omega$ distributed according to~$\mathring\mu_{m,s}$.
In the case $m=0$, the field is called a \emph{grounded massless Fractional Gaussian Field on~$\mssM$ with regularity~$s$}.
\end{defs}

All results for the random fields ${\sf FGF}_{s,m}$ have their natural counterparts for  $\gFGF{s,m}$, now even admitting $m=0$. In particular, we have the grounded versions of
Corollary~\ref{thm:cov-repr} and Theorem~\ref{fgf-ptw}.

\begin{cor} For~$s>0$ and $m\ge0$, the random field
$h^\bullet\sim \gFGF{s,m}$ is uniquely characterized as the centered  Gaussian process with covariance
\begin{equation*}
\begin{aligned}
\Cov\tquadre{\scalar{h^\bullet}{\phi}, \scalar{h^\bullet}{\psi}}= \ \iint \mathring G_{s,m}(x,y)\,\phi(x) \, \psi(y) \diff\vol_\g^{\otimes 2}(x,y)
\end{aligned}
\comma \qquad \phi,\psi\in\mathring{\Test}\subset \mathring{H}^{-s}_m\fstop
\end{equation*}
\end{cor}

\begin{cor} For~$s>n/2$ and $m\ge0$, the series
\begin{align*}
{h}^\omega(x)\eqdef 
 \sum_{j\in \N} \frac{\phi_j(x) \, \xi_j^\omega}{(m^2+\lambda_j/2)^{s/2}}
\end{align*}
converges in $L^2(\Omega)$ and almost surely on $\Omega$ for each $x\in\mssM$. Moreover it converges on $L^2(\Omega\times\mssM)$ and in $L^2(\vol_\g)$ almost surely.
\end{cor}
In particular, ${h}^\bullet\sim \mathring{\sf FGF}_{s,0}$ is given by
${h}^\omega(x)= 
2^{s/2}\, \sum_{j\in \N}\lambda_j^{-s/2}\, \phi_j(x) \, \xi_j^\omega$ if~$s>n/2$.

If $m>0$, the grounding map $f\mapsto \mathring f\eqdef f-\frac1{\vol_\g(\mssM)}\scalar{f}{\car}$ allows us to easily switch between the random fields
 $\FGF[\mssM]{s,m}$ and $\gFGF[\mssM]{s,m}$, as in the next Lemma.

\begin{lem}\label{l:gFGFtoFGF}
For every $s\in\R$ and every $m>0$,
\begin{enumerate}[$(i)$]
\item given $h^\bullet\sim \FGF{s,m}$, put $\mathring{h}^\omega\eqdef h^\omega-\frac{1}{\vol_\g(\mssM)}\scalar{h^\omega}{\car}$. Then~$\mathring{h}^\bullet\sim \gFGF{s,m}$;

\item\label{i:l:gFGFtoFGF:2} given ~${h}^\bullet\sim \gFGF{s,m}$ and independent 
$\xi\sim \mcN(0,1)$, put
$\hat h^\omega\eqdef h^\omega+ \frac1{\sqrt{m^{2s}\,\vol_\g(\mssM)}}\xi^\omega\,\car$.
Then $\hat h^\bullet\sim \FGF{s,m}$. 
\end{enumerate}
\end{lem}

\begin{prop}\label{cont-m0}
Let~$\mathring h^\bullet\sim\gFGF{s,m}$ on~$\mssM$.
If~$s>n/2+k+\alpha$ with $k\in\N_0$ and $\alpha\in[0,1)$, then~$\mathring h^\bullet \in \mcC^{k,\alpha}_\loc(\mssM)$ almost surely.
\begin{proof}
Let~$\xi\sim \mcN(0,1)$ be independent of~$\mathring h^\bullet$.
By Lemma~\ref{l:gFGFtoFGF}\iref{i:l:gFGFtoFGF:2}, $\mathring h^\bullet + \tfrac{1}{\sqrt{m^{2s}\vol_\mssg(\mssM)}}\xi^\bullet \car$ is distributed as an~$\FGF[\mssM]{s,m}$, and thus it satisfies Proposition~\ref{p:Properties}.
Since~$\tfrac{1}{\sqrt{m^{2s}\vol_\mssg(\mssM)}}\xi^\omega \car\in\Test$ for every~$\omega$, the conclusion follows.
\end{proof}
\end{prop}

\begin{rem}
It is worth comparing the grounding of operators and fields presented above with the pinning for fractional Brownian motions in~\cite{Gel14}, where a \emph{Riesz field}~$R^s$ is defined as the centered Gaussian field with covariance
\begin{align*}
\mbfE\quadre{R^s(x)\, R^s(y)}=\frac{1}{\Gamma(s)}\int_0^\infty t^{s-1} \tparen{p_t(x,y)-p_t(x,o)-p_t(y,o)+p_t(o,o)}\diff t\comma \qquad s\in (n/2,n/2+1)\comma
\end{align*}
for some fixed `origin'~$o\in \M$.
In particular, while grounding on a compact manifold~$(\M,\g)$ is canonical, the pinning of a Riesz field at~$o\in \M$, and hence the properties of the corresponding random Riemannian manifold (see~\S\ref{ss:RRG} below), would depend on~$o$.
\end{rem}

\subsection{Dudley's Estimate}
A crucial role in our geometric estimates and functional inequalities for the Random Riemannian Geometry is played by estimates for the expected maximum of the random field. The fundamental estimate of Dudley provides an estimate in terms of the covering number w.r.t.~the pseudo-distance $\rho_{s,m}$, introduced in \eqref{green-dist}.

\begin{notat}\label{notation:Covering}
For any pseudo-distance~$\rho$ on~$\mssM$, we denote by~$N_\rho(\eps)$ the least number of $\rho$-balls of radius~$\eps$ which are needed to cover $\M$.
When~$\rho=\rho_{s,m}$ we write~$N_{s,m}(\eps)$ in place of~$N_{\rho_{s,m}}(\eps)$.
\end{notat}

\begin{thm}[{\cite[Thm.~11.17]{LedTal06}}] \label{dudley} Fix $s>n/2$ and~$m\ge0$ 
Then, for $h\sim \FGF[\mssM]{s,m}$ (and in the compact case also for $h\sim \gFGF[\mssM]{s,m}$),
\begin{align*}
\EEE\bigg[\sup_{x\in\M}h^\bullet(x)\bigg]\le 24\cdot \int_0^\infty\Big(\log N_{s,m}(\eps)\Big)^{1/2}\diff\eps\fstop
\end{align*}
\end{thm}

In Section 6 we will study in detail the asymptotics of the Green kernel close to the diagonal and in particular derive sharp estimates for the noise distance $\rho$ in terms of the Riemannian distance $\mssd$. This will lead to sharp estimates for the covering numbers $N_{s,m}(\eps)$ and thus in turn to sharp estimates for the expected  maximum of the random field.

\section{Random Riemannian Geometry}\label{ss:RRG}
Let a Riemannian manifold $(\mssM,\g)$ be given together with a Fractional Gaussian Field~$h^\bullet\sim \FGF[\mssM]{s,m}$ with~$s>n/2$ and $m>0$. 
If $\mssM$ is compact, we alternatively can choose~$h^\bullet\sim \gFGF[\mssM]{s,m}$ with~$s>n/2$ and $m\ge 0$.
In the sequel, we assume that either $\mssM$ is closed or $m>0$ and~$(\mssM,\g)$ has bounded geometry. 

For almost every~$\omega\in\Omega$, by Propositions~\ref{p:Properties} and \ref{cont-m0},~$h^\omega$ is  a continuous function on~$\mssM$.
For each such~$\omega$, we consider the Riemannian manifold
\begin{align}
(\mssM,\g^\omega)\quad\text{with}\quad \g^\omega\eqdef e^{2h^\omega}\, \g\comma
\end{align}
the new metric being the conformal change of the metric~$\g$ by the conformal factor~$h^\omega$. 
In other words, we consider the \emph{random Riemannian manifold}
\begin{align}
\mssM^\bullet\eqdef\, (\mssM,\g^\bullet)\quad\text{with}\quad \g^\bullet\eqdef e^{2h^\bullet}\, \g
\end{align}
with the \emph{random Riemannian metric }~$\g^\bullet\colon\omega\mapsto \g^\omega$.

Assuming that~$\mssM$ is closed, for a.e.~$\omega$, the Riemannian metric $\g^\omega$ is of class~$\mcC^k$ on~$\mssM$  for~$k\eqdef \ceiling{s-n/2}-1\ge0$, 
where we set $\ceiling{a}\eqdef \min (\Z\cap[a,\infty))$.
In particular, for~$s> n/2+2$, it is almost surely of class $\mcC^2$, and the Riemannian manifolds~$\mssM^\omega$ may be studied by smooth techniques. 
Our main interest in the sequel will be in the case~$s\in (n/2, n/2+2]$ where 
 no such techniques are directly applicable and where we have no classical curvature concepts at our disposal.

\subsection{Random Dirichlet Forms and Random Brownian Motions} 

Our approach to geometry, spectral analysis, and stochastic calculus on the randomly perturbed Riemannian manifolds $(\mssM,\g^\bullet)$ will be based on Dirichlet-form techniques.
Before going into details, 
let us recall some standard results on the canonical Dirichlet form on the `un-perturbed' Riemannian manifold. 
\begin{rem}
The \emph{canonical Dirichlet form} on the Riemannian manifold $(\mssM,\g)$, e.g.~\cite[\S5.1, p.~148]{Dav89}, is the closed bilinear form~$(\mcE,\mcF)$ on~$L^2(\vol_\g)$ given by $\mcF\eqdef W^{1,2}_\ast$ and
\begin{align}\label{class-dir}
\mcE(\phi,\psi)\eqdef \frac{1}{2}\int \scalar{\diff \phi}{\diff \psi}_{\g_*}\,\dvol_\g=\frac{1}{2}\int \scalar{\nabla \phi}{\nabla \psi}_{\g}\,\dvol_\g\fstop
\end{align}
Here $\g_*$ denotes the inverse metric tensor obtained from $\g$ by musical isomorphism, $\diff$  the differential on~$\mssM$, and $\nabla$ the gradient; for functions in $W^{1,2}_\ast$, differentials and gradients have to be understood in the weak sense.
In fact, however, $\mcC^\infty_c$ is dense in the form domain $\mcF$ and thus in \eqref{class-dir} we can restrict ourselves to $\phi,\psi\in \mcC^\infty_c$.

The form~$(\mcE,\mcF)$ is a regular, strongly local, conservative Dirichlet form properly associated with the standard Brownian motion~$\mbfB$ on~$(\mssM,\g)$, the Markov diffusion process with transition kernel~$p_t$ introduced in~\S\ref{s:Setting}. 
\end{rem}
The canonical Dirichlet form and the Laplace--Beltrami operator on $(\mssM,\g)$ uniquely determine each other by
\begin{align*}
\mcE(\phi,\psi)=-\frac{1}{2}\int \Delta\phi\, \psi\,\dvol_\g
\comma \qquad \phi,\psi\in \mcC^\infty_c \fstop
\end{align*}
Under conformal transformations with non-differentiable weights, however, the latter no longer admits a closed expression whereas the  former still is easily representable.

\begin{rem} If $\g'=e^{2f}\g$ is a conformal change of the metric $\g$ by means of a smooth weight $f$, then $\g_*'=e^{-2f}\g_*$, $\vol_\g'=e^{nf}\vol_\g$, and $\nabla'\phi=e^{-2f}\nabla\phi$. Thus in particular,
\begin{align*}
\mcE'(\phi,\psi)\eqdef \frac{1}{2}\int \scalar{\diff \phi}{\diff \psi}_{\g_*}\,e^{(n-2)f}\,\dvol_\g=\frac{1}{2}\int \scalar{\nabla \phi}{\nabla \psi}_\g\,e^{(n-2)f}\,\dvol_\g
\comma
\end{align*}
and $\Delta'\phi=e^{-2f}\tparen{\Delta\phi+(n-2) \scalar{\nabla f}{\nabla \phi}_\g}$.
\end{rem}

Now let us turn to the randomly perturbed Riemannian manifolds $(\mssM,\g^\bullet)$.

\begin{thm}\label{t:Main}
Let~$h^\bullet\sim \FGF{s,m}$ with  $m>0$ and~$s>n/2$. 
Then,
\begin{enumerate}[$(a)$]
\item\label{i:t:Main:1} for~$\mbfP$-a.e.~$\omega\in\Omega$, the quadratic form~$(\mcE^\omega,\mcC^\infty_c)$
\begin{align}\label{eq:RandomForm}
\mcE^\omega(\phi,\psi)= \frac{1}{2}\int  \scalar{\nabla \phi}{\nabla \psi}_\g\,
 e^{(n-2)h^\omega}\diff\vol_\g\comma \qquad \phi,\psi\in \mcC^\infty_c \subset L^2(e^{nh^\omega}\vol_\g)\comma
\end{align}
is closable on~$L^2(e^{nh^\omega}\vol_\g)$;

\item\label{i:t:Main:2} its closure~$(\mcE^\omega,\dom^\omega)$ is a regular, irreducible, strongly local  Dirichlet form,
properly associated with an $e^{nh^\omega}\vol_\g$-symmetric Markov diffusion process $\mbfB^{\omega}$ on $\mssM$;
  
\item\label{i:t:Main:3} the generator of the closed bilinear form $(\mcE^\omega,\dom^\omega)$, denoted by $\Delta^\omega$, is the unique self-adjoint operator on~$L^2(e^{nh^\omega}\vol_\g)$ with $\mcD(\Delta^\omega)\subset \dom^\omega$ and
\begin{align}
\mcE^\omega(\phi,\psi)=-\frac{1}{2}\int (\Delta^\omega\phi)\, \psi\, e^{nh^\omega}\dvol_\g
\comma \qquad \phi\in\mcD(\Delta^\omega)\comma \psi\in \dom^\omega \semicolon
\end{align}

\item\label{i:t:Main:4} the associated intrinsic distance
\[
\mssd_{\mcE^\omega}(x,y)\eqdef {\sup}\set{\abs{f(x)-f(y)}: f\in\dom^\omega\cap \mcC^0(\mssM)\comma { \abs{\nabla f}^2\leq e^{-nh^\omega}} \quad \vol_\g\text{-a.e.}}
\]
coincides with the \emph{Riemannian distance} $\mssd^\omega$ on $\mssM$ given by
 \begin{align} 
 \mssd^\omega(x,y)\eqdef\inf \set{ \int_0^1e^{h^\omega(\gamma_r)}\, \sqrt{\g(\dot\gamma_r,\dot\gamma_r)}\diff r: \ \gamma\in\mathcal{AC}\tparen{[0,1];\mssM}\comma \gamma_0=x\comma\gamma_1=y } \fstop
 \end{align}
 \end{enumerate}
\end{thm}

\begin{proof} \iref{i:t:Main:1} Let $\omega$ be given such that $h^\omega$ is continuous.
Then both~$\sigma\eqdef e^{nh^\omega}$  and~$\rho\eqdef e^{(n-2)h^\omega}$ are positive and in $L^1_{\rm loc}$ and so is $1/\rho$. In particular, the weights thus satisfy the so-called Hamza condition.
A proof of closability under this condition, in the case~$\mssM=\R^n$, is given in~\cite[\S{II.2(a)}]{MaRoe92}, and, for general manifolds in the case~$U=\mssM$ and~$\sigma\equiv 1$, in~\cite[Thm.~4.2]{AlbBraRoe89}. The general case readily follows. 

\iref{i:t:Main:2}+\iref{i:t:Main:3} For the Markov property, see e.g.~\cite[Example~1.2.1 and Thm.~3.1.1]{FukOshTak11}, for the strong locality and the regularity see e.g.~\cite[Exercise~3.1.1]{FukOshTak11}.
Since the local domain~$\dom^\omega_\loc$ coincides with the local domain~$\dom$,
the irreducibility follows from~\cite[Thm.~4.5]{BarCheMur20}.
The assertions on the associated Markov process and on the generator easily follow.

\iref{i:t:Main:4} Choosing $\omega$ such that $h^\omega$ is continuous, the claim follows from~\cite[Lem.~3.5]{HanStu21}.
\end{proof}

\begin{defs} 
\begin{enumerate}[$(a)$, wide]
\item The operator $\Delta^\omega$ is called the \emph{Laplace--Beltrami} or \emph{Laplace operator} on $\mssM^\omega$.

\item The family of operators $\big( e^{t\Delta^\omega/2}\big)_{t>0}$ on $L^2(e^{nh^\omega}\vol_\g)$ is called the \emph{heat semigroup} on $\mssM^\omega$.

\item The process $\mbfB^{\omega}$ is called \emph{Brownian motion} on $\mssM^\omega$.

\item A function $\phi$ on an open subset $U\subset \mssM^\omega$ is called \emph{weakly harmonic} if $\phi\in W^{1,2}_\loc(U)$ and $\mcE^\omega(\phi,\psi)=0$ for all $\psi\in\mcC^\infty_c$ with $\supp(\psi)\subset U$.
\end{enumerate}
\end{defs}

\begin{thm}\label{t:Main-Hold}
Let~$s>n/2$,~$m>0$, and~$h^\bullet\sim \FGF{s,m}$. 
Then, for~$\mbfP$-a.e.~$\omega\in\Omega$, the following assertions hold:
\begin{enumerate}[$(i)$]
\item\label{i:t:Hold:1} 
every weakly harmonic function on $U\subset \mssM^\omega$ admits a version which is  locally H\"older continuous (w.r.t.~$\mssd$ and, equivalently, w.r.t.~$\mssd^\omega$);

\item\label{i:t:Hold:2} 
the heat semigroup $\big(e^{t\Delta^\omega/2}\big)_{t>0}$ on $\mssM^\omega$ has an integral kernel $p^\omega_t(x,y)$ which is jointly locally H\"older continuous in $t,x,y$;

\item\label{i:t:Hold:3} 
for every starting point, the distribution of the Brownian motion on~$\mssM^\omega$ is uniquely defined. 
\item\label{i:t:Hold:4} For all $x,y\in\M$,
\begin{align*}
\lim_{t\to0} 2t\, \log p_t^\omega(x,y)=-\mssd^\omega(x,y)^2\fstop
\end{align*}
\end{enumerate}
\end{thm}

\begin{proof} Let $\omega$ be given such that $h^\omega$ is continuous. Then, locally on $\mssM$, the Dirichlet forms~$\mcE^\omega$ and~$\mcE$ as well as the measures~$\vol_\g^\omega\eqdef e^{nh^\omega}\vol_\g$ and~$\vol_\g$ are comparable.
In other words, the `Riemannian structure' for~$\g^\omega$ is locally uniformly elliptic w.r.t.\ the structure for~$\g$ in the sense of~\cite{SaC92}.
Thus, assertion~\iref{i:t:Hold:1}, resp.~\iref{i:t:Hold:2}, follows from either~\cite[Cor.~5.5]{SaC92} or~\cite[Cor.~3.3, resp.\ Prop.~3.1 and Thm.~3.5]{Stu96}. 

If~$\M$ is compact, assertion~\iref{i:t:Hold:3} is a consequence of~\iref{i:t:Hold:2}. For general $\M$, we will choose an exhaustion of~$\M$ by  relatively compact, open sets $B_n\nearrow\M$ which are regular for~$\mcE^\omega$. For 
instance, according to Wiener's criterion, we can choose the open balls $B_n:=B_n(o)$, $n\in\N$, around any fixed point $o\in\M$.
 Let~$\mcE^{n,\omega}$ denote the Dirichlet form obtained from~$\mcE^\omega$ by imposing Dirichlet boundary conditions on $\M\setminus B_n$, and let $G^{n,\omega}_{1,m}(x,y)$ denote the associated resolvent kernel. Then for any fixed $x\in B_n$ the latter kernel is continuous in $y\in B_n$ (as a consequence  of~\iref{i:t:Hold:2}) and it vanishes as $y$ approaches $\partial B_n$ (due to the regularity of $\partial B_n$).
 Thus $\left(G^{n,\omega}_{1,m}\right)_{m>0}$  extends to a Feller resolvent on the compact space $\overline{B_n}$. The associated Feller process $\mbfB^{n,\omega}$ is pointwise well-defined. 
It will be called Random Brownian Motion with absorption on $\M\setminus B_n$. For any given $k,\ell\in\N$ with $k,\ell\ge n$, the processes $\mbfB^{k,\omega}$ and $\mbfB^{\ell,\omega}$ can be modelled on the same probability space and such that their trajectories coincide until the first hitting time of  $\M\setminus B_n$. With a diagonal argument we then construct the process $\mbfB^{\omega}$ as follows:
 if it starts in $B_n\setminus B_{n-1,\omega}$, it follows the trajectories of the process $\mbfB^{n+1,\omega}$ until it hits $\partial B_n$. Then it follows the trajectories of  $\mbfB^{n+2,\omega}$  etc.
This yields a pointwise well-defined process. By monotonicity of resolvent kernels and Dirichlet forms, it is associated with the monotone increasing limit of Dirichlet forms~$\mcE^{\omega}=\lim_{n\nearrow\infty}\mcE^{n,\omega}$.

Assertion~\iref{i:t:Hold:4} follows from the main result in~\cite{Nor97}.
\end{proof}

\begin{figure}[htb!]
     \centering
     \begin{subfigure}[b]{0.3\textwidth}
         \centering
         \includegraphics[width=\textwidth]{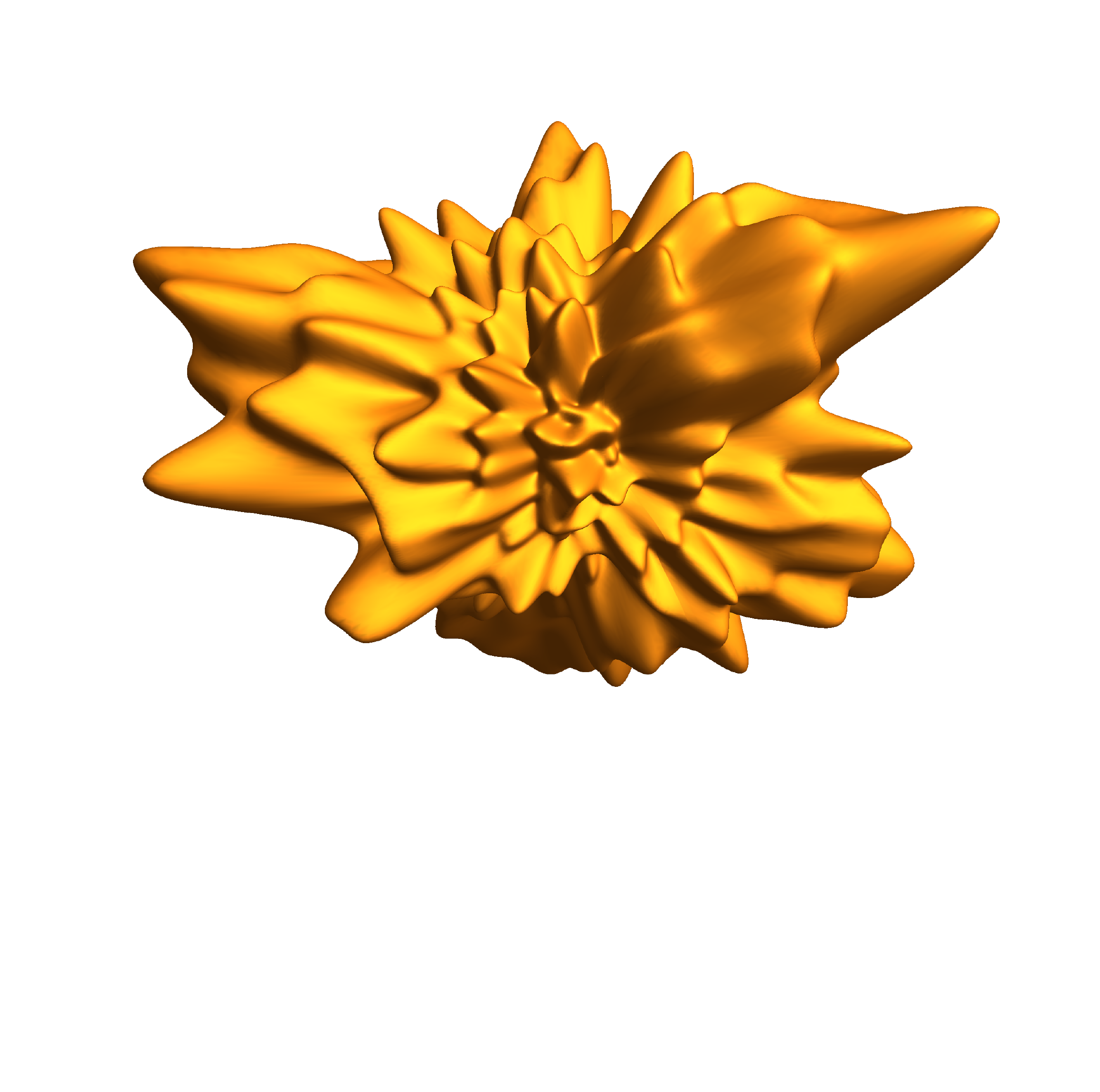}
         \caption{$s=3/2$}
         \label{fig:y equals x}
     \end{subfigure}
     \hfill
     \begin{subfigure}[b]{0.3\textwidth}
         \centering
         \includegraphics[width=\textwidth]{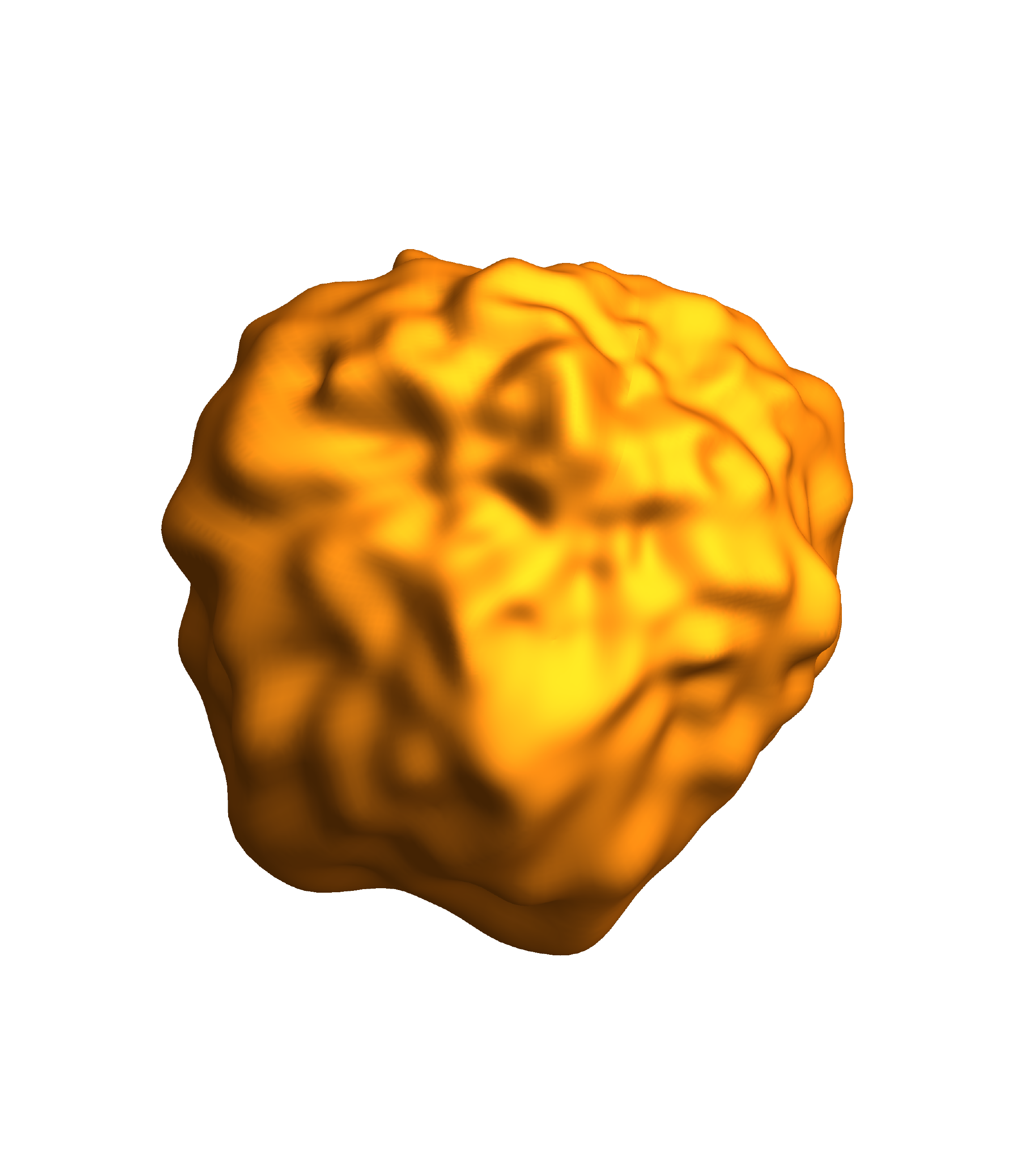}
         \caption{$s=2$}
         \label{fig:three sin x}
     \end{subfigure}
     \hfill
     \begin{subfigure}[b]{0.3\textwidth}
         \centering
        \includegraphics[width=\textwidth]{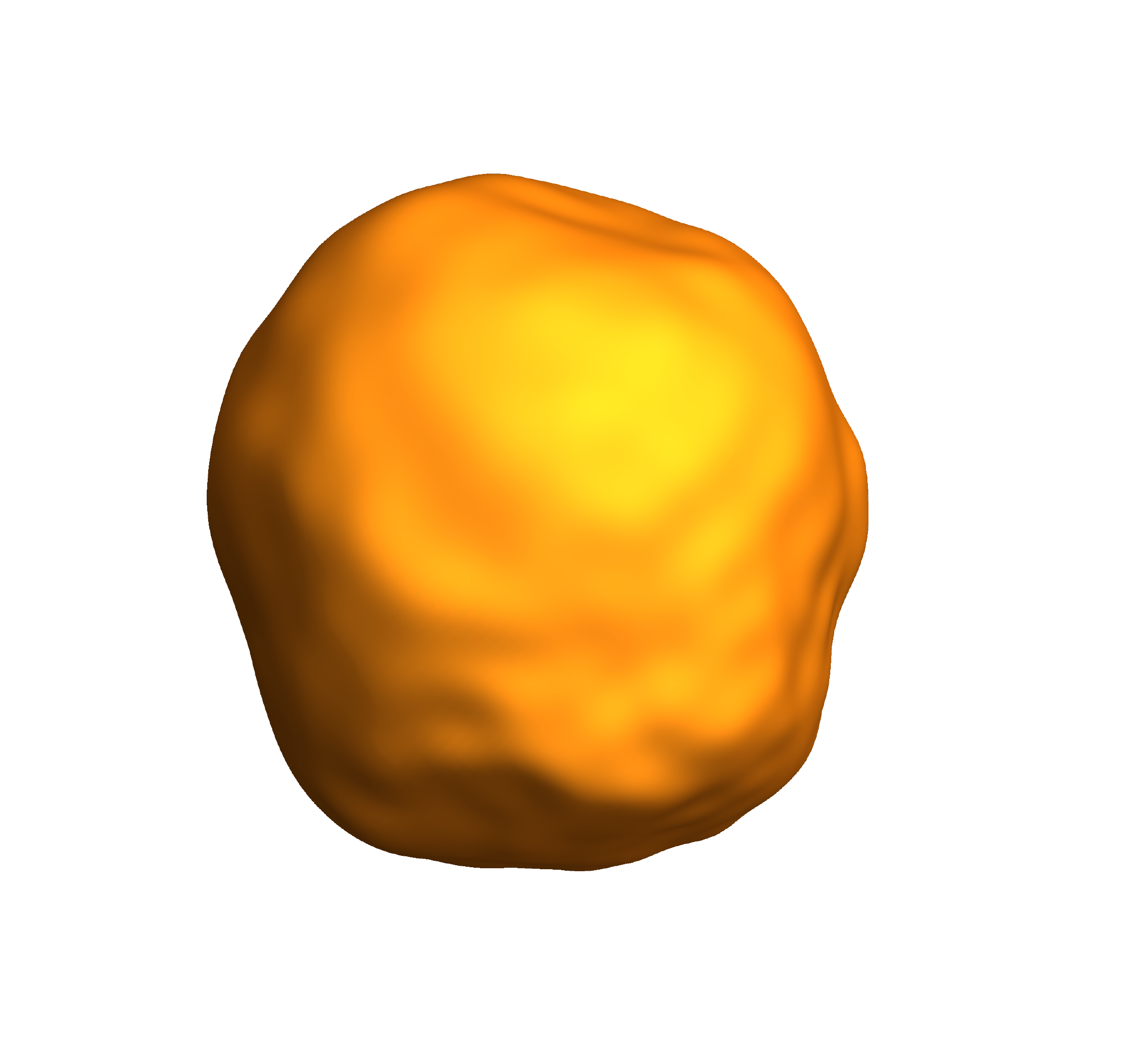}
         \caption{$s=5/2$}
         \label{fig:five over x}
     \end{subfigure}
        \caption{A realization of the random metric~$\g_\ell^\bullet=e^{2h^\bullet_\ell}\g$ on~$\mbbS^2$,~$\ell=30$.}
        \label{fig:three graphs}
\end{figure}

\subsection{Random Brownian Motions in the $\mathcal C^1$-Case} 

More precise insights into the analytic and probabilistic structures on the random Riemannian manifold $(\mssM,\g^\bullet)$ can be gained if the regularity parameter~$s$ is larger than $n/2+1$.
In this case, the conformal weight $h^\bullet$ is a.s.\ a $\mathcal C^1$-function.

To provide an explicit representation for the perturbed Brownian motion, we need some notations and concepts from the abstract theory of Dirichlet forms.

\paragraph{Martingale additive functionals} Denote the Brownian motion on the (`unperturbed') Riemannian manifold $(\mssM,\g)$ by
\begin{align*}
\mbfB\eqdef\paren{\Xi, \seq{\msF_t}_{t\geq 0}, \seq{X_t}_{t\geq 0}, \seq{P_x}_{x\in \mssM}}\fstop
\end{align*}

\begin{lem}[`Fukushima decomposition', see {\cite[\S{6.3}]{FukOshTak11}}]\label{l:FukushimaD}
\begin{enumerate}[$(a)$]
\item\label{i:l:FukushimaD:1}
 For each continuous $\psi\in W^{1,2}_\ast$, there exist a unique 
 martingale additive functional  $M^{[\psi]}$  and a unique continuous additive functional  $N^{[\psi]}$ which is of zero energy such that
 \begin{equation}\label{Fuku}
 \psi(X_t)=\psi(X_0)+M^{[\psi]}_t+N^{[\psi]}_t \qquad t\in[0,\zeta)\quad \quad \text{$P_x$-a.s.~for q.e.~$x\in \mssM$} \fstop
\end{equation}
The quadratic variation of $M^{[\psi]}$ is given by
\begin{equation}\label{fuku-qv}
\langle M^{[\psi]}\rangle_t = \int_0^t \tabs{\nabla \psi(X_s)}_\g^2\diff s \qquad t\in[0,\zeta)\quad \quad \text{$P_x$-a.s.~for q.e.~$x\in \mssM$}
\end{equation}
for any choice of a Borel version of the function $|\nabla \psi|_\g\in L^2(\mssM)$.

\item\label{i:l:FukushimaD:2}  For each  continuous $\psi\in W^{1,2}_{\loc}$, there exists a unique local martingale additive functional  $M^{[\psi]}=\tseq{M^{[\psi]}_t}_{t\in [0,\zeta)}$  such that 
\begin{equation*}
M^{[\psi]}_t = M^{{[\psi_n]}}_t\qquad t\in[0,\tau_n) \quad \quad \text{$P_x$-a.s.~for q.e.~$x\in \mssM$}
\end{equation*}
 where, for every $n\in\N$, we let~$M^{{[\psi_n]}}$ be the martingale additive functional associated with a function $\psi_n\in W^{1,2}_\ast$ such that $\psi=\psi_n$ a.e.~on $\mssM_n$, for some exhausting sequence of relatively compact open sets $\mssM_n\nearrow \mssM$, and where~$\tau_n\eqdef \inf\set{t\geq 0 : X_t\notin \mssM_n}$.
  As before, the energy $\langle M^{[\psi]}\rangle_t$ for $t\in [0,\zeta)$ is given by \eqref{fuku-qv}, now with $|\nabla \psi|_\g\in L^2_\loc(\mssM)$.
  
  \item\label{i:l:FukushimaD:3} For each  continuous $\psi\in W^{1,2}_\loc$, a super-martingale, multiplicative functional is defined by
\begin{align}\label{eq:DDExp}
L^{[\psi]}_t\eqdef \exp\paren{M^{[\psi]}_t-\tfrac{1}{2} \tav{M^{[\psi]}}_t} \car_{\set{t<\zeta}} \fstop
\end{align}
 \end{enumerate}
 \end{lem}
 
For the defining properties of `martingale additive functionals' and of `continuous additive functionals of zero energy' (as well as for the relevant equivalence relations  that underlie the uniqueness statements) we refer to the monograph \cite{FukOshTak11}.

\begin{ese}  If $\mssM=\R^n$ and  $\psi\in \mcC^2$ then $\tseq{M^{[\psi]}_t}_t$ is the {martingale part} in the \emph{It\^{o} decomposition}
\begin{equation*}\label{ito-rn}\psi(X_t)=\psi(X_0)+\int_0^t\nabla\psi(X_s)\diff X_s-\frac12\int_0^t\Delta\psi(X_s)\diff s\qquad P_x\text{-a.s.
~for all~$x\in \mssM$}\fstop
\end{equation*}
\end{ese}

We are now able to provide an explicit construction of  the  Brownian motion
\begin{align}\label{eq:RandomBM}
\mbfB^{\omega}\eqdef\seq{\Xi,\seq{\msF^\omega_t}_{t\geq 0},\seq{ X^\omega_t}_{t\geq 0}, \seq{ P^\omega_x}_{x\in \mssM_\partial}, \zeta^\omega}
\end{align}
on the randomly perturbed manifold~$(\mssM, \g^\bullet)$ which previously was introduced  by abstract  Dirichlet form techniques.

\begin{thm}\label{t:MainProcess}
Let~$h^\bullet\sim \FGF{s,m}$ with $m>0$ and~$s>n/2+1$. Then for~$\mbfP$-a.e.~$\omega\in\Omega$, the process~$\mbfB^\omega$ is a time-changed Girsanov transform of the standard Brownian motion~$\mbfB$ on~$(\mssM,\g)$.
More precisely:
\begin{enumerate}[$(a)$]
\item For q.e.~$x\in \mssM$, the law~$ P^\omega_x$ is locally absolutely continuous up to life-time~$\zeta^\omega$ w.r.t.\ the law~$P_x$ of~$\mbfB$ on the natural filtration~$\seq{\msF_t}_{t\geq 0}$ of~$\mbfB$, viz.
\begin{align}\label{eq:DDExponential}
\frac{\diff  P^\omega_x}{\diff P_x}\Bigg\lvert_{\msF_t\cap \ttset{t< \zeta^\omega}}= 
\exp\paren{\frac{n-2}2M^{\tquadre{h^\omega}}_t-\frac{(n-2)^2}{8}\av{M^{\tquadre{h^\omega}}}_t} \comma\qquad t\in [0,\zeta^\omega) \fstop
\end{align}

\item For q.e.~$x\in \mssM$, a trajectory~$(X^\omega_t)_{t\in [0,\zeta^\omega)}$ started at~$x$ satisfies
\begin{align}\label{eq:TimeChange}
 X^\omega_t=X_{\lambda^\omega_t} \comma \qquad \lambda^\omega_t\eqdef \inf\set{s>0: C^\omega_s>t}\comma\qquad C^\omega_t\eqdef \int_0^t e^{2h^\omega(X_s)}\diff s\fstop
\end{align}

\item\label{i:t:Main:3.3} 
The process~$\mbfB^\omega$ has life-time~$\zeta^\omega=C^\omega_\infty$. 
\end{enumerate}
\end{thm}

\begin{rem}[On conservativeness]
It is not clear to the authors whether the Dirichlet form~$(\mcE^\omega,\dom^\omega)$ is $\mbfP$-a.s.\ conservative.
In particular, the random Brownian motion~\eqref{eq:RandomBM} may in principle have \emph{finite} life-time~$\zeta^\omega$.
\end{rem}

\begin{proof}[Proof of Theorem~\ref{t:MainProcess}]
By Proposition~\ref{p:Properties}, the random field~$h^\bullet$ lies a.s.~in $W^{1,2}_{\loc}\cap \mcC(\M)$.
Thus, also~$e^{(n-2)h^\omega/2}\in W^{1,2}_{\loc}\cap \mcC(\M)$, and we may consider the Girsanov transform~$(\mcE^\varphi,\dom^\varphi)$, e.g.~\cite[\S6.3]{FukOshTak11}, of the canonical form~$(\mcE,\mcF)$ by the function~$\varphi=\varphi^\omega\eqdef e^{(n-2)h^\omega/2}$, satisfying
\begin{align}\label{eq:EPhi}
\mcE^\varphi(\phi,\psi)=\frac{1}{2}\int \g_*(\diff \phi,\diff\psi)\, \varphi^2 \dvol_\g \comma \qquad \phi,\psi\in \mcC^\infty_c \subset L^2(\varphi^2\,\vol_\g) \fstop
\end{align}
By standard results in the theory of Dirichlet forms,~$(\mcE^\varphi,\dom^\varphi)$ is a {regular} Dirichlet form on $L^2(\varphi^2\,\vol_\g)$, properly associated with the Girsanov transform~$\mbfB^\varphi$ of the standard Brownian motion~$\mbfB$. 
{Indeed, choosing $G_n\eqdef B_n(o)$, $n\in\N$, for some fixed $o\in\M$ yields a nondecreasing sequence of (quasi-)open sets with $\bigcup_n G_n=\M$ such that $\varphi,1/\varphi
$ and $\varphi|\nabla\varphi|\in L^2(G_n,\vol_g)$.
Then, according to~\cite[Thm.~4.9]{Fit97}, the Girsanov-transformed process is properly associated with the quasi-regular Dirichlet form obtained as the closure of $\mcE^\varphi$ with pre-domain
\[
\bigcup_{n\in\N}\dom_{G_n}
\]
where as usual $\dom_{G_n}\eqdef\{\psi\in\dom: \tilde\psi=0 \text{ q.e.~on }\M\setminus G_n\}$. Since obviously $\mcC^\infty_c \subset \bigcup_{n\in\N}\dom_{G_n}\subset\dom$, this Dirichlet form is even regular.
}

Now, let us denote by~$\tparen{\mcE^{\varphi,\mu},\dom^{\varphi,\mu}}$ the time-changed form, e.g.~\cite[\S6.2]{FukOshTak11}, of~$(\mcE^\varphi,\dom^\varphi)$ with respect to the measure~$\mu=\mu^\omega\eqdef e^{2h^\omega} \vol_\g$.
It is again standard that~$\tparen{\mcE^{\varphi,\mu},\dom^{\varphi,\mu}}$ is a {regular} Dirichlet form on~$L^2(\varphi^2 \mu)$, properly associated with the time change~$\mbfB^{\varphi,\mu}$ of~$\mbfB^\varphi$ induced by~$\mu$.
Since~$\varphi^2\mu=e^{nh^\omega}\vol_\g$, the form~$\mcE^{\varphi,\mu}$ coincides on~$\mcC^\infty_c$ with the form~$\mcE^\omega$ defined in~\eqref{eq:RandomForm}.
By regularity of both forms we conclude that~$\tparen{\mcE^{\varphi,\mu},\dom^{\varphi,\mu}}=(\mcE^\omega,\dom^\omega)$ is the canonical form on the Riemannian manifold~$\mssM^\omega=(\mssM,\g^\omega)$, properly associated with the corresponding Brownian motion~$\mbfB^\omega=\mbfB^{\varphi,\mu}$.

In order to characterize the law of~$\mbfB^\omega$ as in assertion~\ref{i:t:Main:1},~\ref{i:t:Main:2}, it suffices to note the following.
Since~$\mbfB$ is conservative, it is noted in e.g.~\cite[\S5 a)]{Ebe96} that the process
\begin{align*}
\mbfB^{\varphi}\eqdef\seq{\Xi^\varphi,\tseq{\msF^\varphi_t}_{t\geq 0},\tseq{X^\varphi_t}_{t\geq 0}, \seq{P^\varphi_x}_{x\in \mssM_\partial}, \zeta^\varphi}
\end{align*}
satisfies $X^\varphi_t= X_t$ for $t>0$ and
\begin{align*}
 \frac{\diff P^\varphi_x}{\diff P_x}\Bigg\lvert_{\msF_t\cap \ttset{t<\tau_{n-1}}}=&\ \exp\paren{M^{[\log \varphi_n]}_t-\tfrac{1}{2}\av{M^{[\log \varphi_n]}}_t} \comma \quad n\in \N\comma
\end{align*}
where the functions~$\log \varphi_n$ are given as in Lemma~\ref{l:FukushimaD}\iref{i:l:FukushimaD:2} for~$\log\varphi$ in place of~$\psi$, and the stopping times~$\tau_n$ are defined as~$\tau_n\eqdef \inf\set{t>0:X_t\notin \mssM_n}$ with~$\mssM_n$ again as in Lemma~\ref{l:FukushimaD}\iref{i:l:FukushimaD:2}.
The conclusion follows by letting~$n$ to infinity, since~$\mbfB^\omega$ is a time change of~$\mbfB^\varphi$, and therefore:~$P_x^\omega=P_x^\varphi$ for each~$x\in \mssM$.
Again since~$\mbfB^\omega$ is a time change of~$\mbfB^\varphi$, one has that~$ X^\omega_t=X^\varphi_{\lambda^\omega_t}=X_{\lambda^\omega_t}$ with $\lambda^\omega_t$ as in Equation~\eqref{eq:TimeChange} for each~$t>0$, cf.~\cite[Eqn.~(6.2.5)]{FukOshTak11}; assertion~\iref{i:t:Main:3.3} is~\cite[Exercise~6.2.1]{FukOshTak11}.
\end{proof}

\section{Geometric and Functional Inequalities for RRG's}\label{s:GFInequalities}
Given a Riemannian manifold $(\M,\g)$ and the intrinsically defined FGF noise $h^\bigdot$, we ask ourselves: how do basic geometric and spectral theoretic quantities of $(\M,\g)$ change if we switch on the noise? For instance, will $\EEE\,\vol_{\g^\bigdot}(\M)$ be smaller or larger than $\vol_\g(\M)$? 
How about $\lambda_0^\bigdot$, the random spectral bound, or $\lambda_1^\bigdot$, the random  spectral gap? Can we estimate them in terms of the unperturbed spectral quantities?
Can we estimate in average the rate of convergence to equilibrium on the random manifold? 

In the following, let a Riemannian manifold $(\M,\g)$ \emph{of bounded geometry} be given and a random field $h^\bullet\sim \FGF{s,m}$ with $m>0$ and~$s>n/2$. As before, put
$\g^\bullet=e^{2h^\bullet}\g$.

\subsection{Volume, Length, and Distance}
We will compare the random volume, random length, and random distance in the random Riemannian manifold $(M,\g^\bullet)$ with analogous deterministic quantities in geometries obtained by suitable averages of the conformal weight.
Recall that $\theta(x)\eqdef G_{s,m}(x,x)=\EEE[h^\bigdot(x)^2]\ge0$ and put
\begin{equation*}
\overline\g^n\eqdef e^{n\,\theta}\g, \qquad \overline\g^1\eqdef e^{\theta}\g \fstop
\end{equation*}
Further, recall that for given $\omega$ with continuous $h^\omega$, the volume of a measurable subset $A\subset \M$ w.r.t.~the Riemannian tensor $\g^\omega$ is given by
\begin{equation*}
\vol_{\g^\omega}(A)\eqdef \int_A e^{nh^\omega} \dvol_\g \fstop
\end{equation*}
Similarly, the length of an absolutely continuous curve $\gamma: [0,1]\to \M$
w.r.t.~the Riemannian tensor $\g^\omega$ is given by
\begin{equation*}
L_{\g^\omega}(\gamma)\eqdef  \int_0^1e^{h^\omega(\gamma_r)}\, |\dot\gamma_r|_\g\diff r \fstop
\end{equation*}

 \begin{prop} For any measurable $A\subset \M$
\begin{equation*}
\EEE \big[\vol_{\g^\bigdot}(A)\big]=\vol_{\overline\g^n}(A)\ge \vol_{\g}(A) \fstop
\end{equation*}
In particular, 
\begin{equation*}
e^{n^2 \theta^*/2}\cdot \vol_{\g}(A)\ge\EEE\big[\vol_{\g^\bigdot}(A)]\ge
e^{n^2 \theta_*/2}\cdot \vol_{\g}(A)
\end{equation*}
with $\theta_*\eqdef \inf_x G_{s,m}(x,x), \ \theta^*\eqdef \sup_x G_{s,m}(x,x)$.
\end{prop}

\begin{proof} It suffices to note that
\begin{equation*}
\EEE\big[ \vol_{\g^\bigdot}(A)]=\int_A \EEE[ e^{nh^\bigdot}] \dvol_\g=
\int_A e^{n^2 G_{s,m}(x,x)/2} \dvol_\g(x)=\vol_{\overline\g^n}(A) \fstop \qedhere
\end{equation*}
\end{proof}

\begin{prop} For any absolutely continuous curve $\gamma: [0,1]\to \M$
\begin{equation*}
\EEE\big[L_{\g^\bullet}(\gamma)\big]=L_{\overline\g^1}(\gamma)\ge L_{\g}(\gamma) \fstop
\end{equation*}
\end{prop}

\begin{proof} It suffices to note that
\begin{equation*}
\EEE L_{\g^\bigdot}(\gamma)=\int_0^1\EEE\Big[e^{h^\bullet(\gamma_r)}\Big]\, |\dot{\gamma}_r|_\g\diff r=
\int_0^1e^{\frac12\EEE[h^\bullet(\gamma_r)^2]}\, |\dot{\gamma}_r|_\g\diff r=L_{\overline\g^1}(\gamma)\fstop \qedhere
\end{equation*}
\end{proof}

\begin{prop} For each $x,y\in\M$
\begin{equation*}
\mssd_{\overline\g^1}(x,y)\geq \EEE\big[\mssd_{\g^\bigdot}(x,y)\big] \geq \mssd_\g(x,y)\cdot 
e^{ -\EEE\big[\sup_{z\in\M} h^\bigdot(z)\big]} \fstop
\end{equation*}
\end{prop}

\begin{proof} 
Given $x$ and $y$, let $\overline\gamma$ be any absolutely continuous curve connecting them. Then
\begin{align*}
L_{\overline\g^1}(\overline\gamma)=
\EEE\big[ L_{\g^\bigdot}(\overline\gamma)\big]
\ge \EEE\big[\inf_\gamma L_{\g^\bigdot}(\gamma)\big]=\EEE\big[\mssd_{\g^\bigdot}(x,y)\big].
\end{align*}
This proves the upper bound.

For the lower bound, let us assume that $\inf_{z\in\mssM}h^\bigdot(z)$ is finite for almost every $\omega$. Otherwise, the lower bound is trivially satisfied. Then $(\mssM,\g^\bigdot)$ is complete and locally compact so that there exists a constant speed geodesic $\gamma^\omega:[0,1]\to \M$ connecting $x$ and $y$. Then
\begin{equation*}
\mssd_{\g^\omega}(x,y)=\int_0^1 e^{h^\omega(\gamma^\omega_s)} \cdot|\dot{\gamma}_s^\omega|_\g \, \diff s\ge\mssd_\g(x,y)\cdot \int_0^1 e^{h^\omega(\gamma^\omega_s)} \diff s \geq \mssd_\g(x,y) \cdot \inf_{z\in\M} e^{h^\omega(x)}\fstop
\end{equation*}
Then, by Jensen's inequality and symmetry of the random field,
\begin{equation*}
\EEE\big[\mssd_{\g^\bigdot}(x,y)\big]
\ge \mssd_\g(x,y)\cdot  \EEE\Big[\inf_{z\in\M}e^{h^\bigdot(z)}\Big]\ge \mssd_\g(x,y)\cdot 
e^{- \EEE\big[\sup_{z\in\M} h^\bigdot(z)\big]}
\fstop \qedhere
\end{equation*}
\end{proof}


\subsection{Spectral Bound}
The \emph{$L^2$-spectral bound} for $(\M,\g^\omega)$ is defined by
\[
\lambda^\omega_0\eqdef \inf\spec{-\Delta_{\g^\omega}} \fstop
\]
By the standard variational characterization of the spectrum via Rayleigh--Riesz quotients we have that
\begin{align}\label{eq:VariationalSpectralBound}
\lambda^\omega_0=&\ \inf\set{\frac{\displaystyle\int_\M |\nabla u|_\g^2\,e^{(n-2)h^\omega}\,\dvol_\g}{\displaystyle\int_\M u^2\,e^{nh^\omega}\dvol_\g} : u\in \mcC^\infty_c} \fstop
\end{align}
Note that $\lambda_0$ is not necessarily $0$, e.g. $\lambda_0=\frac{(n-1)^2}4$ for the hyperbolic space of curvature $-1$. 
 
\begin{lem}[Measurability of the spectral bound]\label{l:MeasurableSpectralBound}
The function~$\omega\mapsto \lambda_0^\omega$ is measurable.

\begin{proof}
Let~$\mcC^\infty_c$ be endowed with the $\mcC^1$-topology~$\tau_1$, and note that $(\mcC^\infty_c,\tau_1)$ is separable.
Further note that, $\mbfP$-almost surely, $(\mcC^\infty_c,\tau_1)$ embeds continuously into~$(\dom^\omega,(\mcE^\omega)_1^{1/2})$, and that this embedding has dense image since~$(\mcE^\omega,\dom^\omega)$ is a regular Dirichlet form.
Therefore, there exists a countable $\Q$-vector space~$D\subset \mcC^\infty_c$ simultaneously $(\mcE^\omega)_1^{1/2}$-dense in~$\dom^\omega$ for $\mbfP$-a.e.~$\omega$.
As a consequence, the variational characterization~\eqref{eq:VariationalSpectralBound} holds as well when replacing~$\mcC^\infty_c$ by~$D$.
Since the integrals' quotient in this characterization is measurable as a function of~$\omega$, the corresponding infimum over~$D$ is as well a measurable function of~$\omega$, since~$D$ is countable and the infimum of any countable family of measurable functions is again measurable.
\end{proof}
\end{lem}
 
\begin{prop} For $n\ge2$
\begin{align*}
\big(\EEE\big[{\lambda_0^\bigdot}^{-n/2}\big]\big)^{-2/n}
\le\lambda_0^n
\end{align*}
with $\lambda_0^n$ the  spectral bound for the metric $\overline\g^n\eqdef e^{n\,\theta}\g$.
In particular, whenever~$\theta^*<\infty$, then
\[
\big(\EEE\big[{\lambda_0^\bigdot}^{-n/2}\big]\big)^{-2/n}
\le e^{((n-2)\theta^*-n\theta_*)n/2}\cdot\lambda_0\comma
\]
and, for homogeneous spaces,
\begin{equation*}
\big(\EEE\big[{\lambda_0^\bigdot}^{-n/2}\big]\big)^{-2/n}
\le e^{-n\theta}\cdot\lambda_0\fstop
\end{equation*}
\end{prop}

\begin{proof} For each $u$ and a.e.~$\omega$
\begin{equation*}
\int_\M u^2 e^{nh^\omega}\,\dvol_\g\le \frac1{\lambda_0^\omega}\int_\M |\nabla u|^2e^{(n-2)h^\omega}\,\dvol_\g \fstop
\end{equation*}
Integrating w.r.t.~$\diff\PPP(\omega)$ and applying H\"older's inequality yield
\begin{equation*}
\int_\mssM u^2\cdot \EEE[e^{nh^\bigdot}]
\,\dvol_\g\le\int_\mssM |\nabla u|_\g^2\cdot \EEE\quadre{\paren{\tfrac1{\lambda_0^\bigdot}}^{n/2}}^{2/n}\cdot \EEE\Big[e^{(n-2)h^\bigdot\cdot\frac{n}{n-2}}\Big]^{(n-2)/n}\,\dvol_\g
\end{equation*}
and thus with $\overline h\eqdef \frac n2\theta$,
\begin{equation*}
\int_\mssM u^2\cdot e^{n \overline h}\,\dvol_\g\le\EEE\Big[\big(\lambda_0^\bigdot \big)^{-n/2}\Big]^{2/n}\cdot\int_\mssM |\nabla u|_\g^2\cdot e^{(n-2) \overline h}\,\dvol_\g \fstop
\end{equation*}
Since this holds for all $u$ we conclude that
$\lambda_0^n\ge \big(\EEE[(\lambda_0^\bigdot)^{-n/2}]\big)^{-2/n}$.
\end{proof}

\begin{rem} Following the argumentation from the proof of Theorem \ref{two-sided} below, we can also derive a two-sided, pointwise estimate for the spectral bound, valid
for almost every $\omega$:
\begin{equation}
e^{-\alpha\sup |h^\omega|}
\le \frac{\lambda_0^\omega}{\lambda_0}\le
e^{\alpha\sup|h^\omega|} 
\end{equation}
with $\alpha:=2(n-1)$ if $n\ge2$ and $\alpha:=2$ if $n=1$.
\end{rem}

\subsection{Spectral Gap}
In the following we assume that $\M$ is closed, and we let~$\vol^\omega_\g= \vol_{\g^\omega}\eqdef e^{nh^\omega}\vol_\g$.
Then, the Laplacian~$\Delta_{\g^\omega}$ has compact resolvent and, in particular, it has discrete spectrum.
The spectral gap is defined by
\[
\lambda_1^\omega\eqdef \inf\tparen{\spec{-\Delta_{\g^\omega}}\setminus\{0\}} \fstop
\]
Denoting by
\begin{equation*}
\pi^\omega f\eqdef \frac1{\vol_\g^\omega(\M)}\int_\M f\dvol_\g^\omega\comma
\end{equation*}
the mean value of $f$ w.r.t.~the measure $\vol^\omega_\g$, the spectral gap has the variational representation
\begin{align}\label{eq:VariationalSpectralGap}
\lambda_1^\omega=&\inf\left\{\frac{\displaystyle\int_\M |\nabla u|_\g^2\, e^{(n-2)h^\omega}\dvol_\g}{\displaystyle \int_\M (u-\pi^\omega u)^2 \dvol^\omega_\g} : u\in\mcC^\infty_c\right\}.
\end{align}
Hence the spectral gap is the smallest non-zero eigenvalue of the Laplacian and the inverse of the smallest constant for which the Poincar\'e inequality holds.
By the very same proof of the measurability of the random spectral bound (Lemma~\ref{l:MeasurableSpectralBound}) we have as well the following:

\begin{lem}[Measurability of the spectral gap]
The function~$\omega\mapsto \lambda_1^\omega$ is measurable.
\end{lem}

The function~$h^\bullet$ is $\mbfP$-a.s.\ continuous by Proposition~\ref{p:Properties}, thus $\mbfP$-a.s.\ bounded by compactness of~$\M$. 
As a consequence, the $L^2(\vol_\g^\omega)$-norm is bi-Lipschitz equivalent to the $L^2(\vol_\g)$-norm.
Thus, the spaces $L^2(\vol_\g)$ and~$L^2(\vol^\omega_\g)$ coincide as sets.
Again by boundedness of~$h^\omega$, the form~$\mcE^\omega$ too is bi-Lipschitz equivalent to~$\mcE$ on~$\mcC^\infty_c$.
Set~$\mcE_1(u)\eqdef \mcE(u,u)+\norm{u}_{L^2(\vol_\g)}^2$, and analogously for~$\omega$.
By the equivalence of the $L^2$-norms and forms established above, the norm $\mcE_1^{1/2}$ is bi-Lipschitz equivalent to the norm~$(\mcE^\omega_1)^{1/2}$ on~$\mcC^\infty_c$.
Since~$\mssM$ is compact, both forms are regular, thus~$\dom^\omega$ too coincides with~$\dom$ as a set and the bi-Lipschitz equivalence of~$\mcE_1^{1/2}$ and~$(\mcE^\omega_1)^{1/2}$ extends to~$\dom$.

Given $\omega$ with continuous $h^\omega$, let $P_t^\omega\eqdef e^{t\Delta^\omega/2}$, $t>0$, denote the heat semigroup on $L^2(\vol_\g^\omega)$. 
For each $f\in L^2(\vol_\g)$, the functions $P_t^\omega f$ will converge as $t\to\infty$ to~$\pi^\omega f$.
The rate of convergence is determined by $\lambda^\omega_1$, viz.
\begin{align*}
\Big\|P^\omega_t f- \pi^\omega f \Big\|_{L^2(\vol_\g^\omega)}\le e^{-\lambda_1^\omega t}\cdot \big\|f\big\|_{L^2(\vol_\g^\omega)}
\end{align*}
or, equivalently,
\begin{align*}
\log\Big\|P^\omega_t f- \pi^\omega f \Big\|_{L^2(\vol_\g^\omega)}\le -\lambda_1^\omega t +\log \big\|f\big\|_{L^2(\vol_\g^\omega)} \fstop
\end{align*}
%
%

\begin{lem}\label{l:MeasurabilityPt}
The map~$\omega\mapsto \norm{P^\omega_t f -\pi^\omega f}_{L^2(\vol_\g^\omega)}$ is measurable for every~$f\in L^2(\vol_\g)$ and~$t>0$.

\begin{proof}
Firstly, let us discuss some heuristics.
For~$\alpha>0$, set~$\mcE^\omega_\alpha(\emparg)\eqdef \mcE^\omega(\emparg)+\norm{\emparg}^2_{L^2(\vol_\g^\omega)}$, and denote by~$\seq{G^\omega_\alpha}_{\alpha\geq 0}$ the $L^2(\vol_\g^\omega)$-resolvent semigroup of~$(\mcE^\omega,\dom^\omega)$, satisfying (e.g.~\cite[Thm.~I.2.8, p.~18]{MaRoe92})
\[
\mcE^\omega_\alpha(G^\omega_\alpha u,v)= \scalar{u}{v}_{L^2(\vol_\g^\omega)} \comma \qquad u\in L^2(\vol_\g^\omega)\comma v\in \dom^\omega \fstop
\]
We conclude the measurability in~$\omega$ of the left-hand side from that of the right-hand side which is clear from the identifications of sets~$L^2(\vol_\g^\omega)=L^2(\vol_\g)$ and~$\dom^\omega=\dom$.
For fixed~$t,\alpha>0$, writing the series expansion of~$e^{t\alpha(\alpha G_\alpha-1)}$ we conclude that
\[
\scalar{P^\omega_t u}{v}_{L^2(\vol_\g^\omega)} \comma \qquad u\in L^2(\vol_\g^\omega)\comma v\in \dom^\omega\comma
\]
is measurable as a function of~$\omega$, since~$P_t^\omega=\lim_{\alpha\to\infty} e^{t\alpha(\alpha G_\alpha-1)}$.
The measurability of~$\omega\mapsto\scalar{\pi^\omega u}{v}$ may be concluded in a similar way, which would then show the assertion.

\medskip

In order to make this argument rigorous, we resort to theory of direct integrals of quadratic forms in~\cite{LzDS20}.
In light of Corollary~\ref{c:GHSRevisited}, we may assume with no loss of generality that~$(\Omega,\msF,\mbfP)$ be the completion of a standard Borel space.
Let~$D\subset \mcC^\infty_c$ be the countable $\Q$-vector space simultaneously dense in $(\dom^\omega,(\mcE^\omega)_1^{1/2})$ for $\mbfP$-a.e.~$\omega\in\Omega$ constructed in the proof of Lemma~\ref{l:MeasurableSpectralBound}.

Now, let~$\omega\mapsto \dom^\omega$ be the measurable field of Hilbert spaces with underlying linear space~$S\eqdef \prod_{\omega\in\Omega} \dom^\omega=\dom^\Omega$ in the sense of~\cite[\S{II.1.3}, Dfn.~1, p.~164]{Dix81} with~$D$ as a fundamental sequence in the sense of~\cite[\S{II.1.3}, Dfn.~1(iii), p.~164]{Dix81}.
Further let~$\omega\mapsto L^2(\vol_\g^\omega)$ be the measurable field of Hilbert spaces with underlying space generated by~$S$ as above in the sense of~\cite[\S{II.1.3}, Prop.~4, p.~167]{Dix81}.
In particular, for every~$f\in L^2(\vol_\g)$, the constant field~$\omega\mapsto f\in L^2(\vol_\g^\omega)$ is a measurable vector field.
Furthermore, since
\begin{align*}
\int_\Omega \norm{f}_{L^2(\vol_\g^\omega)}^2\diff\mbfP(\omega)=\mbfE\quadre{\int_\mssM f^2 \dvol_\g^\omega}=\norm{f}_{L^2(\vol_\g)}^2 <\infty\comma
\end{align*}
all constant fields are elements of the direct integral of Hilbert spaces~$\int^\oplus_\Omega L^2(\vol_\g^\omega) \diff\mbfP(\omega)$.

It is readily verified that~$\omega\mapsto (\mcE^\omega,\dom^\omega)$ is, by construction, a direct integral of quadratic forms in the sense of~\cite[Dfn.~2.11]{LzDS20}.
As a consequence,~$\omega\mapsto P^\omega_t$ is a measurable field of bounded operators in the sense of~\cite[\S{II.2.1}, Dfn.~1, p.~179]{Dix81} by~\cite[Prop.~2.13]{LzDS20}.
Furthermore, since~$\omega\mapsto \vol^\omega_\g(\M)$ is measurable,~$\omega\mapsto \scalar{\pi^\omega u}{v}_{L^2(\vol^\omega_\g)}$ is measurable for every~$u,v\in D$.
Thus,~$\omega\mapsto \pi^\omega$ is a measurable field of bounded operators by~\cite[\S{II.2.1}, Prop.~1, p.~179]{Dix81}.

It follows that~$\omega\mapsto (P_t^\omega-\pi^\omega)$ is a measurable field of bounded operators.
Now fix~$f\in L^2(\vol_\g)$.
Since the constant field~$\omega\mapsto f$ is measurable as discussed above,~$\omega\mapsto (P_t^\omega -\pi^\omega)f$ too is a measurable vector field, by definition of measurable field of bounded operators.
Thus, its norm~$\omega\mapsto \norm{(P_t^\omega-\pi^\omega)f}_{L^2(\vol_\g^\omega)}$ too is measurable, which concludes the assertion.
%
\end{proof}
\end{lem}

\begin{lem} For every compact manifold $(\M,\g)$ (with continuous, not necessarily smooth metric $\g$),
\begin{equation}\label{eig-part}
\lambda_1(\M)=\inf\tset{ \max\{\lambda_0(\M_1), \lambda_0(\M_2)\}: \ \M_1, \M_2 \text{ non-polar, quasi-open, disjoint}\subset \M }
\end{equation}
where
\begin{equation}\label{var-bound}
\lambda_0(\M_i)\eqdef \inf \set{\frac{\displaystyle\int \abs{\nabla v}_\g^2\dvol_\g}{\displaystyle\int \abs{v}^2\dvol_\g}: v\in W^{1,2}_\ast\setminus\set{0}\comma \tilde v=0 \text{ q.e.~on }\M \setminus \M_i } \fstop
\end{equation}
Here, as usual in Dirichlet form theory, $\tilde v$ denotes a quasi continuous version of $v$, and  q.e. stands for quasi everywhere, see, e.g.,~\cite[\S2.1]{FukOshTak11}.

The infimum in \eqref{eig-part} is attained for  $\M_1\eqdef \{u>0\}, \M_2\eqdef \{u<0\}$ if $u$ is chosen as an eigenfunction for $\lambda_1(\M)$. In this case, indeed,
\begin{equation*}
\lambda_1(\M)=\lambda_0(\M_1)=\lambda_0(\M_2) \fstop
\end{equation*}
\end{lem}

\begin{proof} Let $u$ be an eigenfunction for $\lambda_1(\M)$ and put $\M_1\eqdef \{u>0\}, \M_2\eqdef \{u<0\}$. Choosing $v=u^+$ or $v=u^-$ in \eqref{var-bound}  one can verify  that $\lambda_0(\M_i)=\lambda_1(\M)$ for $i=1,2$. This proves the $\ge$-assertion in \eqref{eig-part}.

For the converse estimate, let $v_i\not=0$ for $i=1,2$ be minimizers for $\lambda_0(\M_i)$. Put $\lambda\eqdef \lambda_0(\mssM_1) \vee \lambda_0(\mssM_2)$ and
$u\eqdef v_1+t v_2$ with $t\not=0$ chosen such that $\int u\,\dvol_\g=0$. Then
\begin{equation*}
\int |\nabla u|_\g^2=\int |\nabla v_1|_\g^2 +t^2 \int |\nabla v_2|_\g^2\le \lambda \int | v_1|^2+t^2\lambda\int |v_2|^2=\lambda\, \int|u|^2
\end{equation*}
and thus $\lambda_1(\M)\le \lambda$.
\end{proof}

\begin{thm}\label{two-sided}
For $\mbfP$-a.e.~$\omega$,
\begin{equation}
e^{-\alpha\sup |h^\omega|}
\le \frac{\lambda_1^\omega}{\lambda_1}
\le e^{\alpha\sup |h^\omega|} 
\end{equation}
with $\alpha:=2(n-1)$ if $n\ge2$ and $\alpha:=2$ if $n=1$.
In particular,
\begin{equation*}
\EEE\Big[\big|\log\lambda_1^\bullet-\log\lambda_1\big|\Big]\le \alpha\, \EEE\Big[\sup |h^\bullet|\Big] \fstop
\end{equation*}
\end{thm}

\begin{proof}
Choose a minimizer $u$ for $\lambda_1(\M)$ and   put $\M_1\eqdef \{u>0\}, \M_2\eqdef \{u<0\}$. Then for each $\omega$ and each $i=1,2$,
\begin{eqnarray*}
\lambda_0^\omega(\M_i)&=&
\inf \set{\frac{\displaystyle\int |\nabla v|_\g^2\, e^{(n-2)h^\omega}\dvol_\g}{\displaystyle\int | v|^2e^{nh^\omega}\dvol_\g}: \ \tilde v=0 \text{ q.e.~on }\M\setminus \M_i}
\\
&\le&\frac{\sup_x e^{(n-2)h^\omega(x)}}{\inf_y e^{nh^\omega(y)}}\cdot
\inf \set{\frac{\displaystyle\int |\nabla v|_\g^2\dvol_\g}{\displaystyle\int | v|^2\,\dvol_\g}: \ \tilde v=0 \text{ q.e.~on }\M\setminus \M_i }
\\
&\le&
e^{\alpha\sup |h^\omega|}
\cdot \lambda_0(\M_i)
\\
&=&
e^{\alpha\sup |h^\omega|}
\cdot \lambda_1(\M) 
\end{eqnarray*}
with $\alpha \eqdef n+|n-2|$.
Hence according to the previous Lemma,
\begin{equation*}
\lambda_1^\omega(\M)\le e^{\alpha\sup |h^\omega|}
\cdot \lambda_1(\M) \fstop
\end{equation*}
Interchanging the roles of $\lambda_1^\omega$ and $\lambda_1$ and replacing $h^\omega$ by $-h^\omega$ yield the reverse inequality.
\end{proof}

\begin{cor} For all $f\in L^2(\vol_\g)$ and all $t>0$,
\begin{align}
\EEE\bigg[\log\Big\|P^\bullet_t f- \pi^\bullet f \Big\|_{L^2(\vol_\g^\bullet)}\bigg]\leq -\lambda_1 t\cdot e^{-\alpha\, \EEE\big[\sup |h^\bullet|\big]}+\log\norm{f}_{L^2(\vol_\g)}+ \frac{n^2\,\theta^*}{4}
\end{align}
with $\theta^*\eqdef\sup_x \EEE\big[h^\bullet(x)^2\big]$ and $\alpha \eqdef n+|n-2|$.
\end{cor}
\begin{proof}
With Theorem \ref{two-sided} we estimate
\begin{align*}
 \lambda_1^\omega t\geq \lambda_1 t\,  e^{-\alpha\sup |h^\omega|}\fstop 
\end{align*}
By the convexity  we may apply Jensen's inequality and get the estimate
\begin{align*}
\EEE\big[\lambda_1^\bullet t\big]\geq\lambda_1 t \, e^{- \alpha\, \EEE\big[{\sup |h^\bullet|}\big]} \fstop
\end{align*}
Moreover, again by Jensen's inequality 
\begin{align*}
\EEE\Big[\log \big\|f\big\|_{L^2(\vol_\g^\bullet)}\Big]\le \frac12\log \EEE\Big[ \big\|f\big\|_{L^2(\vol^\bullet)}^2\Big]\le\frac12\log \norm{f}_{L^2(\vol_\g)}^2+\frac{n^2\,\theta^*}{4} \comma
\end{align*}
which yields the claim.
\end{proof}

\section{Higher-Order Green Kernels --- Asymptotics and Examples}\label{sec:asymptotics}

\subsection{Green Kernel Asymptotics}

%
%

The next Theorem illustrates the asymptotic behavior of the higher-order Green kernel~$G_{s,m}(x,y)$ close to the diagonal in terms of the Riemannian distance~$\mssd(x,y)$.
The statement of the Theorem is sharp, as readily deduced by comparison with the analogous statement for Euclidean spaces, see Equation~\eqref{eq:ApproxG2} below.

\begin{thm}\label{t:EstimatesC}
Let~$(\mssM,\g)$ be a Riemannian manifold with bounded geometry, and $s>n/2$. Then, for every $\alpha\in (0,1]$ with $\alpha<s-n/2$  there exists a constant~$C_{\alpha, \purple{m}}>0$ so that
\begin{align*}
\rho_{s,m}(x,y)=\abs{G_{s,m}(x,x)+G_{s,m}(y,y)-2\, G_{s,m}(x,y)}^{1/2}\ \leq \
 C_{\alpha,\purple{m}}\cdot  \mssd(x,y)^{\alpha}  \comma
\end{align*}
 for all $m>0$ and all $x,y\in\M$.
 
If~$\mssM$ is additionally closed, then additionally
\begin{align*}
\rho_{s,m}(x,y)=\abs{\mathring G_{s,m}(x,x)+\mathring G_{s,m}(y,y)-2\, \mathring G_{s,m}(x,y)}^{1/2}\ \leq \
 C_\alpha\cdot  \mssd(x,y)^{\alpha}\comma
\end{align*}
for all~$m\geq 0$.
In this case, the constant $C_\alpha$ can be chosen such that
\begin{equation}\label{eq:t:EstimatesC:0}
C^2_\alpha=  C\, \paren{\frac{\lambda_1}{4}}^{n/2+\alpha-s}\ \frac{ \Gamma\big(s-n/2-\alpha\big)}{\alpha^*\cdot \Gamma(s)}
\end{equation}
with $\alpha^*\eqdef\alpha$ whenever $\alpha\in(0,1/2]$ and $\alpha^*\eqdef\alpha-1/2$ whenever $\alpha\in(1/2,1]$ and
$C>0$ is a constant only depending on $\M$.
\end{thm}

\begin{proof}
Note that
\begin{equation*}
\mathring G_{s,m}(x,x)+\mathring G_{s,m}(y,y)-2\, \mathring G_{s,m}(x,y)=G_{s,m}(x,x)+G_{s,m}(y,y)-2\, G_{s,m}(x,y)\comma \qquad m>0\fstop
\end{equation*}
Thus it suffices to prove the claim  for  $\mathring G_{s,m}$.

Assume first that~$\M$ is closed.
Throughout the proof,~$C>0$ denotes a finite constant, only depending on $\M$ but possibly changing from line to line. 
For~$x,y\in \mssM$ denote by~$\seq{[x,y]_r}_{r\in [0,1]}$ any constant speed distance-minimizing geodesic joining~$x$ to~$y$.

Assume first that~$\sigma\eqdef 2\alpha\in (0,1]$. Then,
\begin{align*}
\sup_{\substack{x,y\in \mssM\\ x\neq y}} &\bigg[ \frac{\Gamma(s)}{\mssd(x,y)^\sigma} \abs{\mathring G_{s,m}(x,x)+\mathring G_{s,m}(y,y)-2\, \mathring G_{s,m}(x,y)}\bigg]\leq
\\
\leq&\ 2 \sup_{\substack{x,y\in \mssM\\ x\neq y}} \bigg[\int_0^\infty\frac{\abs{p_t(x,x)-p_t(x,y)}}{\mssd(x,y)} \cdot\mssd(x,y)^{1-\sigma}\cdot e^{-m^2 t}\  t^{s-1} \diff t\bigg]
\\
\leq&\ 2\sup_{\substack{x,y\in \mssM\\ x\neq y}}\bigg[\int_0^\infty e^{-m^2 t} \ t^{s-1} \cdot  \mssd(x,y)^{1-\sigma} \int_0^1 \abs{\nabla p_t(x,[x,y]_r)}\diff r \diff t \bigg]\fstop
\end{align*}
By~\eqref{eq:l:EstimatesC:2} 
\begin{align}
\sup_{\substack{x,y\in \mssM\\ x\neq y}} \bigg[ \Gamma(s)\, \mssd(x,y)^{-\sigma}& \abs{\mathring G_{s,m}(x,x)+\mathring G_{s,m}(y,y)-2\, \mathring G_{s,m}(x,y)}\bigg]\nonumber
\\
\nonumber
\leq C\sup_{\substack{x,y\in \mssM\\ x\neq y}}& \bigg[\mssd(x,y)^{1-\sigma} \int_0^\infty e^{-(m^2+\lambda_1/2) t} \ t^{s-1} \ (t^{-n/2-1/2}\vee 1) \ \cdot
\\
&\ \cdot \int_0^1 
\exp\paren{-\frac{r^2\, \mssd(x,y)^2}{Ct}}\diff r\diff t  \bigg]
\label{eq:Star}
\\
\leq C\sup_{\substack{x,y\in \mssM\\ x\neq y}} &\bigg[ \int_0^\infty e^{-\lambda_1 t/2} \ t^{s-1+(1-\sigma)/2} \ (t^{-n/2-1/2}\vee 1) \ \cdot 
\nonumber
\\
&\ \cdot \int_0^1 \paren{\frac{r^2\, \mssd(x,y)^2}{t}}^{(1-\sigma)/2} 
\exp\paren{-\frac{r^2\, \mssd(x,y)^2}{Ct}}r^{\sigma-1}\diff r\diff t  \bigg]
\nonumber
\\
\leq \frac C\sigma\,  \int_0^\infty &e^{-\lambda_1 t/4}\ t^{s-(n+\sigma)/2-1} \diff t\
= \  \frac C\sigma\, \Big(\frac{4}{\lambda_1}\Big)^{s-(n+\sigma)/2}\ \Gamma\big(s-(n+\sigma)/2\big) \fstop\nonumber
\end{align}
For the  last \emph{in}equality, we used the fact that the function $R\mapsto R^{(1-\sigma)/2}
\exp(-R/C)$ is uniformly bounded on  $(0,\infty)$, independently of $\sigma\in(0,1]$.

Assume now that~$\sigma\eqdef 2\alpha\in (1,2]$. 
Then, similarly to the previous case,
\begin{align*}
\sup_{\substack{x,y\in \mssM\\ x\neq y}}\bigg[ \frac{\Gamma(s)}{\mssd(x,y)^\sigma}& \abs{\mathring G_{s,m}(x,x)+\mathring G_{s,m}(y,y)-2\, \mathring G_{s,m}(x,y)}\bigg]\nonumber
\\
\leq&\ \sup_{\substack{x,y\in \mssM\\ x\neq y}} \int_0^\infty \frac{\abs{p_t(x,x)+p_t(y,y)-2p_t(x,y)}}{\mssd(x,y)^{\sigma}} \, e^{-m^2 t} t^{s-1} \diff t\nonumber
\\
\leq&\ \sup_{\substack{x,y\in \mssM\\ x\neq y}} \int_0^\infty e^{-m^2 t}\  t^{s-1} \ \mssd(x,y)^{1-\sigma} \int_0^1\abs{\nabla_2\, p_t(x,[x,y]_{\rho})-\nabla_2\, p_t(y,[x,y]_{\rho})}\diff \rho \diff t 
\\
\leq&\ \sup_{\substack{x,y\in \mssM\\ x\neq y}} \int_0^\infty e^{-m^2 t} \ t^{s-1} \ \mssd(x,y)^{2-\sigma} \int_0^1\int_0^1\abs{\nabla_1\nabla_2\, p_t([x,y]_{\varrho},[x,y]_{\rho})}\diff \rho\diff \varrho \diff t \nonumber
\fstop
\end{align*}
By~\eqref{eq:l:EstimatesC:3}, similarly 
\begin{align*}
\sup_{\substack{x,y\in \mssM\\ x\neq y}}
\bigg[ \frac{\Gamma(s)}{\mssd(x,y)^\sigma}&\abs{\mathring G_{s,m}(x,x)+\mathring G_{s,m}(y,y)-2\, \mathring G_{s,m}(x,y)}\bigg]
\\
\leq&\ C\,\sup_{\substack{x,y\in \mssM\\ x\neq y}} \int_0^\infty \int_0^1\int_0^1 
\exp\paren{-\frac{(\rho-\varrho)^2\,\mssd^2(x,y)}{Ct}} \diff \rho\diff \varrho \ \cdot
\\
&\phantom{C\, \Gamma(s) \sup_{x,y\in \mssM} \int_0^\infty}
\cdot \mssd(x,y)^{2-\sigma}\,e^{-(m^2+\lambda_1/2)t}\ t^{s-1} \,(t^{-n/2-1}\vee 1)\,  \diff t
\\
\leq &\ C\, \sup_{\substack{x,y\in \mssM\\ x\neq y}} \int_0^\infty\int_0^1 \int_0^1  \paren{\frac{(\rho-\varrho)^2\, \mssd^2(x,y)}{t}}^{1-\sigma/2} \cdot 
 \\
 &\phantom{C\, \sup_{\substack{x,y\in \mssM\\ x\neq y}} \int_0^\infty} 
  \cdot \exp\paren{-\frac{(\rho-\varrho)^2\,\mssd^2(x,y)}{Ct}} \, |\rho-\varrho|^{\sigma-2}\diff \rho\diff \varrho\ \cdot
\\
&\phantom{C\, \sup_{\substack{x,y\in \mssM\\ x\neq y}} \int_0^\infty}  \cdot t^{1-\sigma/2}\ e^{-\lambda_1 t/2} \ t^{s-1} (t^{-n/2-1}\vee 1)\,\diff t 
\\
\leq&\ \frac C{\sigma(\sigma-1)}\,  \int_0^\infty e^{-\lambda_1 t/4}\ t^{s-(n+\sigma)/2-1}\diff t=\frac C{\sigma(\sigma-1)}\, \paren{\frac{4}{\lambda_1}}^{s-(n+\sigma)/2}\, \Gamma\big(s-(n+\sigma)/2\big)\fstop
\end{align*}

Assume now that~$\M$ has bounded geometry.
The proof holds in a similar way to the case of closed~$\M$, having care to replace the application of~\eqref{eq:l:EstimatesC:2} with~\eqref{eq:l:EstimatesBG:2} and~\eqref{eq:l:EstimatesC:3} with~\eqref{eq:l:EstimatesBG:4}.
\end{proof}

\begin{cor}\label{c:EstimatesC}
Let~$\mssM$ be a compact manifold. Then, there exists a constant $C>0$ such that for all $m\ge0$ and all $x,y\in\M$, 
\begin{align*}
\rho_{s,m}(x,y)\ \leq \
\begin{cases}
C\cdot \big(\frac{\lambda_1}2\big)^{-s/2}\cdot \mssd(x,y), \qquad& s\ge \frac n2 +2\comma\\
   \frac C{\sqrt{s-n/2-1}}\cdot \mssd(x,y), \qquad& s\in(\frac n2 +1, \frac n2+2]\comma\\
   \frac C{s- n/2}\cdot \mssd^{s/2-n/4}(x,y), \qquad& s\in(\frac n2 , \frac n2+1]\fstop
   \end{cases}
\end{align*}
  \end{cor}
  The estimate in the third case is not sharp. The previous Theorem provides estimates $\rho_{s,m}\le C_\alpha\,\mssd^\alpha$ for every $\alpha< s-n/2$.
(As $\alpha\to  s-n/2$, however, the constant $C_\alpha$ will diverge.)
\begin{proof}
The eigenfunction representation~\eqref{eq:G0-CompactM} of~$\mathring{G}_{s,m}$ yields that
\begin{equation*}
\rho_{s,m}^2(x,y)=\sum_{j=1}^\infty (m^2+\lambda_j/2)^{-s}\Big[\phi_j^2(x)+\phi_j^2(y)-2\phi_j(x)\phi_j(y)\Big]\fstop
\end{equation*}
Hence, $\rho_{s,m}^2(x,y)\le \rho_{s,0}^2(x,y)$ for all $x,y, s,m$ under consideration. Moreover, 
for all $x,y\in\M$ the function 
\begin{align}\label{dist-mon}s\mapsto (\lambda_1/2)^{s}\cdot \rho_{s,0}^2(x,y) \text{ is decreasing.}
\end{align}
 Therefore, the first case  $s\ge \frac n2 +2$ follows from the choice $s=\frac n2+2$ which is included in the second case.

In the second case  $s\in(\frac n2 +1, \frac n2+2]$,  with the choice $\alpha=1$ the previous Theorem  provides the estimate
\begin{align*}\frac{\rho^2_{s,m}(x,y)}{\mssd^2(x,y)}\le C_1^2 \leq C\, {\lambda_1}^{n/2+1-s}\ \frac{ \Gamma\big(s-n/2-1\big)}{ \Gamma(s)}\le\frac{C'}{s-n/2-1}\fstop
\end{align*}

In the third case $s\in(\frac n2, \frac n2+1]$,  with the choice $\alpha=\tfrac{1}{2}(s-\frac n2)\in (0,1/2]$ the previous Theorem  provides the estimate
\begin{align*}\frac{\rho^2_{s,m}(x,y)}{\mssd^{s-n/2}(x,y)}\le C_\alpha^2\leq C\, {\lambda_1}^{n/4-s/2}\ \frac{ \Gamma\big(s/2-n/4\big)}{(s-n/2)\, \Gamma(s)}\le\frac{C'}{(s-n/2)^2}\fstop & \qedhere
\end{align*}

\end{proof}
\subsection{Supremum estimates}

Now let us combine  Dudley's estimate, Theorem~\ref{dudley}, for the supremum of the Gaussian  field with our H\"older estimate, Corollary~\ref{c:EstimatesC}, for the noise distance.
\begin{thm} For every compact manifold $\M$ there exists a constant $C=C(\M)$ such that for every 
$h^\bullet\sim\gFGF[\mssM]{s,m}$ with any $m\ge0$,
\begin{align*}
\mbfE\bigg[\sup_{x\in \M} h^\bullet(x)\bigg]\leq 
\begin{cases}
C\cdot (\lambda_1/2)^{-s/2}, \qquad& s\ge \frac n2 +1\comma\\
     C\cdot (s-n/2)^{-3/2}, \qquad& s\in\big(\frac n2 , \frac n2+1\big]\fstop
   \end{cases}
\end{align*}
\end{thm}

\begin{proof}  Recall the Notation~\ref{notation:Covering} for the covering number of a pseudo-metric, and let~$\rho=\rho_{s,m}$ be as in~\eqref{green-dist}.
For the Riemannian distance $\mssd$ on the compact manifold $\M$,
\begin{equation*}
N_\mssd(\eps)\le \big(C\cdot \eps^{-n}\big)\vee 1
\end{equation*}
for some constant $C>0$. 

In the case $s\in(\frac n2 , \frac n2+1]$, Corollary~\ref{c:EstimatesC} yields   $\rho\le C_s\,\mssd^\alpha$ with $\alpha\eqdef \tfrac{1}{2}(s-\frac{n}{2})$ and~$C_s\eqdef C/(s-n/2)$, and thus
\begin{equation*}
B^{(\rho)}_\eps(x)\supset B^{(\mssd)}_{(\eps/C_s)^{1/\alpha}}(x)\comma\qquad \eps>0,\quad x\in\M \fstop
\end{equation*}
This implies
\begin{equation*}
N_{s,m}(\eps)\le N_\mssd\big((\eps/C_s)^{1/\alpha}\big)\le \tparen{C\cdot (\eps/C_s)^{-n/\alpha}}\vee 1 \fstop
\end{equation*}
Hence,
\begin{align*}
\int_0^\infty\tparen{\log N_{s,m}(\eps)}^{1/2}\diff\eps&\le \int_0^{C^{\alpha/n}\cdot C_s}\paren{c-\frac n\alpha\log\frac\eps{C_s}}^{1/2} \diff \eps
= C_s\cdot \int_0^{C^{\alpha/n}}\paren{c-\frac n\alpha\log\eps}^{1/2}\diff \eps
\\
&\le \frac{C_s}{\alpha^{1/2}}\cdot\int_0^{C^{1/n}} \tparen{c'- n\log\eps}^{1/2}\diff\eps
=\frac{C_s}{\alpha^{1/2}}\cdot \ C'= \frac{C''}{(s-n/2)^{3/2}} \fstop
\end{align*}

In the case $s>n/2+1$, the monotonicity property \eqref{dist-mon} and the estimate from Corollary \ref{c:EstimatesC} (for $s=n/2+1$)
imply
\begin{equation*}
\rho_{s,m}(x,y)\le (\lambda_1/2)^{(n/2+1-s)/2}\cdot \rho_{n/2+1,0}(x,y)\le C\, (\lambda_1/2)^{(n/2+1-s)/2}\cdot \mssd^{1/2}(x,y) \fstop
\end{equation*}
Hence, following the previous argumentation we obtain
\begin{align*}\int_0^\infty\tparen{\log N_{s,m}(\eps)}^{1/2}\diff\eps&\le 
 C\, (\lambda_1/2)^{(n/2+1-s)/2}
\cdot \int_0^{C^{1/n}}\tparen{c-2n\log\eps}^{1/2}\diff\eps\\
&\le
 C'\, (\lambda_1/2)^{(n/2+1-s)/2}=
 C''\, (\lambda_1/2)^{-s/2} \fstop \qedhere
\end{align*}
\end{proof}

\subsection{Examples}

\subsubsection{Euclidean space}
On the $n$-dimensional Euclidean space, the Green kernels are given by 
\begin{align*}
G_{s,m}^{\R^n}(x,y)\eqdef G^n_{s,m}(|x-y|)
\end{align*}
with
\begin{align}\label{eq:IntegralRepGEuclideanSp}
G^n_{s,m}(r)\eqdef \frac1{(2\pi)^{n/2}\, \Gamma(s)}\int_0^\infty e^{-r^2/2t}\, e^{-m^2t}\, t^{s-n/2-1}\diff t\fstop
\end{align}
Note that 
$G^n_{s,m}(r)\le G^n_{s,m}(0)<\infty$ if $s>n/2$ 
whereas $G^n_{s,m}(r)\approx \log\frac1r$ as $r\to0$ if $s=n/2$ and  $G^n_{s,m}(r)\approx \frac1{r^{n-2s}}$ if $s<n/2$.
Closed expressions for $G^{n}_{1,m}(r)$ are available for odd $n$, e.g.
\begin{align}\label{expli-G}
  G^{1}_{1,m}(r)=\frac1{\sqrt2m}e^{-\sqrt2m \, r}, \qquad
 G^{3}_{1,m}(r)=\frac 1{2 \pi\, r}\, e^{-\sqrt2m\, r}, \qquad
 G^{5}_{1,m}(r)=\frac {(1+\sqrt2mr)}{4\pi^2 r^3}\, e^{-\sqrt2m\, r}\fstop
 \end{align}
 From this, with the relations formulated below, various other explicit expressions can be derived, for instance, 
$G^{3}_{2,m}(r)=\frac{1}{2\pi\,\sqrt{2}m}\,
 e^{-\sqrt{2}m\,  r}$
 and, more generally,
  \begin{align*}G^{n}_{\frac{n+1}2,m}(r)=\frac1{(2\pi)^{\frac{n-1}2}\, {\Gamma(\frac{n+1}2})\, {\sqrt2m}}\, e^{-\sqrt2m \, r}\fstop \end{align*}

\begin{lem} For $m,s,r>0$ and $n\in\N$, the Green kernels $G^n_{s,m}(r)$ satisfy the relations
\begin{subequations}\label{eq:RelationsG}
\begin{align}
\label{eq:RelationsG:1}
G^n_{s,am}(r)&=a^{n-2s} \, G^n_{s,m}(ar)\comma &  &a> 0\comma
\\
\label{eq:RelationsG:2}
G^n_{s+a,m}(r)&=\frac{1}{(2\pi)^a}\frac{\Gamma(s)}{\Gamma(s+a)}\, G^{n-2a}_{s,m}(r)\comma& -s<\ & a <n/2 \comma
\\ sm^2 G_{s+1,m}^n(r)&=(s-n/2)\, G_{s,m}^n(r)+\frac{r^2}{2(s-1)}\, G^n_{s-1,m}(r)\comma & &s>1
\fstop
\end{align}
\end{subequations}
\end{lem}

\begin{proof} The first two formulas follow by change of variable in the integral representation~\eqref{eq:IntegralRepGEuclideanSp}. The third one follows by integration by parts via
	\begin{equation*}
	\int_0^\infty e^{-r^2/2t}e^{-m^2t}t^{s-n/2}\, \diff t
	=\frac1{m^2}\int_0^\infty\frac{\diff}{\diff t}(e^{-r^2/2t}t^{s-n/2})e^{-m^2t}\, \diff t \fstop \qedhere
	\end{equation*}
\end{proof}

\begin{thm}\label{asy-G} For $m>0$, the asymptotics of the higher order Green kernel as $r\to0$ is as follows
\begin{align}
\label{eq:ApproxG2}
G^n_{s,m}(0)-G^n_{s,m}(r) \ \asymp\ \begin{cases}
-\displaystyle{\frac{\Gamma({n/2-s})}{2^{{s}}\, \pi^{n/2}\, \Gamma(s)}} \cdot r^{2s-n} &\text{if } s\in (n/2,n/2+1) \comma
\\
\displaystyle{\frac{1}{2^{{n/2}}\, \pi^{n/2}\, \Gamma(s)}} \cdot r^2 \log\frac1{r}& \text{if } s=n/2+1 \comma
\\
\displaystyle{\frac{\Gamma(s-n/2-1)}{2^{{n/2}+1}\, m^{2s-n-2}\, \pi^{n/2}\,\Gamma(s)}}\cdot r^2 &\text{if } s>n/2+1\comma
\end{cases}
\end{align}
where~$\asymp$ is as in Notation~\ref{notation:Asymp}.
\end{thm}

\begin{proof} For convenience, we provide two proofs. The first one is based on direct calculations.

For proving the claim in the case $s>n/2+1$,
consider
\begin{align*}
\lim_{r \to 0} r^{-2}\cdot(2\pi)^{n/2}\Gamma(s)\big[G^n_{s,m}(0)-G^n_{s,m}(r)\big]=&\ \lim_{r \to 0}\int_0^\infty \frac{1-e^{-r^2/2t}}{r^2}  \, e^{-m^2t}\, t^{s-n/2-1} \diff t
 \\
 =&\ \frac {1}2\cdot \int_0^\infty  e^{-m^2t}\, t^{s-n/2-2}\diff t
 = \frac {1}2\Gamma\paren{s-\frac n2-1} m^{-2s+n+2} \comma
 \end{align*}
since by assumption $s>1+\frac n2$.
In the case $n/2<s<n/2+1$, consider
 \begin{align*}
 \lim_{r\to0}r^{-2s+n}(2\pi)^{n/2}\Gamma(s)\cdot\big[G^n_{s,m}(0)-G^n_{s,m}(r)\big]=&\ \lim_{r\to0}r^{-2s+n}\cdot \int_0^\infty \big(1-e^{-r^2/2t}\big)  \, e^{-m^2t}\, t^{s-n/2-1}\diff t
 \\
 =&\ \lim_{r\to0} \int_0^\infty \big(1-e^{-1/2t}\big)  \, e^{-(mr)^2t}\, t^{s-n/2-1}\diff t
 \\
 =&\ \int_0^\infty \big(1-e^{-1/2t}\big)  \, t^{s-n/2-1}\diff t
  \\
=&\ 2^{n/2-s}\int_0^\infty \big(1-e^{-u}\big)  \, u^{n/2-s-1}\diff u
\\
=&\ -\frac{2^{n/2-s}}{n/2-s}\int_0^\infty e^{-u}  \, u^{n/2-s}\diff u
\\
=&\ -2^{n/2-s}\Gamma(n/2-s) \fstop
\end{align*}
(For the third equality above, we used the monotonicity of the integrand in $r$, 
and for the fifth, we used integration by parts.)
In the case $s=\frac n2+1$, 
applying De l'H\^opital twice yields
  \begin{align*}
  \lim_{r\to0}\frac{(2\pi)^{n/2}\Gamma(s)}{r^{2}\log 1/r}\ \cdot&\ \big[G^n_{s,m}(0)-G^n_{s,m}(r)\big]
\\
 =& \lim_{r\to0}\frac 1{r^{2}\log 1/ r}\cdot \int_0^\infty \big(1-e^{-r^2/2t}\big)  \, e^{-m^2t} \diff t
 \\
 =& -\lim_{r\to0}\frac{1}{r(1+2\log r)}\int_0^\infty r e^{-r^2/2t} e^{-m^2 t} t^{-1}\diff t
\\
=& \lim_{r\to0}\frac{r}{2}\int_0^\infty r e^{-r^2/2t} e^{-m^2 t} t^{-2}\diff t
&& \quadre{ \frac{r^2}{2t}=u\comma -\frac{r^2}{2t^2}\diff t=\diff u }
\\
=& \lim_{r\to0}\int_0^\infty e^{-\frac{m^2r^2}{2u}} e^{-u} \diff u=1\fstop
 \end{align*}
{An alternative proof of the claims may be obtained from the representation~\cite[Eqn.~(15), p.~183]{Wat44} of the Green kernel $G^n_{s,m}(r)$  in terms of the modified Bessel functions $K_\alpha$ for $\alpha\in\R$:
\begin{align}\label{eq:NonIntegerG}
G^n_{s,m}(r)=&\ \frac{{2}}{ (2\pi)^{n/2}\, \Gamma(s)}\,\Big(\frac r{\sqrt2 m}\Big)^{s-n/2}\,  K_{s-n/2}(\sqrt2 m\,r)  \comma
\end{align}
and the known asymptotics 
for~$K_\alpha$ and its derivatives.
}
 \end{proof}
 
\begin{rem}
For all integer values of~$s$ and~$n$, explicit expressions for~$G^n_{s,m}$ may be obtained from~\eqref{eq:NonIntegerG} in terms of the \emph{reverse Bessel polynomials}, e.g.~\cite[\S{II.1}, Eqn.s (7)--(9)]{Gro78}, in view of the characterization in terms of such polynomials of the Bessel function $K_\alpha$ for semi-integer~$\alpha$, e.g.~\cite[\S{III.1}]{Gro78}.
\end{rem}

\begin{figure}[htb!]
\centering
\includegraphics[scale=.5]{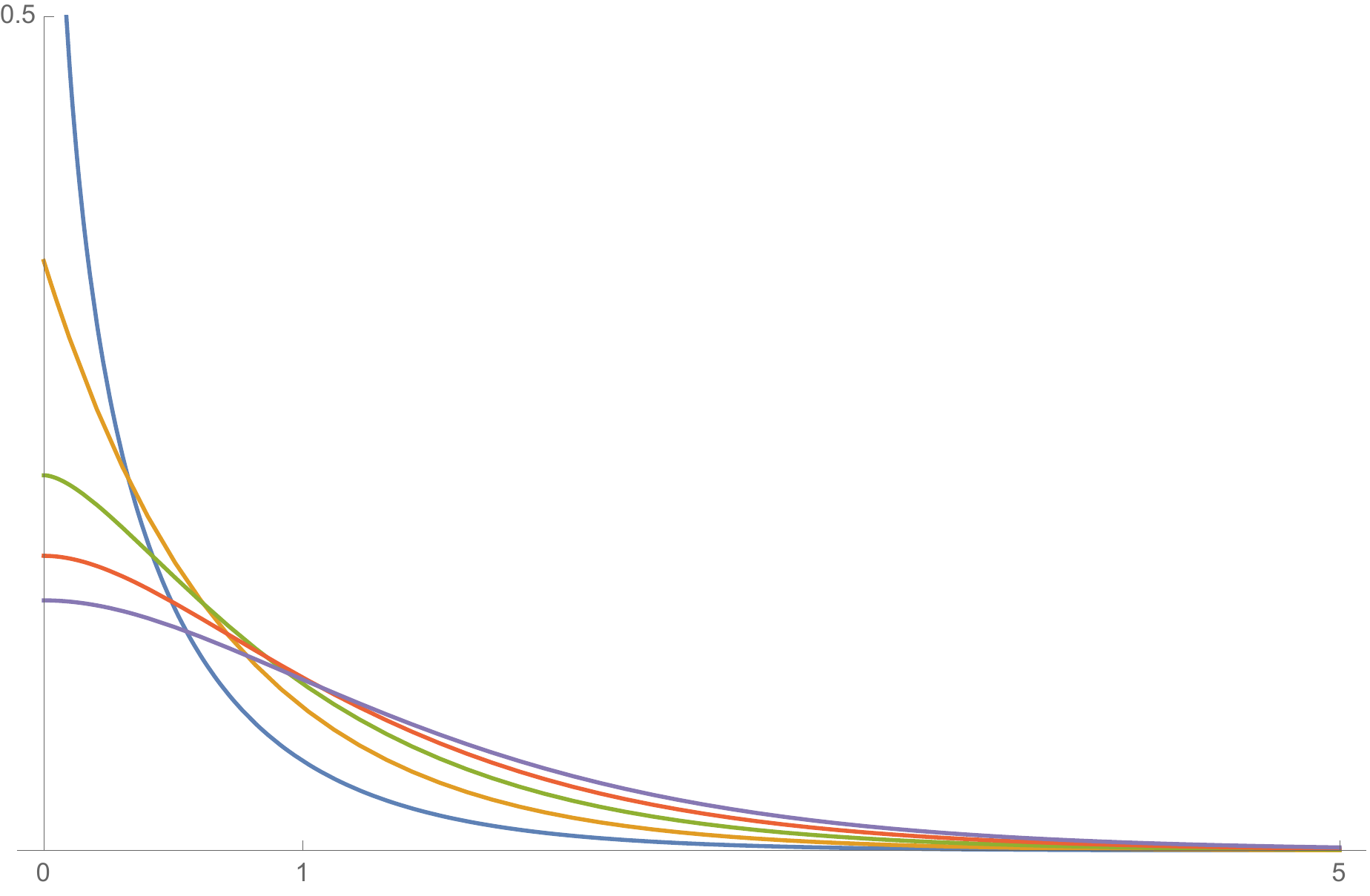}
\caption{The Green kernels~$G^1_{s,1}$ for~$2s=1,\dotsc,5$ (in reverse order w.r.t.\ the value at~$0$). Note that~$\lim_{r\rar 0} G^1_{1/2,1}(r)=+\infty$.}
\label{fig:G1}
\end{figure}

\subsubsection{Torus} 
Let ${\mathbb T}=\R/\N$ be the circle of length 1. 
\begin{prop}
For all $s, m>0$,
 \begin{align}G^{{\mathbb T}}_{s,m}(x,y)=\sum_{j\in\mathbb Z}G^{{\mathbb R}}_{s,m}(x,y+j)\fstop\end{align}
 In particular, $G^{{\mathbb T}}_{s,m}(x,y)=G^{{\mathbb T}}_{s,m}\big(\mssd_{\mathbb T}(x,y)\big)$
 with $\mssd_{\mathbb T}(x,y)=\min\{ |x-y|, 1-|x-y|\}$ for~$x,y\in [0,1]$ and
  \begin{align}\label{G-cosh}G^{{\mathbb T}}_{1,m}(r)=
 \frac{\cosh\Big(\sqrt2m \big(r-1/2\big)\Big)}{\sqrt2m\cdot\sinh(m/\sqrt2)}\fstop\end{align}
\begin{proof} The first claim is an immediate consequence of the analogous formula for the heat kernel:
 \begin{align*}p_t^{{\mathbb T}}(x,y)=\sum_{j\in\mathbb Z}p_t^{{\mathbb R}}(x,y+j)\fstop\end{align*}
 The second claim follows from the first one combined with \eqref{expli-G} according to
 \begin{align*}
 G^{{\mathbb T}}_{1,m}(r)&=\frac1{\sqrt2m}\sum_{k\in\N_0}e^{-\sqrt2m(r+k)}+\frac1{\sqrt2m}\sum_{k\in\N_0}e^{-\sqrt2m[(1-r)+k]}\\
 &=\frac1{\sqrt2m\, \big(1-e^{-\sqrt2m}\big)}\Big(e^{-\sqrt2mr}+e^{-\sqrt2m(1-r)}\Big)
 = \frac{\cosh\Big(\sqrt2m \big(r-1/2\big)\Big)}{\sqrt2m\cdot\sinh(m/\sqrt2)}
\end{align*}
 for $r\in [0,1/2]$. 
 \end{proof}
 \end{prop}
 
 \begin{figure}[htb!]
\includegraphics[scale=.5]{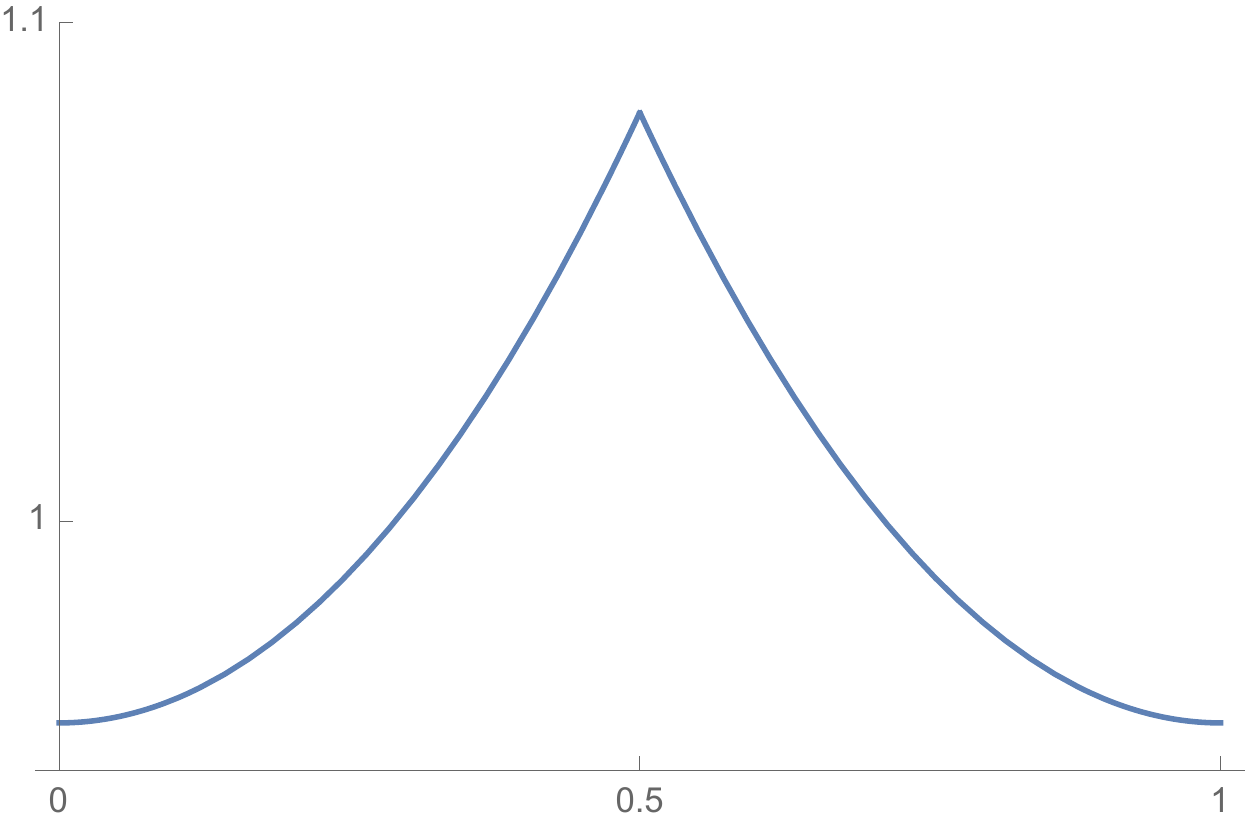}
\caption{The Green kernel~$G^{\mbbT}_{1,1}(\tfrac{1}{2},y)$ with~$y\in[0,1)$.}
\end{figure}

\begin{thm}
For $m=0$ and integer~$s\geq 1$,
\begin{align*}
\mathring G^\mbbT_{s,0}(r)=
(-1)^{s-1}\frac{2^s}{(2s)!}B_{2s}(r) \comma \qquad s\in \N\comma \qquad r\in [0,1/2] \comma
\end{align*}
where~$B_n$ denotes the $n^\text{th}$ Bernoulli polynomial.
 \end{thm}
 
 In particular,
 \begin{align}\label{tor1}
 \mathring G^{{\mathbb T}}_{1,0}(r)&=\bigg(r-\frac12\bigg)^2-\frac1{12}\comma
 \\
 \mathring G^{{\mathbb T}}_{2,0}(x,y)&=-\frac1{6}\bigg(r-\frac12\bigg)^4+ \frac1{12}\bigg(r-\frac12\bigg)^2-\frac7{1440}
 \comma
 \\
 \mathring G^{{\mathbb T}}_{3,0}(x,y)&=\frac1{90}\bigg(r-\frac12\bigg)^6- \frac1{72}\bigg(r-\frac12\bigg)^4+\frac7{1440}\bigg(r-\frac12\bigg)^2-\frac{31}{120960}
 \fstop
 \end{align}

Further observe that \begin{align*}\lim_{r\to0}\frac1r\Big( \mathring G^{{\mathbb T}}_{1,0}(0)- \mathring G^{{\mathbb T}}_{1,0}(r)\Big)\ = \ 
\lim_{r\to0}\frac1r\Big(G^{{\mathbb T}}_{1,m}(0)-G^{{\mathbb T}}_{1,m}(r)\Big)\ = \ 
\lim_{r\to0}\frac1r\Big(G^{{\mathbb R}}_{1,m}(0)-G^{{\mathbb R}}_{1,m}(r)\Big)\ = \  1
\end{align*}
for all $m>0$, and 
 \begin{align*}\lim_{r\to0}\frac1{r^2}\Big( \mathring G^{{\mathbb T}}_{2,0}(0)- \mathring G^{{\mathbb T}}_{2,0}(r)\Big)\ = \ \frac1{6}
 \qquad\text{whereas}\qquad
\lim_{r\to0}\frac1{r^2}\Big(G^{{\mathbb R}}_{2,m}(0)-G^{{\mathbb R}}_{2,m}(r)\Big)\ = \ \frac1{{2}\sqrt{2}\, m}\fstop
\end{align*}

\begin{proof}
For convenience, we provide two proofs. Recall the eigenfunction representation  \eqref{eq:G0-CompactM} for the grounded Green kernel,
  \begin{align*}
\mathring G_{s,m}(x,y)=\sum_{j\in \N} \frac{\phi_j(x)\, \phi_j(y)}{(m^2+\lambda_j/2)^s} \comma \qquad x,y\in \mssM \fstop
\end{align*}
For the torus, we have  $\lambda_{2k-1}=\lambda_{2k}=(2\pi k)^2$ for $k\in\N$ with 
$\phi_{2k-1}(x)=\sqrt2\,\sin\big(2k\pi x\big)$, and $\phi_{2k}(x)=\sqrt2\,\cos\big(2k\pi x\big)$.
Choosing $m=0$, $y=0$, and $x=r$ thus yields
  \begin{align}\label{eq:FourierTorus}
\mathring G^{{\mathbb T}}_{s,0}(r)=\frac1{2^{s-1}}\,\sum_{k\in \N} \frac{1}{(\pi k)^{2s}} \cos\big(2k\pi r\big)\comma \qquad  r\in [0,1/2]\comma
\end{align}
and the conclusion follows by e.g.~\cite[1.443.1]{GraRyz07}.

An alternative proof of the claim can be obtained in the following way.
For $s=1$, the right-hand side of~\eqref{eq:FourierTorus} is indeed the Fourier series for the function given in \eqref{tor1}.
The values of $f_s \eqdef G^{{\mathbb T}}_{s,0}$ for all other $s\in\N$ can then be derived from there and from the facts that
\begin{align*}
f''_{s+1}=-2\, f_s, \qquad f'_s(1/2)=0, \qquad \int_0^{1/2}f_s(r)\diff r=0 \fstop
\end{align*}
The first claim follows from \eqref{distr-G0}.
Moreover,~\eqref{tor1} can be derived from \eqref{G-cosh}  by passing to the limit as $m\to0$:
  \begin{equation*}\mathring G^{{\mathbb T}}_{1,0}(x,y)=\lim_{m\to0}\bigg[G^{{\mathbb T}}_{1,m}(x,y)-\frac1{m^2}\bigg]\fstop \qedhere
  \end{equation*}
\end{proof}

\begin{figure}[htb!]
\centering
\begin{subfigure}[b]{.3\textwidth}
\centering
\includegraphics[scale=.35]{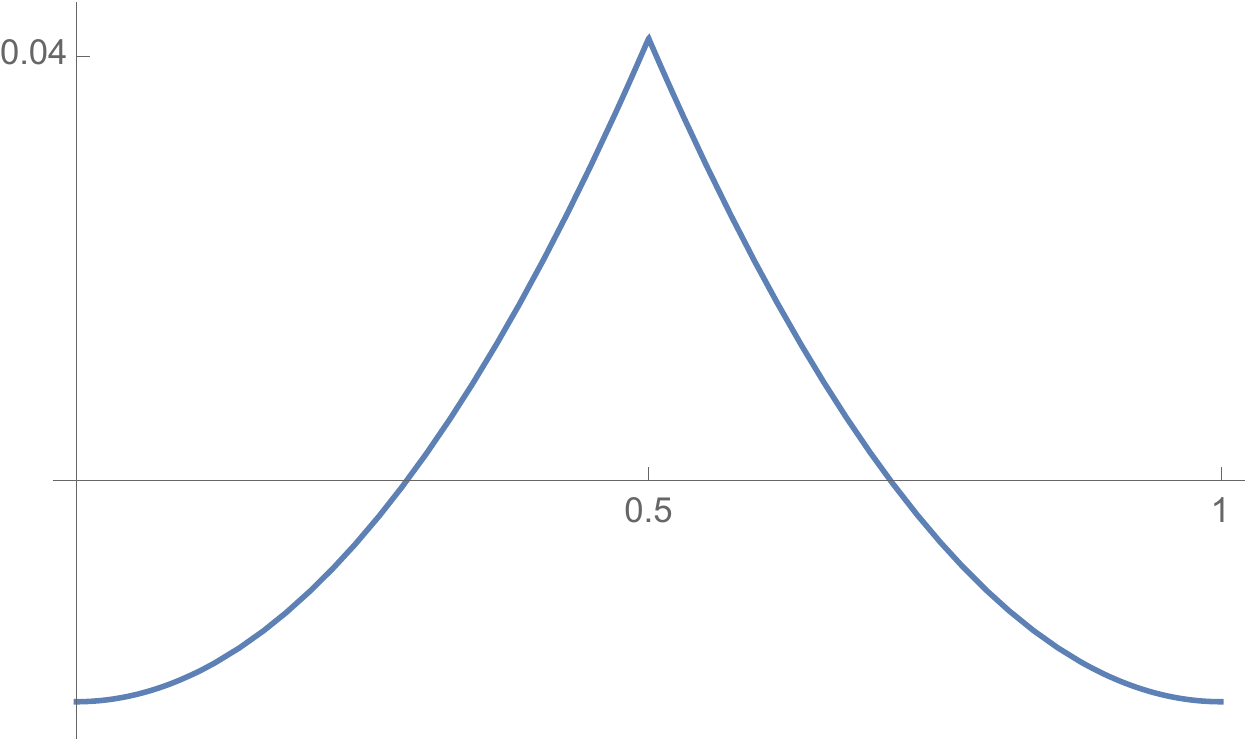}
\caption{$s=1$}
\end{subfigure}
\hfill
\begin{subfigure}[b]{.3\textwidth}
\centering
\includegraphics[scale=.35]{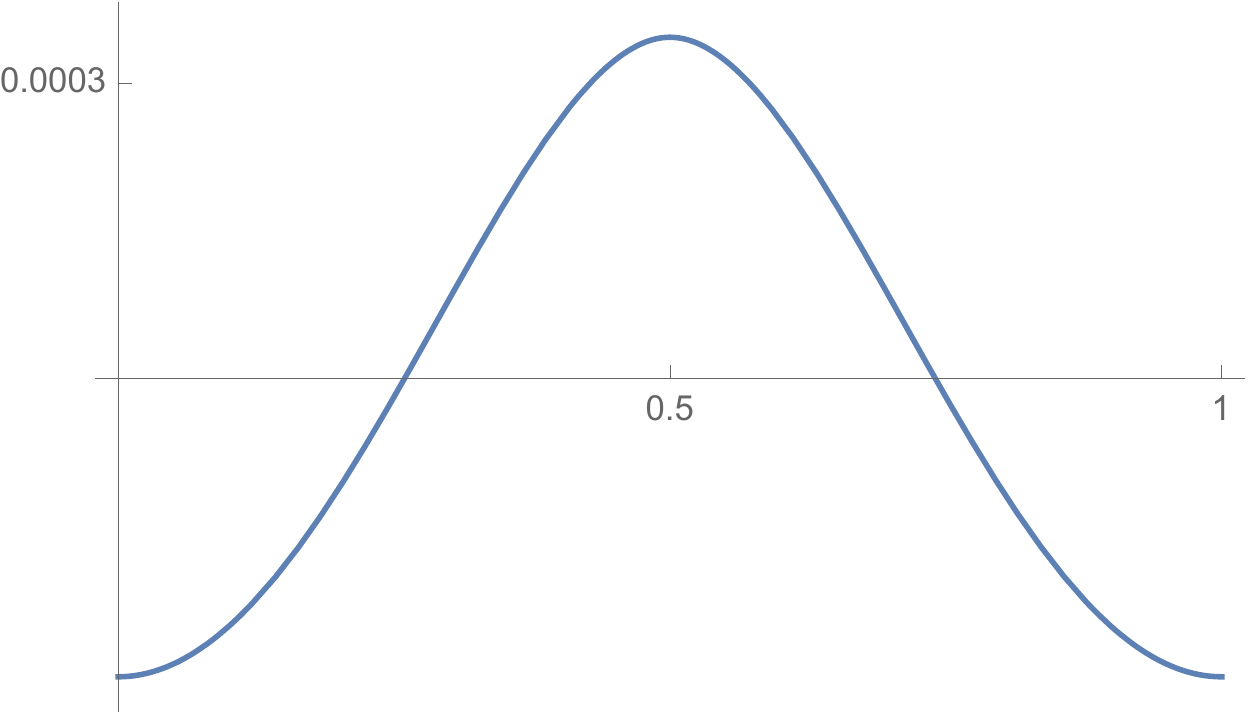}
\caption{$s=2$}
\end{subfigure}
\hfill
\begin{subfigure}[b]{.3\textwidth}
\centering
\includegraphics[scale=.35]{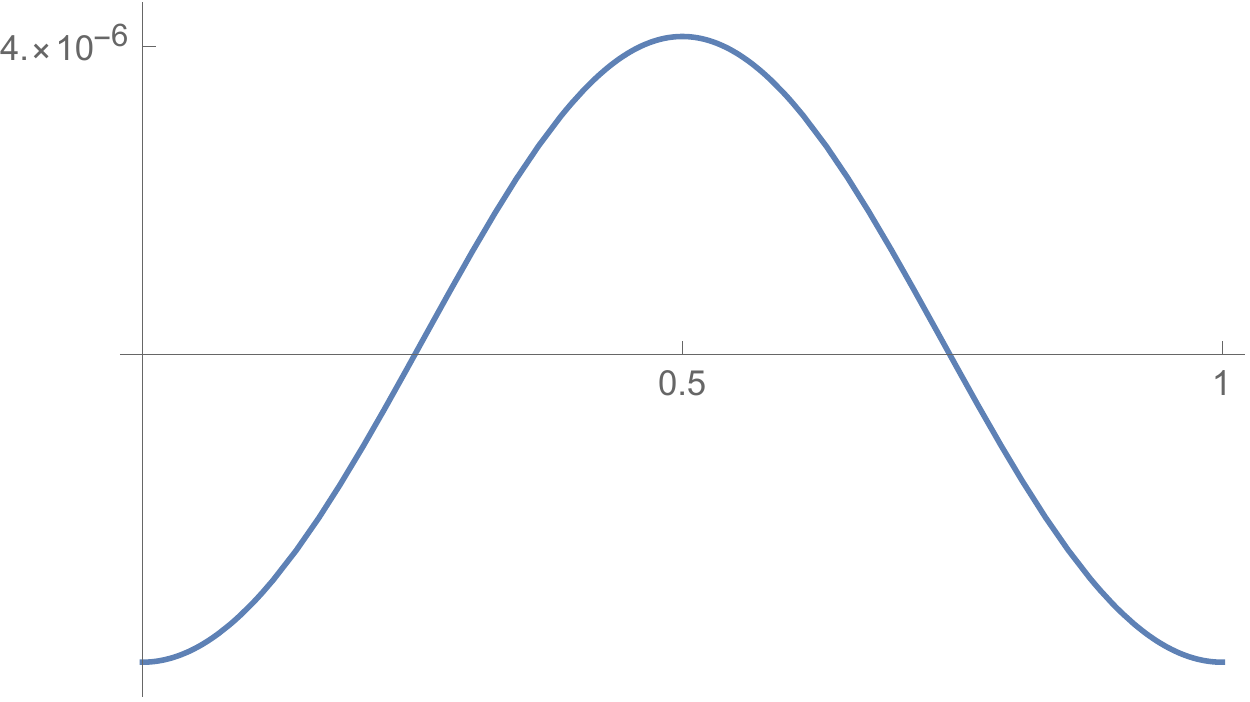}
\caption{$s=3$}
\end{subfigure}
\caption{The grounded Green kernel~$\mathring G^\mbbT_{s,0}(\tfrac{1}{2},y)$ with~$y\in [0,1)$ for~$s=1,2,3$.}
\end{figure}


\subsubsection{Hyperbolic Space}
For the hyperbolic space ${\mathbb H}^n$ of curvature $-1$, a closed expression for the Green kernels is available in dimension 3.

\begin{prop} For all $s,m,r>0$,
\begin{equation}\label{eq:GreenHypSp}
 G^{{\mathbb H}^3}_{s,m}(r)= \frac{r}{\sinh r}\,\frac1{ (2 \pi)^{3/2}\,\Gamma(s)}\,\int_0^\infty e^{-(m^2+1/2)t}\, e^{-r^2/(2t)}\, t^{s-5/2}\diff t\
 =\ \frac{r}{\sinh r}\cdot G^{{\mathbb R}^3}_{s,\sqrt{m^2+1/2}}(r)
 \end{equation}
 with $G^{{\mathbb R}^3}_{s,m}(r)$ denoting the Green kernel for $\R^3$ as discussed above.
 \end{prop}
 Thus, for instance, 
 $G^{{\mathbb H}^3}_{2,m}(r)=\frac{1}{2\pi\,\sqrt{2m^2+1}}\,
 \frac r{\sinh r}e^{-\sqrt{2m^2+1}\,  r}$.
 
 \begin{figure}[htb!]
 \centering
 \includegraphics[scale=.5]{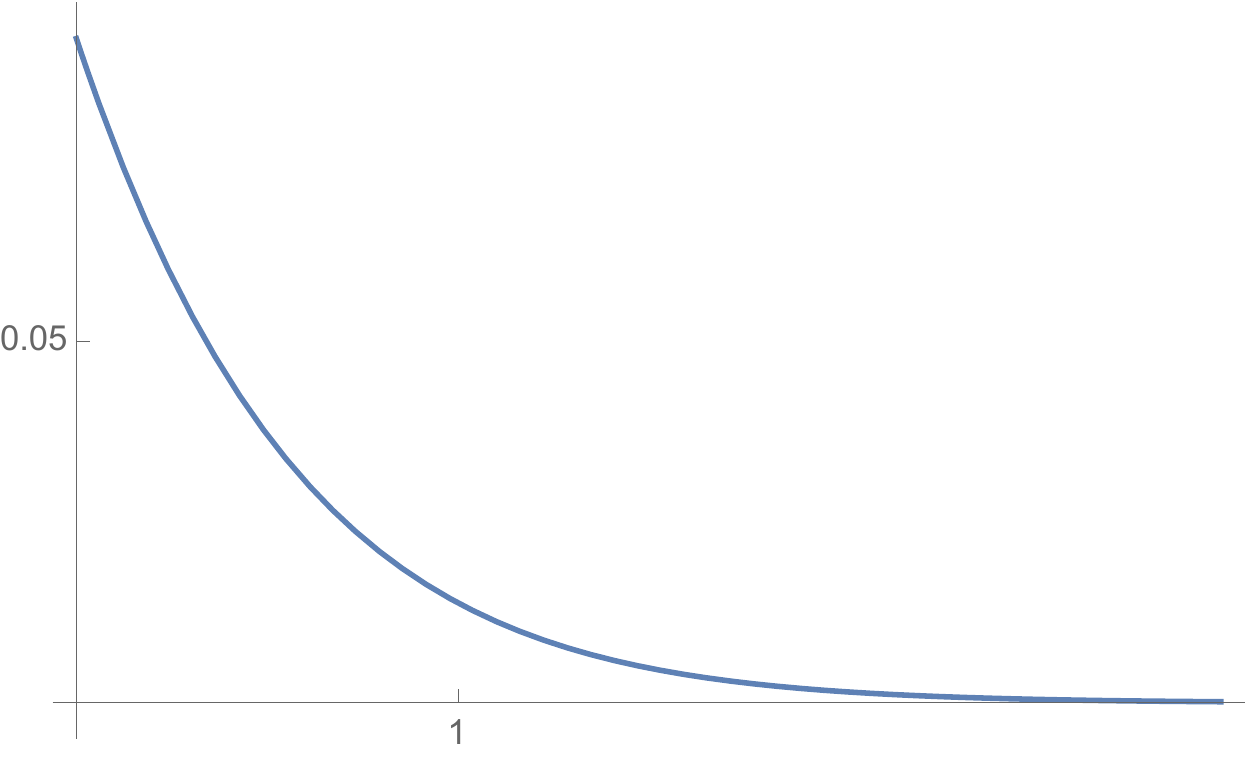}
 \caption{The Green kernel~$G_{2,1}^{\mbbH^3}$.}
 \end{figure}
 
 \begin{proof} The claim is an immediate consequence of the closed expression for the heat kernel on  ${\mathbb H}^3$ given e.g.\ in~\cite[Eqn.~(5.7.3)]{Dav89}.
 \end{proof}
 
 \begin{rem}
Integro-differential representations for~$G_{s,m}^{\mbbH^n}$,~$n\geq 4$, may be obtained in light of the analogous representations for the heat kernel~$p_t^{\mbbH^n}$ in~\cite{GriNog98}.
 \end{rem}

 \begin{cor} 
  The  Green kernel $G^{{\mathbb H}^3}_{s,m}$ on  ${\mathbb H}^3$ has asymptotic behavior close to the diagonal similar to $G^{{\R}^3}_{s,m}$. More precisely, if $C(s,m)$ denote the constants in the asymptotic formula~\eqref{eq:ApproxG2} for the Euclidean Green kernel, then 
  \begin{align}
\label{eq:ApproxG-hyp}
G^{{\mathbb H}^3}_{s,m}(0)-G^{{\mathbb H}^3}_{s,m}(r)\asymp& \begin{cases} 
C\big(s,\sqrt{m^2+1/2}\big) \cdot r^{2s-3} &\text{if } s\in (3/2,3/2+1) \comma
\\
C\big(s,\sqrt{m^2+1/2}\big)\cdot r^2 \log\frac1{r}& \text{if } s=3/2+1 \comma
\\
\paren{C\big(s,\sqrt{m^2+1/2}\big)+\frac16}\cdot r^2 &\text{if } s>3/2+1\fstop
\end{cases}
\end{align}
 \end{cor}
 
 \begin{proof}
It suffices to compute the Taylor expansion of~$G^{\mbbH^3}_{s,m}(r)$ in the form~\eqref{eq:GreenHypSp} around~$r=0$.
 For~$s\in (3/2,3/2+1]$ the Taylor expansion of~$r/\sinh(r)$ only provides terms of order~$O(r^2)$, which are smaller than the leading order of~$G^{\R^3}_{s,\sqrt{m^2+1/2}}(r)$ as~$r\to 0$ computed in~\eqref{eq:ApproxG2}.
 When~$s>3/2+1$, the same Taylor expansion provides a further additive factor~$r^2/6$.
 \end{proof}

\subsubsection{Sphere} For the unit  sphere we can derive explicit formulas for the grounded Green kernel of any order $s\in\N$ in any dimension, based on the observation \eqref{distr-G0}, the well-known representation of the radial Laplacian on spheres, and symmetry arguments. We present the results in some of the most important cases. 

\begin{thm}
For the sphere in 2 and 3 dimensions, 
\begin{align}\label{eq:GreenSphere1}
\mathring G_{1,0}^{\mathbb S^2}(r)=&-\tfrac{1}{2\pi}\paren{1+2\log\sin\tfrac{r}{2} }\comma
&
\mathring G_{1,0}^{\mbbS^3}(r)=&\tfrac{1}{2\pi^2}\paren{ -\tfrac{1}{2}+(\pi-r)\cdot\cot r }\comma
\\
\label{eq:GreenSphere2}
\mathring G_{2,0}^{\mathbb S^2}(r)=&\frac1{\pi}\int_0^{\sin^2(r/2)} \frac{\log t}{1-t}\,\diff t+\frac1\pi\comma
&
\mathring G_{2,0}^{\mathbb S^3}(r)=&\frac{(\pi-r)^2}{4\pi^2}+\frac1{8\pi^2}-\frac1{12}
\fstop
\end{align}
\end{thm}

\begin{figure}[htb!]
\centering
\begin{subfigure}[b]{.45\textwidth}
\centering
\includegraphics[scale=.5]{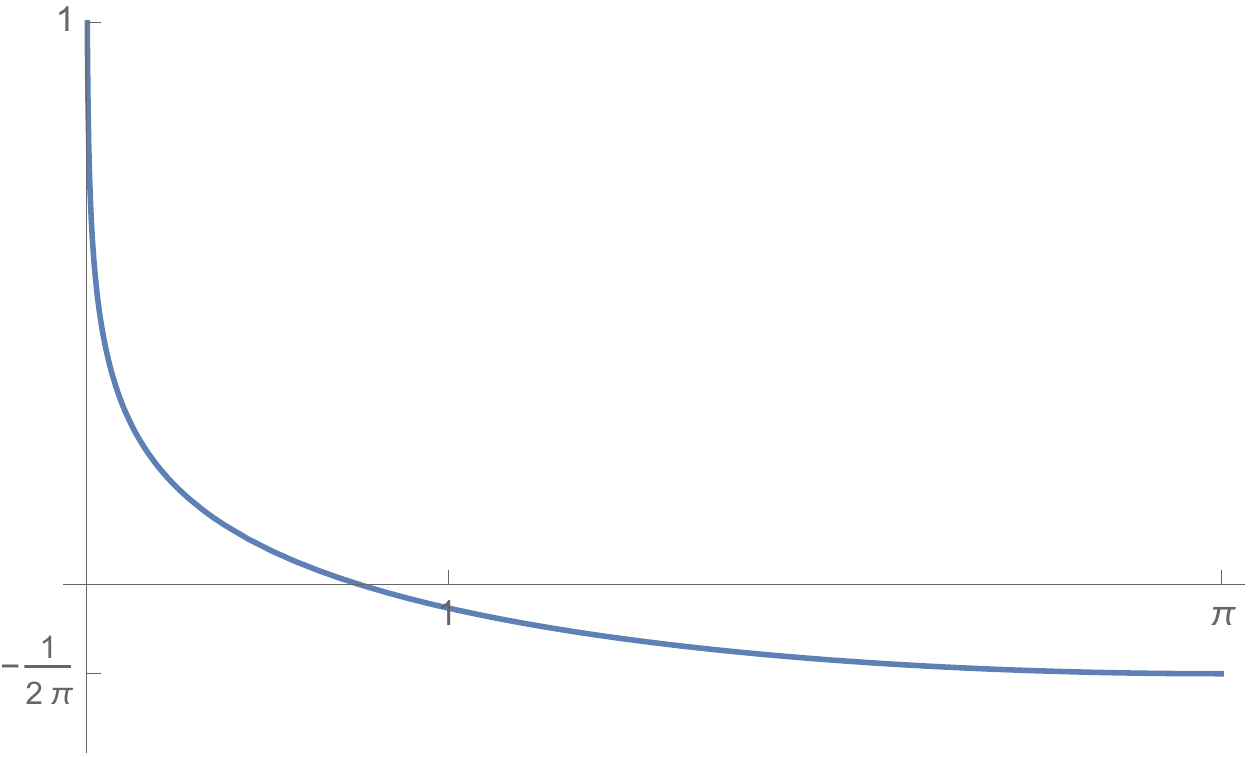}
\caption{$\mathring G^{\mbbS^2}_{1,0}$}
\end{subfigure}
\hfill
\begin{subfigure}[b]{.45\textwidth}
\centering
\includegraphics[scale=.5]{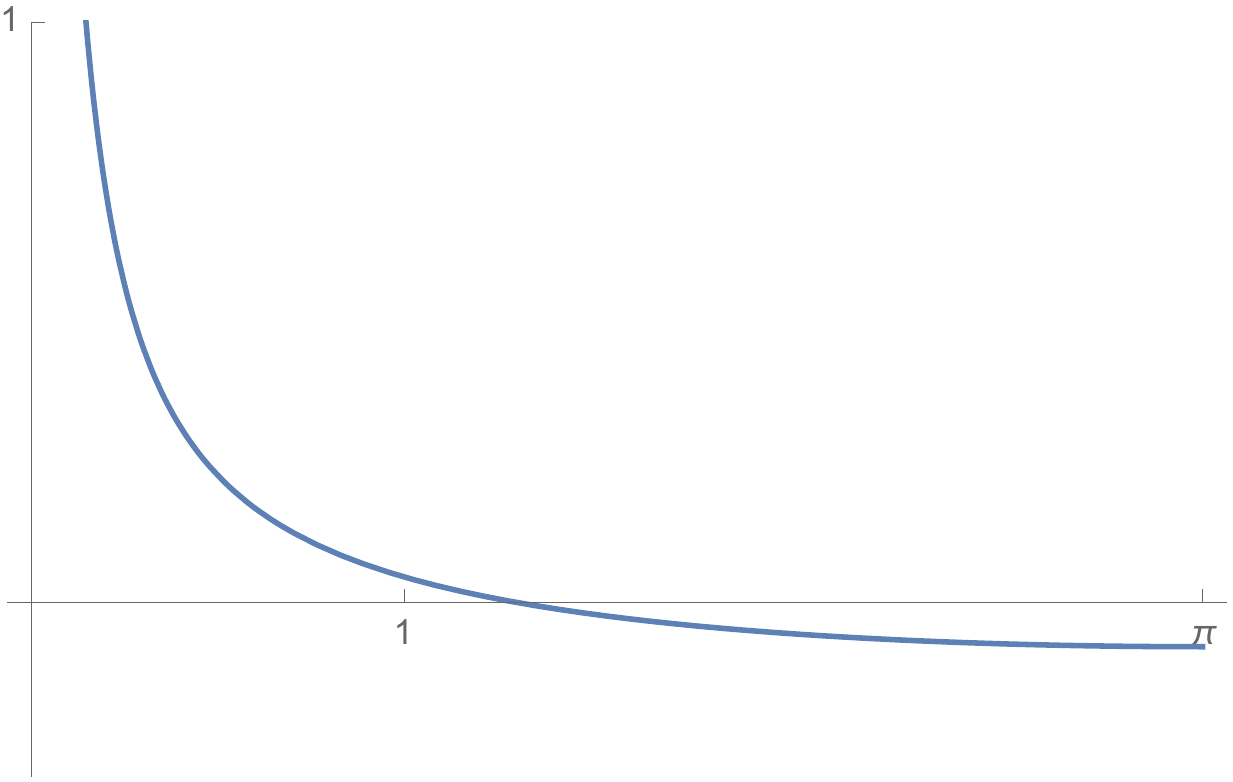}
\caption{$\mathring G^{\mbbS^3}_{1,0}$}
\end{subfigure}
\caption{The grounded Green kernels on~$\mbbS^n$ for~$s=1$ and~$n=2,3$.}
\end{figure}

\begin{figure}[htb!]
\centering
\begin{subfigure}[b]{.45\textwidth}
\centering
\includegraphics[scale=.5]{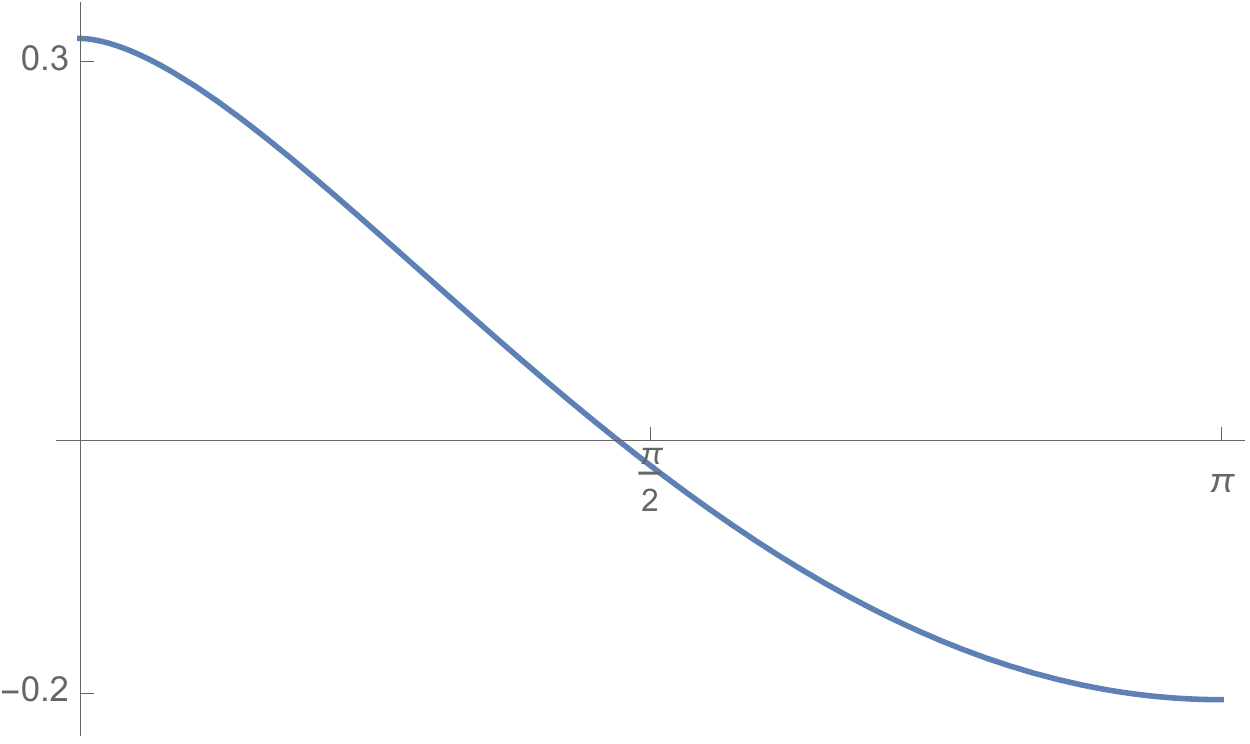}
\caption{$\mathring G^{\mbbS^2}_{2,0}$}
\end{subfigure}
\hfill
\begin{subfigure}[b]{.45\textwidth}
\centering
\includegraphics[scale=.5]{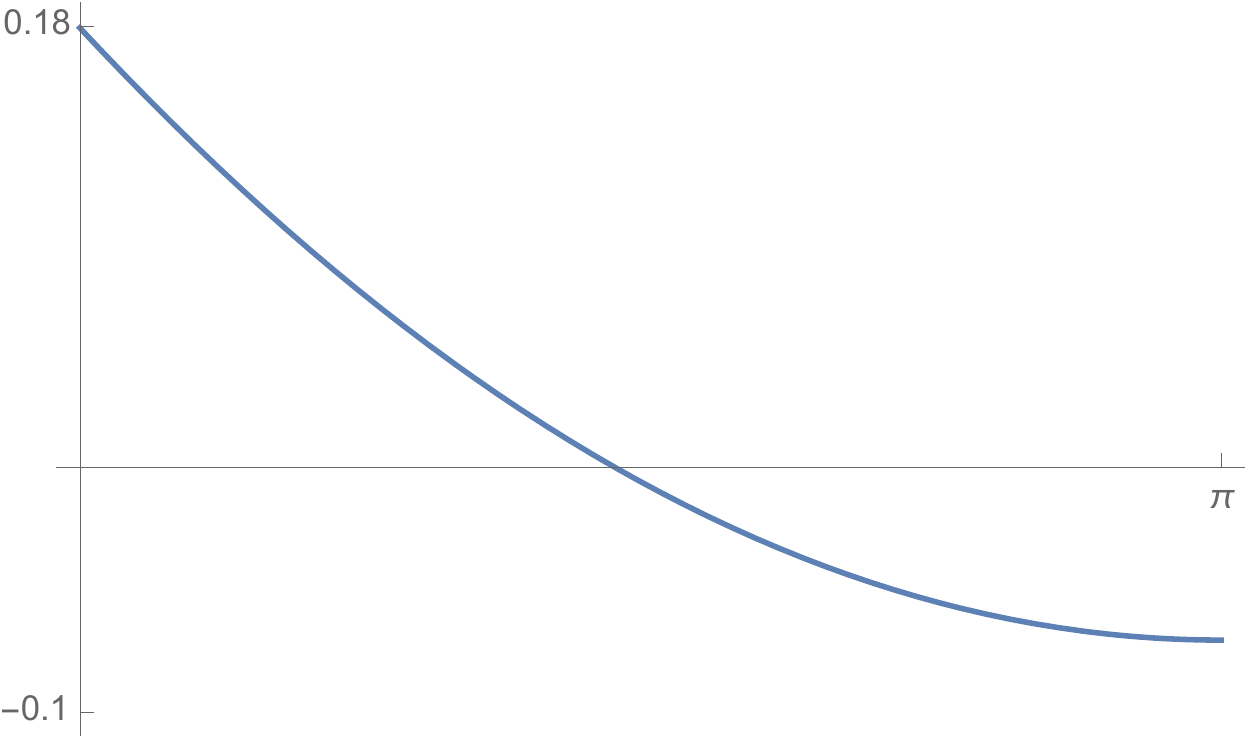}
\caption{$\mathring G^{\mbbS^3}_{2,0}$}
\end{subfigure}
\caption{The grounded Green kernels on~$\mbbS^n$ for~$s=2$ and~$n=2,3$.}
\end{figure}

Observe that for all $m>0$ as $r \to 0$,
\begin{align*} 
\mathring G^{{\mathbb S}^2}_{1,0}(r)\asymp
G^{{\mathbb R}^2}_{1,m}(r) \asymp- \frac1{\pi}\log r
\comma\qquad
\mathring G^{{\mathbb S}^3}_{1,0}(r)\asymp
G^{{\mathbb H}^3}_{1,m}(r) \asymp
G^{{\mathbb R}^3}_{1,m}(r)\asymp \frac1{2\pi\, r}
\comma
\end{align*}
and
\begin{align*} 
\mathring G^{{\mathbb S}^3}_{2,0}(r)
-\mathring G^{{\mathbb S}^3}_{2,0}(0)\ 
\asymp \
G^{{\mathbb H}^3}_{2,m}(r)
- G^{{\mathbb H}^3}_{2,m}(0)\ 
\asymp \
 G^{{\mathbb R}^3}_{2,m}(r)
- G^{{\mathbb R}^3}_{2,m}(0)\ 
\asymp \ -\frac1{2\pi}\, r
 \fstop\end{align*}

\begin{proof} Recall that for a radially symmetric function~$f(\emparg)=u\tparen{\mssd(x,\emparg)}$ on the $n$-sphere, the Laplacian and the volume integral are given by
\begin{align*}
\Delta f(y)=u''(r)+(n-1)\cot(r)\,u'(r)=\frac1{\sin^{n-1}(r)}\Big(\sin^{n-1}(r)\, u'(r)\Big)' \text{ with } r=\mssd(x,y)
\end{align*}
and $\int_{\mbbS^n} f\,\dvol=c_n\,\int_0^\pi u(r)\,\sin^{n-1}(r)\,\diff r$.
The representations in~\eqref{eq:GreenSphere1} thus follow from the fact that the functions  $u_2$ and $u_3$  given by the respective right-hand sides of~\eqref{eq:GreenSphere1}  are the unique solutions  on the interval $(0,\pi)$ to the second-order differential equation
\begin{align*}
u_n''(r)+(n-1)\cot(r)\,u_n'(r)=\frac2{\vol({\mathbb S^n})}\comma \qquad \lim_{r \to 0} r^{n-1} u_n'(\pi-r)=0\comma \qquad \int_0^\pi u_n(r) \sin^{n-1}(r)\diff r=0\comma
\end{align*}
which may be easily verified. Indeed, the function $u=u_2$ given above satisfies~$u'(r)=-\tfrac{1}{2\pi}\cot\tfrac{r}{2}$ and thus
\begin{align*}
(u'(r)\cdot \sin r)'=-\tfrac1{2\pi}\paren{1+\cos r}'=\tfrac{1}{2\pi}\sin r \comma
\end{align*}
hence $\Delta u=\frac1{2\pi}=\frac{2}{\vol({\mathbb S^2})}$.
Moreover, $\int_0^\pi u(r)\,\sin(r)\,\diff r=0$.

Similarly, $u=u_3$ satisfies $u'(r)=-\frac1{2\pi^2}\paren{\cot r+(\pi-r)\frac1{\sin^2 r} }$ and thus
\begin{equation*}
(u'(r)\cdot \sin^2r)'=-\tfrac1{2\pi^2} \paren{\cos r\,\sin r+\pi-r }'=\tfrac{1}{\pi^2}\sin^2r \comma
\end{equation*}
hence $\Delta u=\frac1{\pi^2}=\frac2{\vol({\mathbb S^3})}$. Moreover, $\int_0^\pi u(r)\,\sin^2(r)\,\diff r=0$.

The representations in~\eqref{eq:GreenSphere2}  follow from the fact that the functions $v_2$ and $v_3$  given by the respective right-hand sides of~\eqref{eq:GreenSphere2} are the unique solutions to
\begin{equation*}
v_n''(r)+(n-1)\cot(r)\,v_n'(r)=-2 u_n(r), \qquad \lim_{r\to 0}r^{n-1}v_n'(\pi-r)=0,\qquad \int_0^\pi v_n(r) \sin^{n-1}(r)\diff r=0
\end{equation*}
with $u_n=\mathring G_{1,0}^{\mathbb S^n}$ for $n=2, 3$ as specified above. To verify this, observe that 
$v_2$ satisfies 
$v_2'(r)\,\sin r=\frac2\pi \sin^2\frac r2\, \log \sin^2\frac r2$
and thus
$(v_2'(r)\,\sin r)'\,\frac1{\sin r}= -2u_2$. Moreover, 
\begin{equation*}
\int_0^\pi \Big(v_2(r)-\frac1\pi\Big)\sin(r)\,\diff r=\frac2\pi\int_0^1\int_0^t \frac{\log r}{1-r}\,\diff r\,\diff t=-\frac2\pi=-\frac1\pi\,\int_0^\pi\sin(r)\,\diff r \fstop
\end{equation*}

Similarly,  $v_3$ as defined above satisfies
\begin{equation*}
-\frac1{\sin^2 r}\Big(v'_3(r)\, \sin^2r\Big)'=\frac1{2\pi^2\, \sin^2 r}\Big((\pi-r)\, \sin^2r\Big)'=\frac1{2\pi^2}\,\Big(-1+2(\pi-r)\,\cot r\Big)=2 u_3 \fstop \qedhere
\end{equation*}
\end{proof}

\begin{rem}
The expression for~$\mathring G^{\mbbS^2}_{1,0}$ is in fact well known (see e.g.~\cite[Eqn.~(9)]{KimOka87}) and may equivalently be derived by means of complex geometry.
\end{rem}

{

}

\end{document}